\documentclass[envcountsect,numbook]{svjour3_arxiv4}   
\smartqed  
\usepackage{graphicx}
\graphicspath{{./Figures11/}}
\usepackage[utf8]{inputenc}
\usepackage{dsfont}
\usepackage{subfigure}

\usepackage{color}
\usepackage{url}
\usepackage{enumerate} 
\usepackage{booktabs}
\usepackage{tabularx}
\usepackage{ltablex}
\usepackage{mathptmx}
\usepackage{amssymb,amsmath,amsfonts,latexsym} 
\usepackage{longtable} 
\usepackage{textcomp}
\usepackage{lscape}

\usepackage{pdfpages}
\usepackage[]{hyperref}
\usepackage[12pt]{extsizes}
\usepackage{geometry}
\usepackage{ifthen}
\usepackage{mathtools}
\mathtoolsset{showonlyrefs}
\usepackage{amsmath}
\usepackage{amsfonts}
\usepackage{calrsfs}
\DeclareMathAlphabet{\pazocal}{OMS}{zplm}{m}{n}
\let\mathcal\pazocal
\usepackage{algorithm}
\usepackage{algpseudocode}
\usepackage[algo2e]{algorithm2e}

\allowdisplaybreaks

\let\remark\relax

\spnewtheorem{theorem}{Theorem}{\bfseries}{\itshape}
\spnewtheorem{corollary}[theorem]{Corollary}{\bfseries}{\itshape}
\spnewtheorem{lemma}[theorem]{Lemma}{\bfseries}{\itshape}
\spnewtheorem{proposition}[theorem]{Proposition}{\bfseries}{\itshape}
\spnewtheorem{definition}[theorem]{Definition}{\bfseries}{\itshape}
\spnewtheorem{remark}[theorem]{Remark}{\bfseries}{\upshape}
\spnewtheorem{assumption}[theorem]{Assumption}{\bfseries}{\itshape}
\counterwithin{theorem}{section}
\smartqed

\usepackage{tikz}
\def \R{\mathbb{R}}               
\def \1{{\bf 1}}                
\def \0{{\bf 0}}

\def \State{X}

\def \admiss{\mathcal{A}}

\newcommand{\Cov}{\operatorname{Cov}}

\usepackage[english]{babel}  

\renewcommand{\paragraph}[1]{{\smallskip\noindent\textbf{#1.}}}
\def \texteuro{EUR}

\def \texteuro{EUR}

\def \State{X}
\def \Statespace{\mathcal{\State}}

\def \admiss{\mathcal{A}}

\newcommand{\diff}{\mathop{}\!\mathrm{d}} 

\def \texteuro{EUR}
\def \controlset{U}

\def \seasonality{\mu_{R}}

\def \meanZ{m_{Z}(n,z)}
\def \meanX{m_{X}(n, x,a)}
\def \meanQ{m_{Q}(n,z, q,a)}
\def \meanG{m_{G}(n,z,g,a)}
\def \ConVarZ{\Sigma^2_{Z}(n)}
\def \ConVarX{\Sigma^2_{X}}
\def \ConVarQ{ \Sigma^2_{Q}(n,a)}
\def \ConVarG{ \Sigma^2_{G}(n,a)}
\def \CondevZ{\Sigma_{Z}(n)}

\def \CondevQ{ \Sigma_{Q}(n,a)}
\def \CondevG{ \Sigma_{G}(n,a)}
\def \CovZG{\Sigma_{Z G}(n,a)}
\def \CovZQ{\Sigma_{Z Q}(n,a)}

\def \CorrZQ{\rho_{Q}(n,a)}
\def \CorrZG{\rho_{G}(n,a)}

\def \Timehorizon{T}
\def \desresdemand{Z}
\def \SoC{Q}
\def \soC{q}
\def \Fuellevel{G}
\def \fuellevel{g}

\def \Concontrol{u}

\def \State{X}

\def \capacitybat{C_Q}
\def \capacityfuel{C_G}
\def \efficiency{\eta_E}

\def \stZ{\mathcal{Z}}
\def \stQ{\mathcal{Q}}
\def \stG{\mathcal{G}}
\def \discstZ{\widetilde{\mathcal{Z}}}
\def \discstQ{\widetilde{\mathcal{Q}}}
\def \discstG{\widetilde{\mathcal{G}}}

\def \consQ{\mathcal{U}_Q}
\def \consG{\mathcal{U}_G}

\def \charging{u^C}
\def \discharging{u^D}
\def \ecodischarge{u^{DL}}
\def \fueluse{u^F}
\def \ecofueluse{u^{FL}}
\def \Waiting{u^W}
\def \overspilling{u^O}
\def \meanRn{m_{R}(t_n)}

\def \muRn{\mu_{R,n}}

\def \termcostD{\phi^D_N}

\def \Statespace{\mathcal{\State}}

\def \admiss{\mathcal{A}}

\definecolor{mygreen}{rgb}{0,0.7,0}
\newcommand{\revi}{} 

\usepackage[normalem]{ulem}
\usepackage{soul}
\begin{document}
	
	\title{ Stochastic Optimal Control Problems for the Cost-Optimal Management of a Standalone Microgrid
	}		
	
	\titlerunning{Stochastic Optimal Control Problems of a Standalone Microgrid}        
	

    \author{Paul Honore Takam \and  Nathalie Fruiba}  
	\authorrunning{P. H. Takam, N. Fruiba} 
	
	\institute{
 Paul Honore Takam    \at
		Brandenburg University of Technology Cottbus-Senftenberg, Institute of Mathematics\\
		\emph{P.O. Box 101344, 03013 Cottbus, Germany} \\\email{\texttt{takam@b-tu.de} }  \\\\
        	 Nathalie Fruiba \at
             University of Luxembourg, Kirchberg Campus\\
             \emph{ 6 Rue Richard Coudenhove-Kalergi, 1359 Kirchberg Luxembourg}\\
             \email{\texttt{nathalie.fruiba@aims-cameroon.org}}
	}
	
	
	

	\maketitle
	
	\begin{abstract}
    In this paper, we consider a domestic standalone microgrid equipped with local renewable energy generation such as photovoltaic panels, consumption units, and battery storage to balance supply and demand and investigate the stochastic optimal control problem for its cost-optimal management. As a special feature, the manager does not have access to the power grid but has a local generator, making it possible to produce energy using fuel when needed. Such systems are very important for rural electrification, particularly in developing countries. However, these systems are very complex to control due to uncertainties in the weather and environmental conditions, which affect the energy generation and the energy demand. In addition, we assume that the battery and the fuel tank have limited capacities and that the fuel tank can only be filled once at the beginning of the planning period. This leads us to the so-called finite fuel problem. In addition, we allow the energy demand to not always be satisfied, and we impose penalties on unsatisfied demand, the so-called discomfort cost. The main goal is to minimize the expected aggregated cost of generating power using the generator and operating the system. This leads to a mathematical optimization problem. The problem is formulated as a discrete-time stochastic control problem and solved numerically using methods from the theory of Markov decision processes.

		%
	\end{abstract}	
  \newpage
 \setcounter{tocdepth}{3}
\tableofcontents   
 \newpage
		\section{Introduction}
%
%
		
   \textbf{ Motivation.}
   The increasing demand for reliable and sustainable energy solutions has led to a significant interest in the development and deployment of standalone microgrids. These systems, particularly those incorporating renewable energy sources like solar photovoltaics, offer a promising avenue for rural electrification and energy independence, especially in regions with limited or nonexistent grid infrastructure. It is essential to develop knowledge about the optimal management and expansion of microgrids to promote efficient, reliable, and sustainable energy distribution worldwide without dependence on the electricity grid. Unlike traditional grid-connected systems, these microgrids must balance intermittent renewable generation, fluctuating demand, and limited storage capacity. However, the inherent intermittency of renewable energy, coupled with fluctuating load demands, presents significant challenges to the efficient and cost-effective operation of these energy systems.\\
 In this work, we consider a microgrid, as shown in Fig.~\ref{modl}, equipped with photovoltaic panels to locally produce electricity to satisfy the building's demand. However, because of factors such as fluctuating temperatures, seasonal variations, and unpredictable weather conditions, the solar energy supply can often fall short of meeting the building's energy demand. To address this imbalance, the model incorporates a battery storage system that helps balance supply and demand. A special feature is that the microgrid is not connected to the grid. Instead, we have access to a generator, which makes it possible to generate electricity by consuming fuel when the battery cannot satisfy the demand.  This is essential to ensure a reliable power supply of electricity whenever required. 
\begin{center}
\begin{figure}[htbp!]
	\includegraphics[width=1\textwidth]{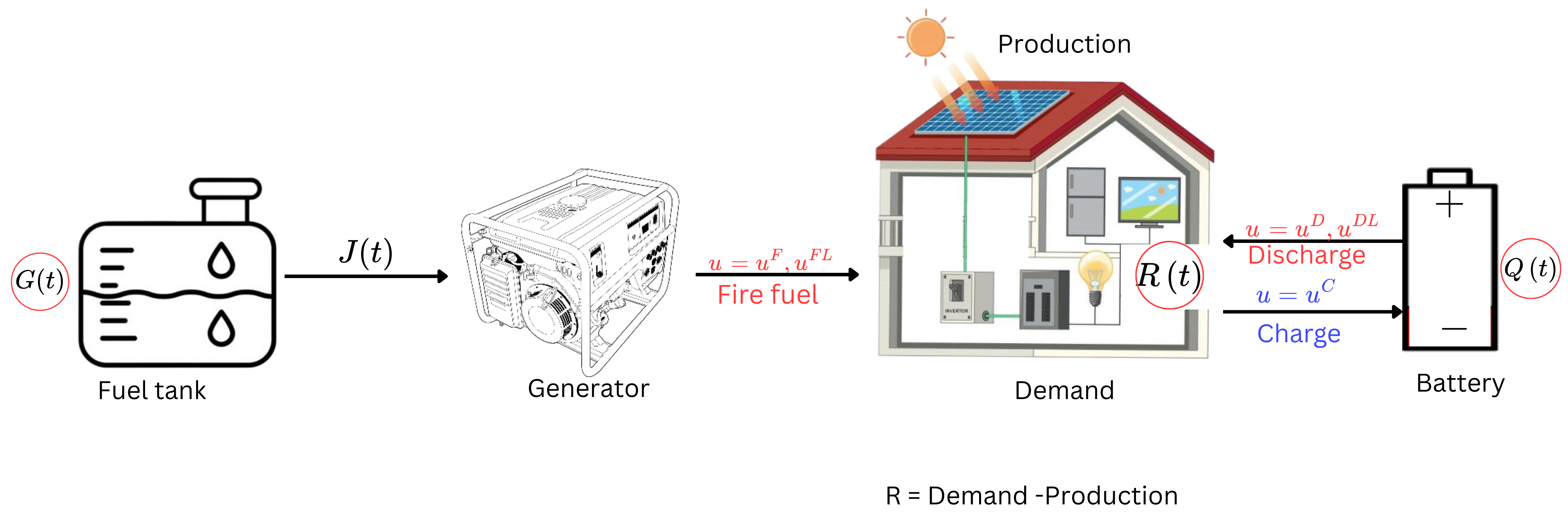}
	\caption{Simplified model of a microgrid that {\revi uses} solar panels, a battery and fuel driven generator to satisfy the electricity demand of a building. }
	\label{modl}
\end{figure}
\end{center}
The manager's decision consists of charging and discharging the battery and generating electricity using the generator by firing fuel. In addition, we include another feature, the power safe mode, in which demand can only be satisfied up to a certain threshold. Therefore, it must be accepted that demand may not always be fully met. We assign a penalty to the unsatisfied demand, the so-called discomfort cost. It is assumed that for a given finite planning horizon, the fuel tank cannot be refilled. This leads to a mathematical optimization problem, the finite-fuel problem.  Within a planning horizon, the aim of the manager is to minimize the expected aggregated cost, which includes the fuel cost for generating power
using the generator, the injection and withdrawal cost (degradation of the cost of the battery), and the discomfort cost for unsatisfied demand. Such a system is complex to control because, on the one hand, the renewable production and demand of energy are affected by uncertainties in the weather and environmental conditions, and on the other hand, the fuel tank cannot be refilled during the planning horizon. Therefore, the manager must make optimal decisions at all times in a context of uncertainty and take into account the constraints imposed by limited fuel and battery capacities.
To formulate the problem, we set up the mathematical model in continuous time for a multidimensional controlled process, whose dynamics is described by a system of random ordinary differential equations (ODEs) and stochastic differential equations (SDEs). However, closed-form solutions for the associated optimal control problems are generally not available, and traditional numerical methods suffer from the curse of dimensionality, especially for higher-dimensional state spaces. To overcome these challenges and facilitate practical implementation, we transform the state and performance criterion into discrete time and formulate the continuous-state Markov decision process (MDP) with a finite time horizon and a finite action space. The optimal strategy is found using discrete-time dynamic programming. This leads to the Bellman equation, which is solved numerically using a backward recursion.

\paragraph{Literature review on microgrids}
The foundation for understanding standalone microgrid management lies in the broader literature on microgrid systems. Shahgholian \cite{shahgholian2021brief} provided a review of operational frameworks and modeling aspects, crucial to understanding the behavior of these systems. Abbasi et al. \cite{abbasi2023review} extended the setting by discussing structural designs, communication systems, and advanced control methodologies that are applicable to both grid-connected and standalone microgrids. The works mentioned above highlight the ability of microgrids to operate autonomously, a key characteristic of standalone systems. Hatziargyriou \cite{ray2020microgrid}, Li and Zhou \cite{tabatabaei2019microgrid}, and Anvari-Moghaddam et al. \cite{anvari2021microgrids} presented a comprehensive review of microgrid systems, including control architectures and operational strategies relevant for standalone applications. The integration of battery energy storage systems is particularly critical in standalone microgrids to ensure a stable power supply, as explored by Zhao et al. \cite{zhao2013operation}. There, they emphasized the importance of considering battery life in the optimization process.\\
Optimal cost management is essential for the viability of stand-alone microgrids. Research in this area focuses on minimizing operating costs, particularly those related to fuel, maintenance, and battery degradation.  Zhao et al. \cite{zhao2013operation} investigated standalone microgrid optimization that directly addresses economic considerations, highlighting the need to take into account the characteristics of battery lifetime. Furthermore, \cite{cavus2024energy,devidas2022cost,hu2021model} highlighted the importance of improving reliability and efficiency in the face of variable environmental conditions to ensure cost-effectiveness. Advanced frameworks for cost-optimal energy management of microgrids have been proposed, along with hybrid optimization models designed to handle the complexities of cost functions. In Riou et al. \cite{riou2021multi}, the authors demonstrated that multi-objective optimization approaches that balance cost with reliability and environmental goals are particularly relevant for standalone microgrids.

\paragraph{Literature review on stochastic optimal control} 
  Due to the intermittent nature of renewable energy sources faced by standalone microgrids, stochastic optimal control is an essential tool to manage them cost-effectively. This approach enables optimal decision-making in situations of uncertainty. A considerable range of literature investigates techniques for solving stochastic optimal control problems, among which are Fleming and Soner \cite{fleming2006control}, Pham \cite{pham2009c}, and Øksendal and Sulem \cite{oksendal2013stochastic}. 
Kriett et al.  \cite{kriett2012optimal} and Khodabakhsh et al. \cite{khodabakhsh2016optimal} applied stochastic approaches to optimize energy storage operations by minimizing conditional value-at-risk. Liu et al.  \cite{liu2017adaptive} provided a broader perspective on adaptive dynamic programming applications in energy systems, highlighting the role of optimal control in managing microgrid dynamics. General methodologies for stochastic control in discrete-time systems are investigated in Bertsekas \cite{bertsekas2012dynamic}.  Takam and Wunderlich \cite{takam2025cost} demonstrated the importance of geothermal storage in the optimal management of residential heating systems while taking into account the uncertainty in renewable thermal energy production. Ganet Som\'{e} \cite{some2025stochastic} investigated the optimal control problem of prosumers in a district heating system, which also accounts for the uncertainty in the production of renewable thermal energy. Specific to electrical {\revi microgrids}, Qin et al. \cite{qin2018stochastic} proposed a stochastic control scheme to extend the battery's lifetime in standalone microgrids. Belloni et al. \cite{belloni2016stochastic} presented a stochastic solution for energy management in microgrids with renewable storage, emphasizing the interplay between uncertainty and optimization. In addition, Pacaud et al. \cite{pacaud2018stochastic} extended this analysis to domestic microgrids equipped with solar panels and batteries, showcasing practical implementations of stochastic control strategies. Model predictive control, often integrated with optimization algorithms, represents another significant approach to stochastic optimal control in microgrids, with the aim of minimizing electricity costs and optimizing battery usage \cite{hu2021model}. 

\paragraph{Literature review on Markov decision processes} Markov decision processes (MDP) provide a mathematical framework for a sequence of decision-making under uncertainty, making them well-suited to real-world challenges, particularly in the scientific field. For example, Puterman's book \cite{puterman2014markov} demonstrates the application of MDPs to the problem of engine maintenance and replacement. The work of B\"{a}uerle and Rieder \cite{bauerle2011markov} makes an important contribution to the field of MDP for finite, random, and infinite time horizons, as well as their application in the field of finance. Hu and Yue \cite{Hu} investigated the optimization of resource allocation using state-based decision models using MDP. McAuliffe \cite{mcauliffe2012markov} explored the applications of MDPs in financial systems, demonstrating their versatility in all domains. For microgrid applications, Jain et al. \cite{jain2018battery} investigated using MDP for battery optimization in a microgrid integrated with load and solar forecasting. In addition, Xiong et al. \cite{xiong2018microgrid} proposed a wavelet packet-fuzzy control policy within the MDP framework to improve microgrid power management, while Vergine et al. \cite{vergine2022markov} used Markov processes to effectively manage microgrid operations under uncertainty. 
Denardo \cite{denardo2012dynamic} provided detailed models and applications, emphasizing the computational advantages of dynamic programming in large-scale systems. Kennedy \cite{kennedy2012dynamic} explored applications in agriculture and natural resources, demonstrating its relevance to broader optimization problems. In microgrids, dynamic programming is employed to optimize system operations under varying conditions. For example, Bouman et al. \cite{bouman2018dynamic} used dynamic programming to address routing and scheduling problems, while Kappen \cite{kappen2011optimal} examines the role of the linear Bellman equation in optimal control theory.  The traditional method for solving the Bellman equation is to use a backward recursion. However, the backward recursion algorithm becomes computationally inefficient for higher-dimensional state space; therefore, the solution to the Bellman equation requires more efficient approximation techniques, such as optimal quantization (see Pages et al.~\cite{pages2004optimal}) and reinforcement learning methods, such as Q-learning (see Powell \cite{powell2007approximate}) and Sutton and Barto \cite{sutton2018reinforcement}.  For the application in energy storage, we refer to Pilling et al.\cite{pilling2024stochastic}, where the authors used the Q-learning method to solve the stochastic optimal control problem of an industrial heating system.

\paragraph{Our contribution} This paper investigates the cost-optimal management of a microgrid equipped with photovoltaic panels, a battery storage device, and a generator. {\revi While several existing studies have applied stochastic optimal control and MDP techniques to microgrid energy management \cite{jain2018battery,pacaud2018stochastic,qin2018stochastic,belloni2016stochastic}, none of them, to the best of our knowledge, jointly combine a finite-fuel constraint with state-dependent economical/limited operating modes for both the battery and the generator, nor do they provide a closed-form, discretization-error-free derivation of the running cost and of the transition operator within a single MDP framework. The present paper closes this gap.} In particular, 
\begin{itemize}
	\item We set up a continuous-time mathematical model of the energy system described above using appropriate differential equations for the dynamics of residual demand, the state of charge of the battery, and the fuel tank level. 
 \item We formulate mathematically a performance criterion that reflects the economic objectives of the system manager and that also takes into account the discomfort cost (cost of not fully satisfying demand) through appropriate quantification and expression in monetary units.
    
	\item We investigate the time-discretization of the state variables and the running cost, which leads to a continuous-state MDP.  
   The transition operator and kernel ensure that the discrete-time state dynamics closely preserve the distribution of the continuous-time state process at sampled time points. This enables the use of longer time intervals, reflecting the reality that decisions in control of energy systems can only be changed after certain periods have elapsed.
    \item  We use MDP theory techniques, such as backward recursion, to solve the Bellman equation for the control problem. To approximate the value function and optimal decision rule efficiently, we discretize the continuous-state MDP, transforming it into a finite-state Markov chain. 

     \item Finally, we calibrate some key parameters, perform numerical simulations, and discuss the results. Based on numerical experiments, we investigate the properties of the value function and the optimal control. This study is helpful for the practical implementation of such systems and also for future solution approaches for MDPs with a higher-dimensional state using approximate solution techniques such as quantization.

    \item {\revi As a key methodological novelty, unlike \cite{jain2018battery,pacaud2018stochastic,qin2018stochastic}, we combine, within a single finite-fuel MDP formulation, (i) a single, non-replenishable fuel tank over the planning horizon, and (ii) state-dependent economical/limited operating modes for both the battery and the generator, which allow the manager to partially satisfy strongly positive residual demand up to explicit thresholds $R_{Q0}$ and $R_{G0}$ rather than resorting to a purely binary on/off control structure.}

    \item {\revi In addition, we derive closed-form expressions for the one-step conditional expected running cost (Lemma~\ref{Conditional_C}) and for the mean and covariance structure of the joint Gaussian transition operator (Subsection~\ref{cond-distrib}), so that the discrete-time approximation of the continuous-time performance criterion introduces no additional discretization error beyond the state-space discretization itself, in contrast to existing MDP-based microgrid studies, which typically rely on purely numerical approximations (e.g., Monte Carlo simulation) of the transition kernel.}

    \item {\revi We compare the optimal decision rule with rule-based strategies (its expected cost over a week is $18\%$ to $80\%$ lower), and quantify the discretization error and the sensitivity to the discount rate.}
\end{itemize}	
\paragraph{Paper Organization} 
This paper is structured as follows: Section \ref{model} introduces a mathematical model of a microgrid, which includes a battery energy storage system and a generator. In particular, Subsection \ref{Control-S} introduces the control and state variables, and in Subsection \ref{Dyn-st}, we formulate the mathematical model of the energy system, detailing the dynamics of the deseasonalized residual demand, battery level, and fuel level. In addition, Subsection \ref{Stoc-OPtCont} presents a continuous-time performance criterion for cost-optimal management of microgrids, which takes into account fuel cost, battery degradation cost, and discomfort cost. Since the continuous-time problem cannot be solved directly, we transform the state and performance criterion into discrete-time and formulate a stochastic optimal control problem in Section \ref{SOCP}. This transformation enables the derivation of closed-form solutions for state variables given in Subsection \ref{time-discrete} and their marginal and joint distributions described in Subsection \ref{cond-distrib}. Subsequently, we present the state-dependent control constraints in Subsection \ref{state-depend}, formulate the Markov decision process in Subsection \ref{MDP}, and investigate it using dynamic programming. This leads to the so-called Bellman equation, which is solved using a backward recursion.  Section \ref{Num-rsult} presents the numerical results, where we discuss the properties of the value function and the optimal decision rule{\revi , compare it with rule-based strategies, and study the sensitivity to the discretization and to the discount rate}. In Section \ref{conclus}, we summarize the main findings and present some future directions of the current work. Finally, an appendix provides the nomenclature and calibration of key parameters used in the numerical results, as well as proofs of the lemmas, theorems, and propositions stated in the text.

\section{Mathematical Model of a Microgrid}
\label{model}
\subsection{Management of a  Microgrid}
\label{Control-S}
A microgrid is a small-scale electrical network consisting of loads, controllers, and dispersed energy resources. One of its main advantages is that it can operate in both grid-connected and stand-alone modes, allowing it to generate, distribute, and control the flow of electricity to nearby users \cite{mohammed2019ac}. Here, we consider an autonomous microgrid equipped with a local renewable energy generation unit, such as photovoltaic panels, a consumption unit, and a battery energy storage system to balance supply and demand, as well as a generator to produce energy using fuel when needed. The system considered is continuously managed in a time interval $[0,T]$, where $T>0$ is a finite time horizon.

 \paragraph{Residual demand} 
 The system is equipped with solar panels to produce electricity and several consumption units. The energy produced may not meet the demand for the building; this imbalance is represented by the residual demand, here denoted by $(R(t))_{t \in [0,T]}$, where $T > 0$ denotes a finite time horizon. This residual demand is simply the difference between the building's energy demand and the available solar power. Consequently, when the building's demand exceeds the solar power production, the residual demand is positive (\(R>0\)), and in cases of excess production, it is negative (\(R<0\)).
Given the inherent uncertainty in predicting future weather conditions and temperatures, which significantly affects both energy production and energy supply, we split into two components, $R(t)= \mu_R(t)+Z(t)$, where \(\mu_R\) denotes the seasonality function and \(Z\) is the deseasonalized residual demand modeled by a stochastic differential equation (SDE) detailed in Section \ref{Dyn-st}.

\paragraph{Battery state of charge} When there is overproduction (\(R < 0\)), excess energy is stored in the battery, provided that it is not already fully charged. In contrast, when the building's energy demand exceeds the energy produced (\(R > 0\)), the battery can be discharged to meet this unsatisfied demand.
We assume that the battery has a maximum level of $C_Q~ [kWh]$ and a minimum level of $0$.
We denote the relative battery level or battery state of charge (SoC) at time \(t \in [0,T ]\) by \(\SoC(t)\). The values of \(\SoC(t)\) are in the interval \([0,1]\), where \(\SoC(t) = 1\) indicates a fully charged battery and \(\SoC(t) = 0\) indicates an empty battery. To capture the continuous nature of the charging and discharging processes, the evolution of \(\SoC(t)\) of a battery over time will be modeled using an ordinary differential equation (ODE); see Equation~\eqref{dyn-Q}. 

\paragraph{Fuel tank level} In scenarios where the battery's state of charge is inadequate to fulfill the positive residual demand, additional electricity can be generated using fuel through a generator. We denote by $C_G ~[\ell]$ the maximum fuel tank capacity and by \(\Fuellevel(t)\) the relative fuel tank level at time \(t\), which also ranges from 0 to 1. Here, \(\Fuellevel(t) = 1\) represents a full fuel tank, while \(\Fuellevel(t) = 0\) indicates that the tank is empty. Similarly to the battery's SoC, the dynamics of the fuel tank level \(\Fuellevel(t)\) will be described by an ODE (see Equation~\eqref{dyn-G}), reflecting the continuous fuel consumption. We formulate the following important assumption.

\begin{assumption}\label{ass}
\phantom{x}
\begin{enumerate}
    \item The fuel tank is filled only once throughout the time horizon \([0,\Timehorizon]\), typically at the initial time, and thereafter no refilling is possible.

    \item The microgrid is equipped with a battery-saving mode that can be activated when the residual demand is above a certain threshold, $R_{Q0}$, and a generator with a fuel-saving mode that can be activated when the residual demand is above a certain threshold, $R_{G0}$.
    
    \item The battery and generator cannot operate at the same time; specifically, the battery cannot be charged or discharged while the generator is in operation.
\end{enumerate}
    \end{assumption}

    \begin{remark}\label{remark2}
\phantom{x}
\begin{enumerate}
\item The first assumption leads to the so-called finite fuel problem. In this case, the fuel tank is only refilled once, and the generator must operate optimally to avoid running out of fuel before the terminal time. We denote $F_0$ as the fixed {\revi initial} fuel price, which we pay at the initial time.

\item  The second assumption allows the manager to extend battery and fuel usage by satisfying a strongly positive residual demand only up to the thresholds $R_{Q0}>0$ (battery) and $R_{G0}>0$ (generator). {\revi The remaining unmet demand incurs a discomfort cost. This allows the manager to preserve resources for future use at the expense of only partially satisfying current demand.}
\item  {\revi The third assumption reflects the physical hardware layout of the low-power domestic microgrids considered in this paper. In such topologies, the battery inverter and the generator feed the same local AC bus; concurrent operation would require additional power-electronic synchronization and load-sharing control systems, which are typically omitted from low-cost, standalone rural installations. Mathematically, this mutual exclusivity structures the control set $\controlset$ into disjoint battery-only and generator-only action subsets. This structural separation allows the transition operator and the one-step conditional running cost to be derived in closed form (see Subsection~\ref{cond-distrib}, Lemma~\ref{Conditional_C}, and Remark~\ref{Rem-Tkernel}). Relaxing this constraint to permit concurrent operation (e.g., cycle-charging strategies) is a subject for future research discussed in Section~\ref{conclus}.}
\end{enumerate}
\end{remark}
\subsection{Control System}
This subsection describes the control and state processes that regulate the microgrid energy management system.

\paragraph{State process}  The state process at time \(t \in [0,\Timehorizon]\) is given by \(\State = (\desresdemand,\SoC,\Fuellevel)^\top\), taking values in \(\mathcal{X} \space \subset \mathbb{R}^3\). Here, $\desresdemand(t)$ [kW] is the deseasonalized residual demand, $\SoC(t)$ is the battery level that takes values in $[0,1]$, and $\Fuellevel(t)$ is the fuel tank level, which also takes values in $[0,1]$. {\revi For convenience, we use $x=(z,q,g)^\top$ to denote the value of the state $X(t)=(Z(t),Q(t),G(t))^\top$, and we write $r=\mu_R(t)+z$ for the corresponding value of $R(t)$ at a given time $t$.}

 \paragraph{Control process} We define controls as actions that determine how energy is distributed, stored, and generated within the microgrid. These actions are designed to optimize the use of available resources while meeting the stochastic energy demand of the building. \\
The control process primarily manages three key components: the solar panels, the battery, and the generator. When the electricity produced by solar panels exceeds the demand, resulting in a negative residual demand, excess energy can be stored in the battery provided that it is not fully charged. We denote by $\charging$ the control action of charging the battery. If the battery is full, the excess energy is dissipated. We materialize this situation by the control action of over-spilling denoted by $\overspilling$. When the residual demand exactly matches the production, we do nothing; that is, we have to wait. During this period, no action is taken: the battery is neither charged nor discharged, and the generator remains inactive. We denote this control action by $\Waiting$. Conversely, when the demand exceeds the production, resulting in a positive residual demand, we can either wait and pay a penalty or discharge the battery in full mode using the control input $\discharging$ or in a limited mode using the control action $\ecodischarge$, or generate electricity using the generator in full with control $\fueluse$ or in a limited mode using $\ecofueluse$. In summary, the control process denoted by $u$ takes  the following values:
\begin{align}
    \controlset= \{\overspilling, \charging,  \Waiting, \ecodischarge, \discharging,  \ecofueluse,\fueluse\}.
    \label{Controlset}
\end{align}
\subsection{Dynamics of the State Variables}
\label{Dyn-st}
This subsection outlines the mathematical models that govern the dynamics of state variables within our microgrid system over a continuous-time horizon \(t \in [0,\Timehorizon]\). We start with the exogenous state variable, that is, the residual demand, which is modeled by an SDE. Then, we derive the random ODEs that describe the evolution of energy in the battery and the evolution of the fuel tank level. 

\paragraph{Residual demand}
The uncertainties in the residual demand are modeled by a standard Wiener process $W$ on \([0,\Timehorizon]\), defined in a filtered probability space \((\Omega,\mathcal{F},\mathbb{F},\mathbb{P})\). Here, filtration \(\mathbb{F}\) is generated by the process \((W(t))_{t \in [0,T]}\), that is, \(\mathbb{F} = \mathbb{F}^W = (\mathcal{F}^W(t))_{t \in [0,T]}\) with $\sigma$-algebras $\mathcal{F}^W(t)=\sigma\{W(s), s\le t\}$, augmented by the $\mathbb{P}$-null sets, so that $\mathbb{F}$ satisfies the usual assumptions   of right-continuity and completeness.  
The residual demand is modeled as a mean-reverting stochastic process with seasonality as follows:
\[
R(t) = \seasonality(t) + \desresdemand(t),
\]
where $\desresdemand(t)$, measured in kW, represents the deseasonalized residual demand at time $t \in [0,T]$, capturing unpredictable deviations of the residual demand from the average trends while excluding seasonal variations.  This decomposition of $R$ allows us to focus on the deseasonalized residual demand $\desresdemand(t)$, which captures the uncertainty in $R$. We assume that $\desresdemand(t) \in \stZ \subset \mathbb{R}$, allowing it to take both positive and negative values. The dynamics of $\desresdemand(t)$ is described by an Ornstein-Uhlenbeck {\revi (OU)} process mean-reverting to zero of the form
\begin{equation}{\label{dyn-Z}}
    \diff\desresdemand(t) = - \beta_R \desresdemand(t) \diff t + \sigma_R\diff W(t),\quad \desresdemand(0) = z_0\in \stZ,
\end{equation}
where {\revi $\beta_R>0$ is the mean-reversion speed and $\sigma_R>0$ is the volatility parameter}. \\
The function  $\mu_R:[0,T]\longrightarrow \mathbb{R}$ is a bounded deterministic function that describes the seasonality pattern; see \cite{takam2025cost}. Typical examples include functions of the  form:
\begin{equation}\label{val-Mu}
    \seasonality(t) = \mu_0^R + \kappa_1^{R} \cos\left(\frac{2\pi(t-t_1^R)}{\delta_1}\right) +  \kappa_2^{R} \cos\left(\frac{2\pi(t-t_2^R)}{\delta_2}\right),
\end{equation} 
where $\mu_0^R > 0$ is a constant long-term mean, ${\revi \kappa}^R_i > 0, i=1,2$ are constants representing the amplitude of seasonality, $\delta_1 = 365$ days (annual seasonality) with the time shift parameter $t^R_1$, and $\delta_2 = 1$ day (daily seasonality) with time shift parameter $t^R_2$.\\
 The left panel of Fig. \ref{residual} shows the residual demand over and  the right panel a zoom in to 7 day-periods. These plots demonstrate that the residual demand fluctuates around a mean value. The right plot shows short-term variations, while the left plot reveals a similar pattern maintained over a longer duration, suggesting that energy consumption follows a relatively stable trend over long periods.
\begin{figure}[htbp!]
    \centering     \includegraphics[width=0.49\linewidth,height=0.3\linewidth]{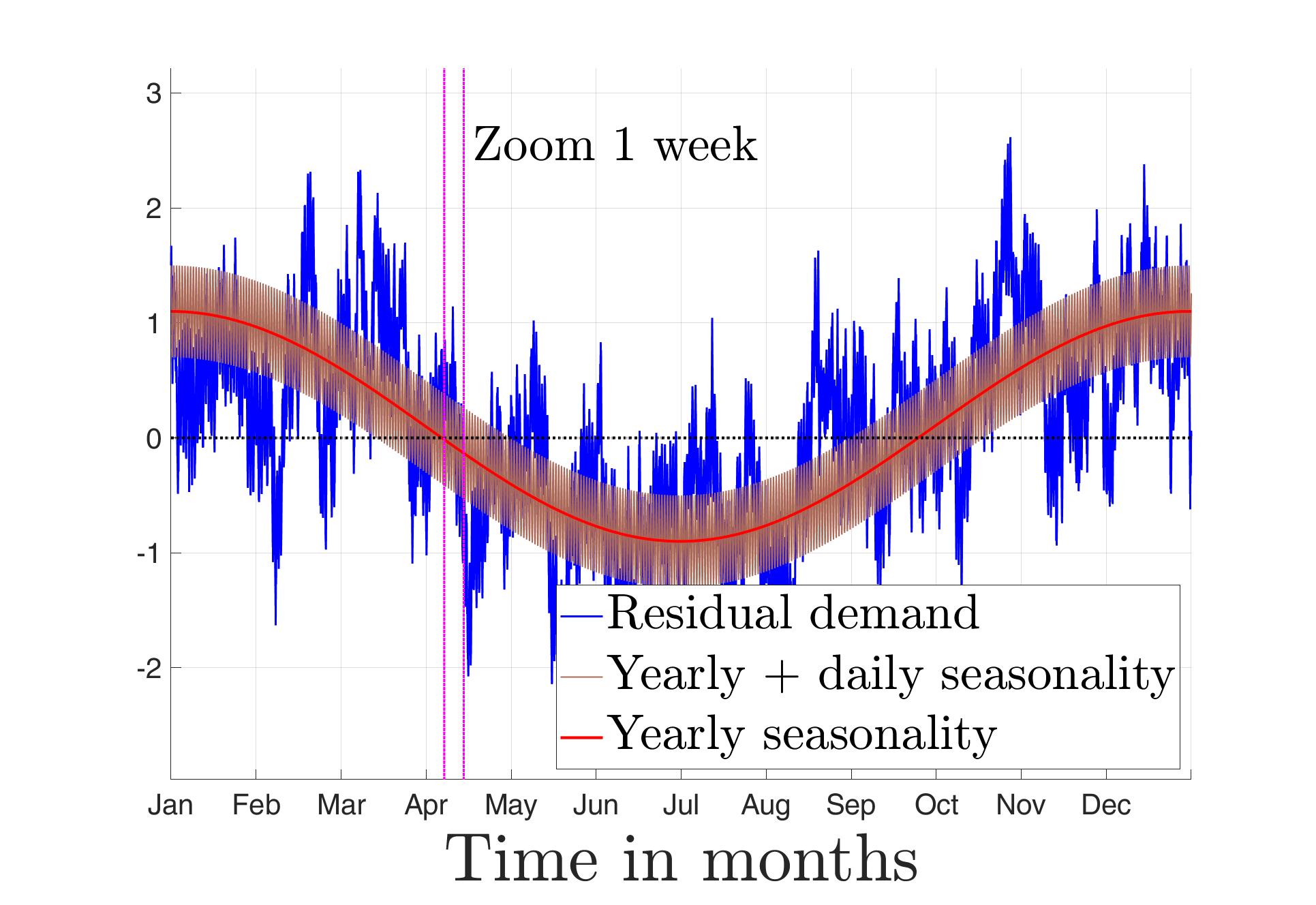}
     \includegraphics[width=0.49\linewidth,height=0.3\linewidth]{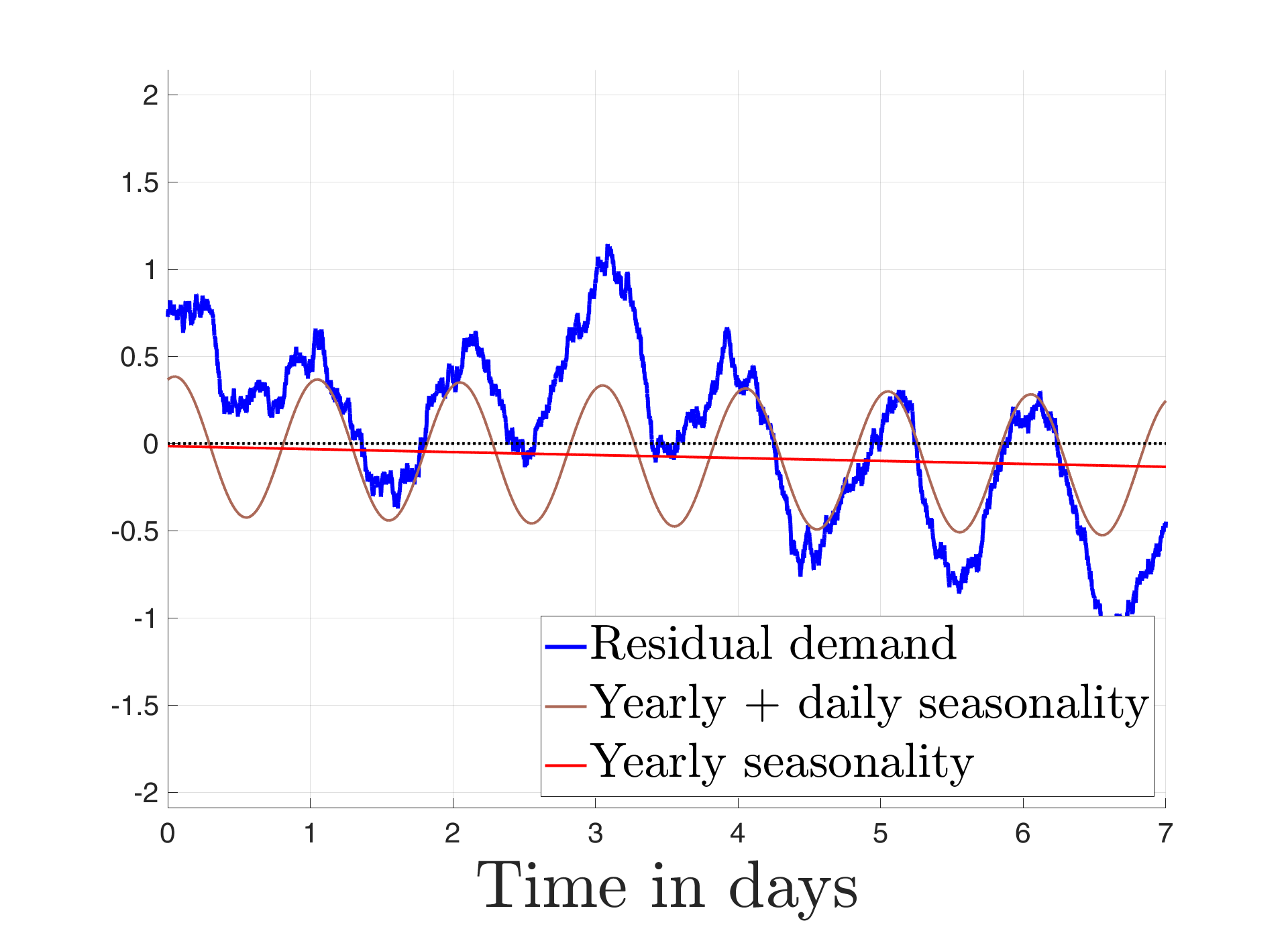}
    \caption{Residual demand with daily and yearly seasonal patterns over a period of one year (left) and a zoom into a week (right) with {\revi illustrative} parameters $\beta_R = 0.2, \sigma_R = 0.45, \kappa_1^R = 1,\kappa_2^R =0.4, \mu_0^R = 0.1, \delta_1= 365, \text{ and } \delta_2 = 1$. }
    \label{residual} 
\end{figure}

\paragraph{State of charge of the battery}
 The state of charge (SoC) of the battery evolves continuously in time in response to the charging or discharging processes constrained by its capacity and efficiency. During charging, the battery's level increases, while during discharging, it decreases. The two cycles are affected by the residual demand and the current state of the battery. Note that even in idle mode, the battery gradually loses energy; this is called self-discharge. In addition, we assume that the battery is equipped with a battery-saving mode that can be activated {\revi when $\mu_R(t)+z \geq R_{Q0}$, $t\in[0,T]$}. However, when the economical or limited mode is activated, the battery can only satisfy the positive residual demand up to a certain threshold $R_{Q0}$.
 {\revi We emphasize that the SoC $\SoC(t)$ is a continuous variable taking values in the whole interval $[0,1]$; the bounds $\SoC(t)=0$ and $\SoC(t)=1$ denote an empty and a fully charged battery, respectively, and not the only admissible levels. This continuous-state formulation is discretized only in the numerical implementation stage, using $N_Q$ grid points, for the purpose of solving the Bellman equation via backward recursion (see Subsection~\ref{State-discrete}).}
Therefore, the dynamics of $\SoC=\SoC^u$ can then be described by the following ODE
\begin{align}
\label{dyn-Q}
    \diff\SoC(t) &= \mathcal{H}(t,Z(t),\SoC(t),u(t))\diff t, \quad \SoC(0) = q_0\in [0,1],
\end{align}
where 
\begin{equation}\label{eq_H}
		\mathcal{H}(t,z,q,\nu) = \begin{cases}
			-\frac{1}{C_Q}(\mu_{R}(t)+z ) \efficiency(t,z,q) - \eta_0q ,&\nu \in \{\charging,\discharging\},\\
			-\frac{1}{C_Q}R_{Q0}  \efficiency(t,z,q) - \eta_0q ,&\nu= \ecodischarge,\\
			-\eta_0q, & \text{otherwise}.
		\end{cases}
	\end{equation}
    represents the energy injected into or withdrawn from the battery, which depends on the control actions. Here, \(C_Q\) is the maximum capacity of the battery and $\eta_0$ the self-discharge rate of the battery, which we assume to be constant. We denote the state-dependent battery's charging/discharging efficiency function by \(\efficiency(t,z,q)\), which is defined as
\begin{equation}
    \efficiency(t,z,q) = \begin{cases}
        \efficiency^C(q), &  \qquad z + \mu_{R}(t) \leq 0  \quad\text{(charging)},\\
        \frac{1}{\efficiency^D(q)}, &  \qquad z + \mu_{R}(t) > 0 \quad \text{(discharging)},
    \end{cases}
    \label{efficiency}
\end{equation} 
with $\eta_E^C$ and $\eta_E^D \in (0,1)$ the state-dependent charging and discharging efficiencies, respectively.
The state-dependent efficiencies take into account the fact that an empty storage is less efficient to discharge while being more efficient to charge; also, a full storage is less efficient at charging while being more efficient at discharging.  To achieve this, we consider functions of the form
\begin{equation}
\eta_E^{C}(q)=C_0^{C}+C_1^{C}q^{\ell_C}(1-q)^{m_C} \qquad \text{and} \qquad\eta_E^{D}(q)=C_0^{D}+C_1^{D}q^{\ell_D}(1-q)^{m_D},
	\label{effic}
\end{equation}
where $C_0^\dagger> 0, ~C_1^\dagger\geq0, ~\ell_\dagger,~m_\dagger\geq 1,~ \dagger\in\{C,D\}$ are constants. \\ $R_{Q0}$ represents the exact amount of residual demand that the battery is able to meet in the economical mode. 
 When $r(t)=\mu_R(t)+z \geq R_{Q0}$, the limited mode can be activated; in this case, the battery meets only a positive demand $R_{Q0}$, and the remaining unsatisfied demand $r(t) -R_{Q0}$ is penalized (see Section \ref{Stoc-OPtCont}).
\begin{figure}[htbp!]
	\centering
\includegraphics[width=.49\textwidth,height=0.3\linewidth]{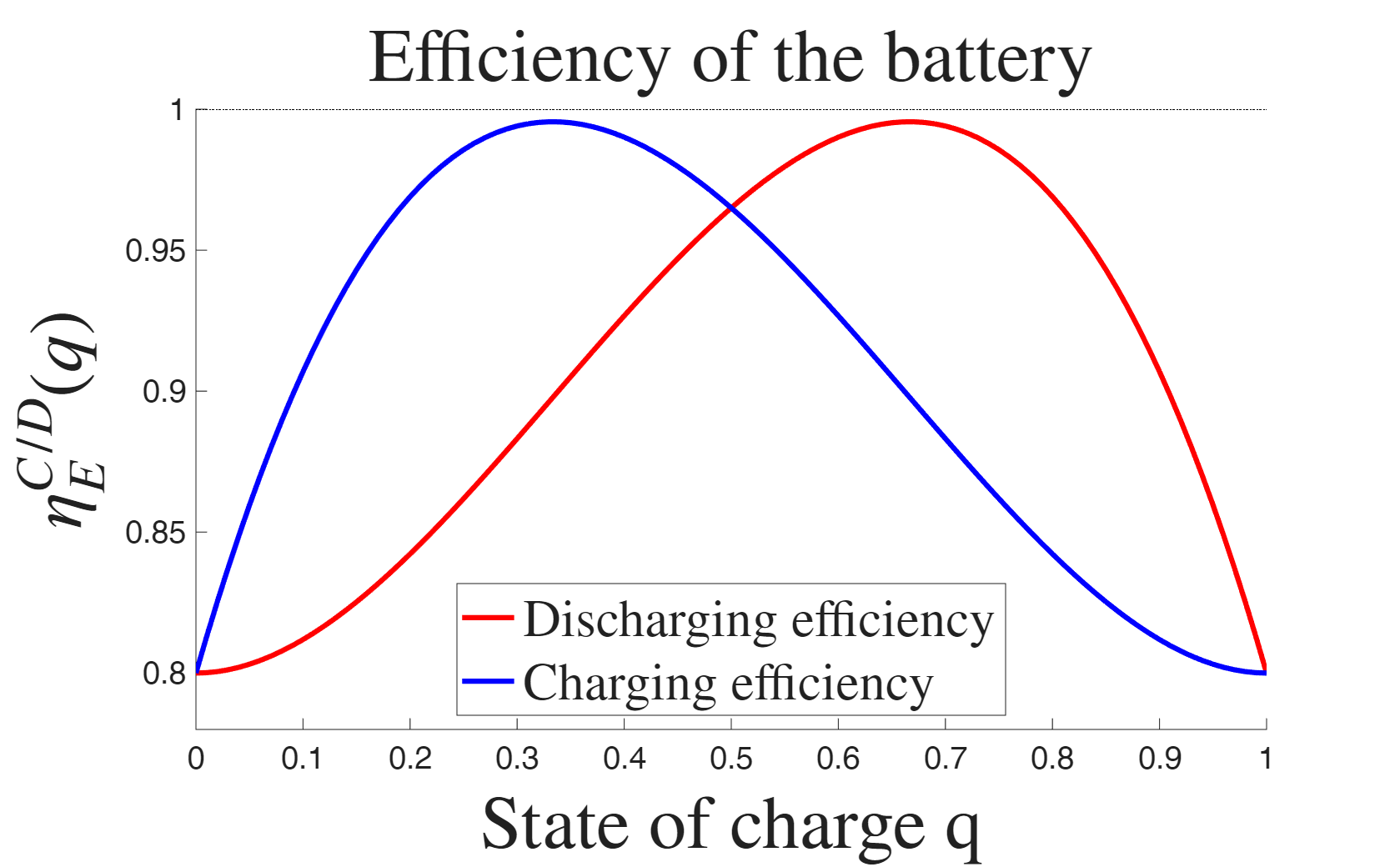}  \includegraphics[width=0.5\textwidth]{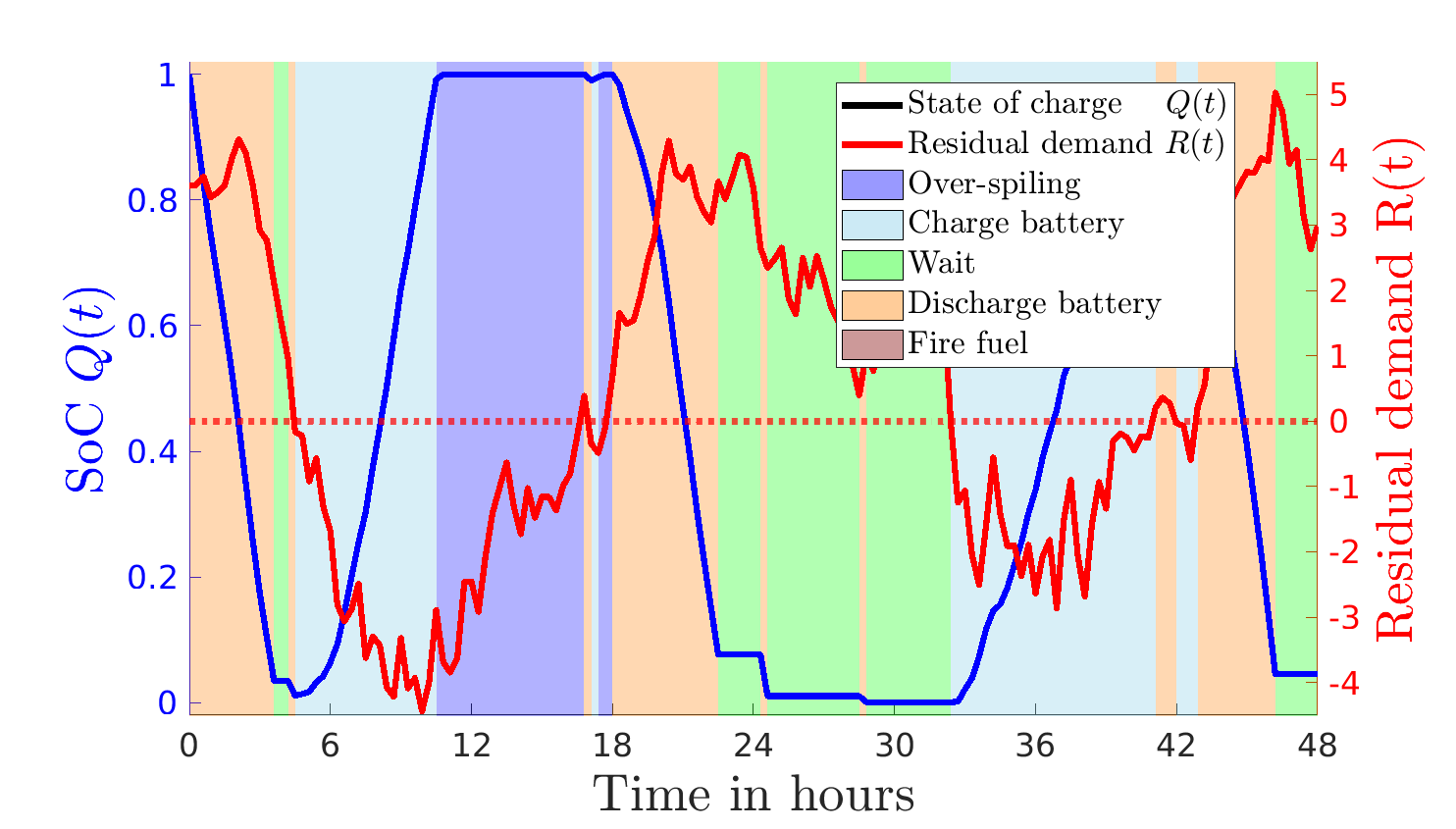}
    \caption{Left: Efficiency as a function of the state of charge $q$ with $\eta_E^C(q) = 0.8+1.32q(1-q)^2$ and $\eta_E^D(q)= 0.8+1.32q^2(1-q)$. Right: State of charge and residual demand as functions of time in response to control actions.}
	\label{eff}
\end{figure} 

 The left panel of Fig.~\ref{eff} indicates that the charging efficiency is high when the battery is empty but decreases as it approaches the maximum. Conversely, the discharging efficiency is low when the battery is empty but increases rapidly at higher battery levels. 
The right panel of Fig.~\ref{eff} illustrates the state of charge (SoC) of the battery (blue line) responding to fluctuating residual demand (red line) over a 48-hour period. In the first {\revi 5} hours, the battery \emph{discharges} due to positive residual demand. {\revi Thereafter}, the residual demand becomes significantly negative, causing the battery to \emph{charge} up to its full capacity. During periods of very high negative demand, the excess energy cannot be stored (\emph{overspill}). When the residual demand {\revi switches} back to positive, the battery begins to discharge. Once the battery SoC reaches a low level and the residual demand remains positive, the only operational state is to wait, as there is no more energy to discharge.

\remark \label{remark1}
 Note that this model does not account for overall battery degradation, assuming constant capacity $\capacitybat$ over the operating period ($t \in [0, \Timehorizon]$). However, to account for degradation, we penalize each energy injection into or withdrawal from the battery.

\paragraph{Fuel tank level} 
The fuel tank level changes depending on whether the generator is running or not. When active, the generator consumes fuel according to the associated load.  When the generator operates but does not produce power, it consumes some amount $c_0 \geq 0$ of fuel, where $c_0$ represents a non-negative constant. When satisfying a positive residual demand $R(t)$, that is, when the generator is in active mode, it consumes fuel according to load at a rate $C_L$, where $C_L$ is an increasing function, with $C_L(0) = 0$. In this work, we use the linear function of the form $C_L(x) = c_1 x$, where $c_1 > 0$, and $\capacityfuel$ denotes the fuel tank capacity. In addition, we assume that the generator is equipped with a saving mode, which is activated only when the residual demand $r(t)$ is positive with $r\geq R_{G0}$. When activated, it satisfies the exact amount of positive residual demand $R_{G0}$, and the remaining unsatisfied demand $r(t) -R_{G0}$ is penalized (see Section \ref{Stoc-OPtCont}).
Therefore, the dynamics of the fuel tank level $\Fuellevel=\Fuellevel^u$ is given by
\begin{equation}{\label{dyn-G}}
    d\Fuellevel(t) =-\frac{1}{\capacityfuel}\mathcal{J}(t,\desresdemand(t),\Concontrol(t)) dt, \qquad \Fuellevel(0)=G_0, 
\end{equation}
where
\begin{equation}
    \mathcal{J}(t,z,\nu) = \begin{cases}
       c_0 + c_1(\mu_{R}(t)+z) & \nu=\fueluse,\\
        c_0 + c_1 R_{G0} & \nu=\ecofueluse,\\
        0 &\text{otherwise}.
    \end{cases}
\end{equation}

\paragraph{State constraints}
Effective management of the microgrid system requires well-defined state and control constraints. These constraints ensure that battery and fuel levels remain within safe limits while guiding the operational decisions of the system manager.
As described in Remark~\ref{remark2}, both the battery level \(\SoC\) and the fuel level \(\Fuellevel\) must be constrained between \(0\) and \(1\). That is,
\begin{align}
    \State(t) \in \mathcal{K}=\Big\{(z,q,g),  ~~\soC\in [0, 1],\quad \fuellevel \in [0,1]\Big\}.
    \label{Constraint_state}
\end{align}
This ensures that neither level falls below zero nor exceeds one. This assumption on the limited capacity of the fuel tank and the battery level leads to the state-dependent control constraints described below in Subsection \ref{state-depend}.  For example, battery charging and discharging, as well as running the generator, are constrained by the state of the system. Specifically, charging is infeasible when the battery is full; discharging is infeasible when the battery is empty. However, the generator cannot be used when fuel is depleted.

\subsection{Continuous-Time Performance Criterion}
\label{Stoc-OPtCont} 
Now, we want to describe the costs that arise in the operation of the domestic microgrid and derive a performance criterion that evaluates the efficiency of the control strategy. This includes the operating cost and the terminal cost.

\paragraph{Running cost}
 Let \(x = (z,q,g)\), where \(z\), \(q\), and \(g\) are the deseasonalized residual demand, the battery level, and the fuel level, respectively.   The operating cost contains the \emph{fuel cost}, that is, the cost of using fuel to generate electricity to satisfy demand. It is represented by   \begin{equation*}
			\Psi_F(t,x,\nu)=\begin{cases}
				\left(c_0 + c_1(\mu_{R}(t)+z)\right)F_0 & \qquad \nu = u^F, \\
					\left(c_0 + c_1R_{G0}\right)F_0 & \qquad \nu = u^{FL}, \\
				0 & \qquad \text{ otherwise},
			\end{cases}
		\end{equation*}
		where \(F_0\) is a fixed fuel price and \(R_{G0}\) is the exact amount of residual demand that the generator meets in limited mode.  

\emph{Battery degradation cost}:
The lifetime of a battery is closely related to the number of charging and discharging cycles. At the end of the battery's life cycle, it must be replaced. The money needed to purchase a new battery at the end of its life span is distributed over the entire operating period. Practically, we impose a penalty for each injection into or withdrawal of energy from the battery. We call this a degradation cost, which represents the wear and tear associated with battery use.
The cost is proportional to the amount of energy injected into or withdrawn from the battery at each time, which is given by
\begin{equation*}
			\Psi_Q(t,x,\nu)=\begin{cases}
   \gamma_{deg}|\mu_{R}(t)+z| & \nu = u^{C},\\
				\gamma_{deg}(\mu_{R}(t)+z) & \nu=u^D,\\
					\gamma_{deg}R_{Q0} & \nu =u^{DL},\\
				0 & \text{otherwise}.
			\end{cases}
		\end{equation*}
 Here, $\gamma_{deg}$ is a nonnegative constant representing the degradation cost per unit of residual demand injected into or withdrawn from the battery and \(R_{Q0}\) is the exact amount of residual demand that the battery meets in limited mode. 
 The value of $\gamma_{deg}$ depends on many factors, including the battery capacity and the future price of the new battery.
        
\emph{Penalty for unsatisfied demand}: This cost captures the penalty incurred when demand exceeds supply \((Z(t)+\mu_{R}(t)>0)\) and the manager decides to do nothing \(u^W\), and when the economical mode of the battery or generator is activated. This occurs when $Z(t)+\mu_{R}(t) \geq R_{Q0}$ for the battery or $Z(t)+\mu_{R}(t) \geq R_{G0}$ for the generator. In this case, the battery or the generator meets the exact amount $R_{Q0}$ or $R_{G0}$ of residual demand and the remaining $\mu_{R}(t)+z-R_{Q0}$ or $\mu_{R}(t)+z-R_{G0}$ is penalized. This penalty represents the cost of not meeting the building's demand (discomfort cost). To emphasize the severity of the discomfort, we assume a quadratic penalty function of the form $f_P(x)=k_0 x^2$, where $k_0$ is a nonnegative constant.
This penalty, $\Psi_P(t,x,\nu)$, is given for by
 \begin{equation}
 \Psi_P(t,x,\nu) = \begin{cases}
  k_0(\mu_{R}(t)+z)^2 & \text{ if } \nu = u^W,\\
k_0(\mu_{R}(t)+z-R_{Q0})^2 & \text{ if } \nu = u^{DL},\\
 k_0(\mu_{R}(t)+z-R_{G0})^2 &\text{ if } \nu = u^{FL},
 \end{cases}
 \end{equation}
      
	In summary, the operating cost for $\nu \in \controlset$ and \(x=(z,q,g)\) is given by 
	\begin{align}
		\Psi(t,x,\nu) &= \Psi_F(t,x,\nu) + \Psi_Q(t,x,\nu) + \Psi_P(t,x,\nu). \nonumber
	\end{align}
    
\paragraph{Terminal cost} 
We consider a terminal cost depending on the state \(X\) at time \(T\) given by the function \(\phi(X(T))\).  Examples include \emph{zero cost} $\phi(X(T)) = 0$, that is, the battery and fuel tank expire worthless; \emph{penalty} if the state of charge (SoC) falls below the desired threshold (\(Q(T)< q_{\text{ref}}\)) or \emph{liquidation} of excess energy in the battery if the SoC exceeds the desired threshold (\(Q(T) \geq q_{\text{ref}}\)). We also consider the \emph{liquidation} of the remaining fuel in the tank at the terminal time if it is not empty (\( G(T)>0 \)). In this case, the profit generated from selling the leftover fuel in the tank is proportional to the fuel tank level at the terminal time. However, for the penalty or liquidation of the excess energy in the battery, charging and discharging efficiencies must be taken into  account.  Similarly to the fuel tank level, the profit generated from the sale of excess energy in the battery is proportional to the actual amount of energy above the reference threshold $q_{\text{ref}}$ sold to the market. This depends on the discharging efficiency, and the penalty is proportional to the actual amount of energy needed to reach the threshold $q_{\text{ref}}$, which depends on the charging efficiency. Note that the charging and discharging efficiencies are not constant, but state-dependent. Therefore, the actual surplus sold or the compensated battery deficit is given as the integral (of the inverse) of the efficiency function $\eta_E$ with respect to the SoC over the desired range. Summarizing, the terminal cost is given by

        \begin{align}
			\phi(X(T)) = \gamma^Q_{\text{pen}} \capacitybat\left(\int_{Q(T)}^{q_{ref}} \frac{dq}{\eta_E^C(q)}\right)^+-\gamma^Q_{\text{liq}}\capacitybat \left(\int_{q_{ref}}^{Q(T)} \eta_E^D(q)dq\right)^+ -\gamma^G_{\text{liq}}\capacityfuel G(T)
		 ,
		\end{align}
\noindent where \(\gamma^Q_{\text{pen}}\) is the penalty cost for the battery deficit, \(\gamma^Q_{\text{liq}} \text{ and }\gamma^G_{\text{liq}}\) are the liquidation prices for the battery surplus and the remaining fuel in the tank, respectively.\\

{\revi The performance criterion} is then defined as the expected aggregated discounted cost, given by
\begin{equation}
    \label{performance}
    J(t,x;u) =  \mathbb{E}_{t,x}\left[\int_{t}^{T} \mathrm{e}^{-\rho(s-t)}\Psi(s,X(s),u(s)) ds + \mathrm{e}^{-\rho(T-t)}\phi(X(T))\right].
\end{equation}

Here, \(\mathbb{E}_{t,x} = \mathbb{E}[.|X(t) = x]\) denotes the conditional expectation given that at time \(t\), the state \(X(t) = x = (z,q,g) \in \mathcal{X}\)   and \(\rho\) is the discount rate. The function \(\Psi\) represents the operating cost, while \(\phi\) denotes the terminal cost.

\section{Discrete-Time Stochastic Optimal Control Problem}
\label{SOCP} 
In this {\revi section}, we consider the time discretization of the state variables and the {\revi performance criterion}, which leads to a Markov decision process with a finite time horizon and finite action spaces. Further, we derive the associated dynamic programming equation and solve it using the backward recursion techniques.  More details on Markov decision processes can be found in the book by B{\"a}uerle and Rieder \cite{bauerle2011markov} and references therein.
We recall that the state space is denoted by \(\mathcal{X}\) and has dimension \(d = 3\). The planning horizon $[0,T]$ is subdivided into $N$ uniformly spaced sub-intervals of length $\Delta_N=T/N$, and the time grid points are defined by \(t_n = n\Delta_N\), where \(0 = t_0<  ~t_1  < \ldots <   ~t_N = T\). The state process \(X=(X_n)_{n=0,\cdots,N} \in \mathcal{X}\) is sampled at discrete times \(t_n, ~n=0,\ldots,N\).

\paragraph{Discrete-time control and state process} Here, we want to study the dynamics of the state process $X$ under the following assumption.
\begin{assumption}[Piecewise constant control] 
\label{Ass-co}
  We assume that  the control $u$ process is kept constant between two consecutive time points, that is, 
\[
u(s) =u(t_n)=:\alpha_n, ~ \text{ for }~s\in    [t_n,t_{n+1}), ~~n=0,\ldots,N-1.
\]  
\end{assumption}
 We will employ the shorthand notation \(R_n = \mu_{R}(t_n) + Z_n\) that denotes the residual demand \(R\) at time \(t_n\). The control process is defined as \(\alpha = (\alpha_0, \ldots, \alpha_{N-1})\), where \(\alpha_n = \tilde{\alpha}(n, X_n^\alpha)\) for \(n = 0,~\ldots,~N-1\). The mapping \(\tilde{\alpha} : \{0,  \ldots, N-1\} \times \mathcal{X} \longrightarrow \controlset\) represents the action taken in state \(x\) at time \(n\), where ~\(\controlset\) is given by \ref{Controlset}. \\
Next, we will derive the discrete-time approximation of the dynamics of the state variables over the time interval \([t_n, t_{n+1})\). We start with the dynamics of the continuous-time state process given by the SDE \eqref{dyn-Z} and ODEs \eqref{dyn-Q} and \eqref{dyn-G}. Since the SDE is linear, we can benefit from the availability of closed-form solutions that allow us to minimize the errors in the derivation of the discrete-time state dynamics. For the dynamics of the battery level that contains a state-dependent efficiency function, we use a combination of Euler and closed-form solution methods (semi-explicit discretization).

\subsection{Time-Discretization}
\label{time-discrete}
Let \(X=(X_n)_{n=0,\cdots,N} \in \mathcal{X}\) denote the state process. We make the following assumption on the parameters. 
\begin{assumption}[Piecewise constant model parameters]
\label{Ass-par}
 The time-varying seasonality $\mu_R$, the charging and discharging efficiencies are constant between two consecutive time points,
that is   
\[\mu_{R}(s) = \mu_{R}(t_n)=:\mu_{R,n} \text{ and}~~\eta_E(s, Z_n, Q_n)=\eta_E(t_n, Z_n, Q_n) = \eta_{E,n} \quad for~~ s\in[t_n,t_{n+1}).\]
\end{assumption}

\paragraph{Time-discretization  of the deseasonalized residual demand} 
Recall that the continuous-time dynamics of the deseasonalized residual demand is described by the SDE \eqref{dyn-Z}.
Let \(Z(t_{n}) = z\)  be the sampled value of \(Z(t)\) at time \(t_n\).
\begin{lemma}\label{lem-Z}
    The closed-form solution of the SDE \eqref{dyn-Z} at time \(t_{n+1}\) with initial condition \(Z(t_n) = z\) is given by:
    \begin{equation}\label{Rec-R}
        Z(t_{n+1}) = z \mathrm{e}^{-\beta_R\Delta_N} + \sigma_R \int_{t_n}^{t_{n+1}} \mathrm{e}^{-\beta_R(t_{n+1}-s)} dW(s).
    \end{equation}
\end{lemma}
\begin{proof}
    See Appendix \ref{proof-lemZ}.
\end{proof}

 \paragraph{Discrete-time approximation of the battery level} The following lemma describes the discrete-time evolution of the battery level based on its continuous-time dynamics given in \eqref{dyn-Q} and the shorthand notation of the mean of $m_{R,n} = \meanRn$. 

\begin{lemma}\label{lem-Q}
Under Assumptions \ref{Ass-co} and \ref{Ass-par}, the closed-form solution to the dynamics  of the battery level with known $ q=Q(t_n)$ and $z=Z(t_n)$ is given based on the control  $a \in U$ as follows:
 \begin{equation}\label{Rec-Q1}
	Q_{n+1} = q \mathrm{e}^{-\eta_0\Delta_N}- \frac{ \eta_E^n}{\capacitybat}\begin{cases}
		\frac{z}{\eta_0-\beta_R}\left(\mathrm{e}^{-\beta_R\Delta_N}-\mathrm{e}^{-\eta_0\Delta_N}\right)+ \frac{ \mu_{R,n}}{\eta_0 }(1-\mathrm{e}^{-\eta_0\Delta_N})+ \Upsilon_Q  &  \quad a \in \{u^C, u^D\},\\[1ex]
	 \frac{R_{Q0} }{\eta_0 }(1-\mathrm{e}^{-\eta_0\Delta_N})&  \quad a=u^{DL},\\
		0 & \quad \text{otherwise},
	\end{cases}
\end{equation}
where
\begin{align}
   \Upsilon_Q&= \sigma_{R}\displaystyle\int_{t_{n}}^{t_{n+1}}\mathrm{e}^{-\eta_0(t_{n+1}-s)}\left(\displaystyle\int_{t_n}^{s}\mathrm{e}^{-\beta_R(s-u)}dW(u)\right)ds. 
   	\label{upsilon_1}
\end{align}
\end{lemma}

\begin{proof}
	    See Appendix \ref{proof-lemQ}
\end{proof}

   \paragraph{Discrete-time dynamics of the fuel level}
   The continuous-time dynamics of the fuel level, denoted $G(t)$, is governed by the ordinary differential equation (ODE) given in \eqref{dyn-G}. The following lemma presents the discrete-time approximation over a time interval $ [t_n, t_{n+1})$.

\begin{lemma}\label{lem-G} Under Assumption \ref{Ass-par}, the discrete-time approximation  of the ODE \eqref{dyn-G} at time $t_{n+1}$ knowing $g=G(t_n) $ and $z=Z(t_n) $is given as follows:
	  \begin{align}
		G_{n+1}= g - \frac{c_0}{C_G} \Delta_N -\frac{c_1}{C_G}\begin{cases}
			\frac{z}{\beta_R}(1-\mathrm{e}^{-\beta_R\Delta_N})+ \Upsilon_G & \quad a = u^F\\[1ex]
			R_{G0}\Delta_N & \quad a= u^{FL}\\
			0 & \quad \text{otherwise},
		\end{cases}
		\label{G_1}
	\end{align}
where \begin{align*}
	 \Upsilon_G&=\sigma_R\int_{t_n}^{t_{n+1}}\int_{t_n}^{s}\mathrm{e}^{-\beta_R(s-u)}dW(u)ds. \label{Ups-G}
\end{align*}
\end{lemma}
\begin{proof}
    See Appendix \ref{proof-lemg}.
\end{proof}

\noindent It is important to note that closed-form solutions for \(Z_{n+1}:=Z(t_{n+1})\), \(Q_{n+1}:=Q(t_{n+1})\), and \(G_{n+1}:=G(t_{n+1})\)  given in Lemma \ref{lem-Z}, \ref{lem-Q}, and \ref{lem-G}, respectively, involve integrals with respect to a Wiener process. Consequently,  
these variables are Gaussian random variables influenced by the same Wiener process. 
Therefore, for $a \in \{u^C,u^D,u^{DL}\}$, \(Z_{n+1}\) and \(Q_{n+1}\) are correlated, as are \(Z_{n+1}\) and \(G_{n+1}\) for $a \in \{u^F,u^{FL}\}$. Next, we study the marginal and joint distributions of the state variables. 

\subsection{Conditional Distributions of State Variables}
\label{cond-distrib}
In this Subsection, we analyze the conditional marginal and joint distributions of  the process \(X_{n+1}=(Z_{n+1},Q_{n+1}, G_{n+1})\) given \(X_n=(Z_n, Q_n, G_n) = (z, q, g)\) and the action \(\alpha_n = a\), using the discrete-time approximation derived above.\\
We denote by
$\meanZ$,
$\meanQ$, and 
$\meanG$ 
the conditional means 
 and by $\ConVarZ$, 
$\ConVarQ$, and 
$\ConVarG$ 
the conditional variance of $Z_{n+1}$, $Q_{n+1}$, and $G_{n+1}$  given $X_n=x= (z, q, g)$ and $\alpha_n=a$, respectively.
Since $Z_{n+1}$ and $Q_{n+1}$ are correlated, we denote by
$\CovZQ$ and 
$\CorrZQ$ 
the conditional covariance and correlation between $Z_{n+1}$ and $Q_{n+1}$, respectively. $\CovZG$ and 
$\CorrZG$ 
represent the conditional covariance and correlation between $Z_{n+1}$ and $G_{n+1}$.

\paragraph{Conditional distribution of \(Z_{n+1}\) given \(Z_n = z\)}
We recall that at time $t_{n+1}$, the random variable \(Z_{n+1}\) is Gaussian. The following proposition provides its conditional mean and variance.
\begin{proposition}{\label{prop-Z}}
    \textbf{(Conditional Mean and Variance of \(Z_{n+1}\))} The conditional distribution of \(Z_{n+1}\) given \(Z_n=z\) is Gaussian with mean and variance given by
    \begin{equation}\label{Cond_mean_Z}
       \meanZ= z \mathrm{e}^{-\beta_R \Delta_N} \quad \text{and} \quad \ConVarZ = \frac{\sigma_R^2}{2\beta_R} \left(1 - \mathrm{e}^{-2\beta_R \Delta_N}\right),
    \end{equation}
   respectively.
\end{proposition}

\begin{proof}
    See Appendix \ref{proof-propZ}.
\end{proof}

\paragraph{Conditional distribution of $Q_{n+1}$ given $Q_n$ and $Z_n$}
 For $a \in \{u^C,u^D,u^{DL}\}$, the random variable $Q_{n+1}$, defined by the recursion in Equation~\eqref{Rec-Q1} is  Gaussian as an integral functional of the Wiener process with a deterministic integrand. The following proposition provides its conditional mean and variance.

\begin{proposition}\textbf{(Conditional Mean and Variance of \(Q_{n+1}\))} \label{prop-Q}

The conditional distribution of the battery level $Q_{n+1}$ given $Q_n = q$ and $Z_n = z$ is Gaussian.
\begin{enumerate}
    \item  The conditional mean is given by:
\begin{equation}
    \meanQ= q \mathrm{e}^{-\eta_0\Delta_N}- \frac{ \eta_E^n}{\capacitybat}H(n,a),
\end{equation}
where 
\begin{equation}
H(n,a)=\begin{cases}
		\frac{z}{\eta_0-\beta_R}\left(\mathrm{e}^{-\beta_R\Delta_N}-\mathrm{e}^{-\eta_0\Delta_N}\right)+ \frac{ \mu_{R,n}}{\eta_0 }(1-\mathrm{e}^{-\eta_0\Delta_N})  &   a \in \{u^C, u^D\},\\[1ex]
		\frac{R_{Q0} }{\eta_0 }(1-\mathrm{e}^{-\eta_0\Delta_N})&   a=u^{DL},\\
		0 &  \text{otherwise},
	\end{cases}
\end{equation}
    \item The conditional variance is given by 
\begin{equation}
    \ConVarQ = \begin{cases} \frac{\eta_{E,n}^2 \sigma_R^2}{2\beta_R C_Q^2}  I_1
          & \qquad a \in \{u^C, u^D\},\\
        0& \qquad \text{otherwise},
        \end{cases}  
\end{equation}
\end{enumerate}
where
\begin{align}
I_1 &= 
 \frac{1}{\eta_0^2 - \beta_R^2} \left(
\frac{\beta_R}{\eta_0} \left( 1-\mathrm{e}^{-2\eta_0 \Delta_N} \right)+ 1- 2\mathrm{e}^{-(\eta_0 +\beta_R)\Delta_N}+\mathrm{e}^{-2\eta_0 \Delta_N} \right) -\frac{\bigg(\mathrm{e}^{-\eta_0 \Delta_N}-\mathrm{e}^{-\beta_R \Delta_N} \bigg)^2 }{(\eta_0 - \beta_R)^2}. \label{cond_covQ}
\end{align}
\end{proposition}
\begin{proof}
See Appendix \ref{proof-propQ}.
\end{proof}

\paragraph{Conditional distribution of $G_{n+1}$ given $G_n$} Similarly, at time $t_{n+1}$, $G_{n+1}$ is an integral function of the Wiener process with deterministic integrand. Therefore, $G_{n+1}$ is Gaussian and the following proposition characterizes its conditional mean and variance.
 \begin{proposition}\label{prop-G}\textbf{(Conditional Mean and Variance of \(G_{n+1}\))} 
 
The conditional distribution of the fuel tank level at time $t_{n+1}$, $G_{n+1}$  given $G_n = g$ and $Z_n = z$ is Gaussian.
\begin{enumerate}
     \item  The conditional mean is given by:
\begin{equation}
    \meanG =  g - \frac{c_0}{C_G} \Delta_N -\frac{c_1}{C_G}\begin{cases}
    	\frac{z}{\beta_R}(1-\mathrm{e}^{-\beta_R\Delta_N})& \quad a = u^F\\[1ex]
    	R_{G0}\Delta_N & \quad a= u^{FL}\\
    	0 & \quad \text{otherwise},
    \end{cases}
\end{equation}
\item The conditional variance is given by
 \begin{align}
    \ConVarG 
    =\frac{ 1}{\beta_R^3}\left(\frac{c_1 \sigma_R}{\sqrt{2}C_G}\right)^2 \begin{cases}
       2\beta_R \Delta_N -3+4\mathrm{e}^{-\beta_R\Delta_N} -\mathrm{e}^{-2\beta_R\Delta_N}  & \quad a = u^F,\\
       0 & \quad  \text{otherwise.}
    \end{cases} 
\end{align}
\end{enumerate}
\end{proposition}

\begin{proof}
    See Appendix \ref{proof-propG}.
\end{proof}

\paragraph{Joint conditional distribution of \(Z_{n+1}\) and \(Q_{n+1}\)}
 Recall that at time \(t= t_{n+1}\), the residual demand \(Z_{n+1}\) and the battery level \(Q_{n+1}\), defined by the recursions \eqref{lem-Z} and \eqref{lem-Q}, respectively, are correlated and normally distributed random variables.  Therefore, their joint distribution is bivariate normal. The subsequent lemma derives their covariance and correlation coefficient.

\begin{lemma}[Conditional covariance of \(Z_{n+1}\) and \(Q_{n+1}\)]\label{CovZQ}

The conditional covariance \(\CovZQ\) of \(Z_{n+1}\) and \(Q_{n+1}\) is given by
\begin{align}
    \CovZQ = 
    -\frac{\eta_E^n \sigma_{R}^2}{2\beta_RC_Q}\begin{cases}
         \frac{1 - \mathrm{e}^{-(\eta_0+\beta_R)\Delta_N}}{\eta_0+\beta_R}  - \frac{\mathrm{e}^{-2\beta_R\Delta_N} - \mathrm{e}^{-(\eta_0+\beta_R)\Delta_N}}{\eta_0-\beta_R} & \qquad a\in\{u^D,u^C\},\\
        0 & \qquad \text{otherwise}.\end{cases} 
\end{align}
  \end{lemma}

    \begin{proof}
    See Appendix \ref{proof-CovZQ}.
\end{proof}

		\begin{proposition}{\label{Corr- ZQ}}(\textbf{Conditional correlation between \(Z_{n+1}\) and \(Q_{n+1}\)})
        The conditional correlation coefficient of \(Z_{n+1}\) and \(Q_{n+1}\) is given by 
		\begin{equation}
			\CorrZQ = \frac{\CovZQ}{\CondevZ\CondevQ}.
		\end{equation}
		Here, $\CondevZ$ and $\CondevQ$ are given by Propositions \ref{prop-Z} and \ref{prop-Q}, respectively, and $\CovZQ$ is given by Lemma \ref{CovZQ}.
\end{proposition}

\paragraph{Joint conditional distribution of \(Z_{n+1}\) and \(G_{n+1}\)}  Recall that at time \(t= t_{n+1}\), the residual demand \(Z_{n+1}\) and the fuel tank level \(G_{n+1}\), defined by Lemma \eqref{lem-Z} and Lemma \eqref{lem-G}, respectively, are correlated and normally distributed random variables.  Hence, their joint distribution is bivariate normal. The following Lemma derives their covariance and correlation coefficient.

\begin{lemma}[Conditional covariance of \(Z_{n+1}\) and \(G_{n+1}\).]\label{CovZG}
    
    The conditional covariance of \(Z_{n+1}\) and \(G_{n+1}\) is given by
\begin{align}
    \CovZG =  -\frac{c_1 \sigma_R^2}{2 C_G \beta_R^2}
    \begin{cases}
 (1-\mathrm{e}^{-\beta_R\Delta_N})^2 & \qquad a=u^F,\\
        0 &\qquad \text{otherwise}.        
    \end{cases}
\end{align}
    \end{lemma}
   \begin{proof}
    The proof of this theorem can be found in Appendix \ref{proof-CovZG}.
\end{proof}  
  
\begin{theorem}{\label{Corr-ZG}}(\textbf{Conditional correlation between \(Z_{n+1}\) and \(G_{n+1}\)}).
The correlation coefficient between the random variables \(Z_{n+1}\) and \(G_{n+1}\) is given by 
	\begin{equation}
		\CorrZG = \frac{\CovZG}{\CondevZ \CondevG},
	\end{equation}	
	where $\Sigma_Z$ and $\Sigma_G$ denote the conditional standard deviation of $Z_{n+1}$ and $G_{n+1}$ given by \eqref{prop-Z} and \eqref{prop-G}, respectively. Here,  $\Sigma_{ZG}$ is the conditional covariance of $Z_{n+1}$ and $G_{n+1}$,  \(\sigma_{ZG},\) is given by Lemma \ref{CovZG}
\end{theorem}

 \begin{remark} \label{Rem-Tkernel}
	\phantom{x} 
	Note that according to Assumption \ref{ass}, the battery and generator are never used at the same time. Hence,  there is no correlation between \(Q_{n+1}\) and \(G_{n+1}\).
    Further, for $\nu \in \controlset\backslash \{ u^C,u^D,u^{DL}\}$, $\Sigma_Q(n,\nu)=0$ and $\nu \in \controlset\backslash \{u^F,u^{FL}\}$, $\Sigma_G(n,\nu)=0$. Then~$\CovZG=0$ for $\nu \in \controlset\backslash \{u^F,u^{FL}\}$ and $\CovZQ=0$ for $\nu \in \controlset\backslash \{ u^C,u^D,u^{DL}\}$. Therefore, the correlation between $Z_{n+1}$ and \(Q_{n+1}\) and the correlation between $Z_{n+1}$ and \(G_{n+1}\) never happens simultaneously, that is, $\CorrZQ \CorrZG=0$ for all $a \in \controlset$ and $n=0, \ldots, N-1$. 
\end{remark}

\paragraph{Transition kernel}
Based on the above closed-form expressions, the marginal and joint distributions of the state variables, we can define the recursion for the linear transition operator associated with the MDP.
\begin{proposition}[Transition operator] Let $x=(z,q,g) \in \Statespace$, $a \in \controlset$, and for all  $n = 0, \dots, N-1$, let $\mathcal{T}_n : \Statespace \times \controlset \times \mathcal{E} \rightarrow \Statespace$ denote the transition operator at time step $n$.\\ 
Then, there exists a sequence of independent standard normally distributed random vectors $(\mathcal{E}_n)_{n =1,\dots,N}$ with $\mathcal{E}_n = (\mathcal{E}_n^Z,\mathcal{E}_n^Q,\mathcal{E}_n^G)^T {\revi \sim}~ \mathcal{N}(0_3,\mathds{I}_3)$ such that the state process $X=(X_n)_{n=0,\ldots N}$ satisfies the recursion
\begin{equation}
X_{n+1} = \mathcal{T}_n(X_n, \alpha_n, \mathcal{E}_{n+1}),~~ X_0=X(0)=x_0, \quad \text{ for } n = 0, \dots, N-1.
\label{Trans_operator}
\end{equation}
Here, $\mathds{I}_3$ is a $3 \times 3$-identity matrix and the transition operator $\mathcal{T}_n$ is defined for all $n=0, \ldots,N-1$  and for $\epsilon=(\epsilon^Z,\epsilon^Q, \epsilon^G) \in \R^3$ by  \(\mathcal{T}_n = (\mathcal{T}_n^Z,\mathcal{T}_n^Q,\mathcal{T}_n^G)\) with
	\begin{align*}
			\mathcal{T}_n^Z(x,a,\epsilon )&= \meanZ + \CondevZ \epsilon^Z,\\
		\mathcal{T}_n^Q(x,a,\epsilon) &= \meanQ+\CondevQ \left(\sqrt{1-\rho_{Q}^2(n,a)}\epsilon^Q + \rho_{Q}(n,a)\epsilon^Z \right), \\
\mathcal{T}_n^G(x,a,\epsilon) &= \meanG + \CondevG \left(\sqrt{1-\rho_{G}^2(n,a)}\epsilon^G + \rho_{G}(n,a)\epsilon^Z \right),	
	\end{align*}
	where, $m_Z$, $m_Q$, and $m_G$, 
 are the conditional means and $\Sigma_Z$, $\Sigma_Q$, and $\Sigma_G$ are the standard deviations of  $Z_{n+1}$, $Q_{n+1}$, and $G_{n+1}$, respectively. 
\end{proposition}

\begin{proof}
The proof follows from  \cite[Proposition 6.3]{takam2025cost}.    
\end{proof}
The closed-form solutions of the deseasonalized residual demand given in Lemma~\(\ref{lem-Z}\), the approximations of the battery level given in Lemma \(~\ref{lem-Q}\), and the fuel tank level given in Lemma~\ref{lem-G} show that \(Z_{n+1}\) is normally distributed, and for $a \in \{ u^C,u^D,u^{DL}\}$, the discrete-time state random variable \(Q_{n+1}\) is normally distributed, and for $\nu \in  \{u^F,u^{FL}\}$, \(G_{n+1}\) is normally distributed. Therefore, as explained in Remark \ref{Rem-Tkernel}, the conditional distribution of the process \(X_{n+1} = X(t_{n+1})\) is a multivariate Gaussian (degenerated) with the mean \(\meanX\) and the covariance matrix $\ConVarX$ given by
\begin{align*}
	\meanX = \begin{pmatrix}
				\meanZ\\
		\meanQ\\
		\meanG
	\end{pmatrix},~~~~~~~~~~~ & \ConVarX = \begin{pmatrix}
	\ConVarZ&\CovZQ& \CovZG\\
	\CovZQ&\ConVarQ&0\\
	\CovZG&0&\ConVarG
	\end{pmatrix},
\end{align*}
where \(\Sigma_{ZG}=\rho_{Q}\Sigma_{Z}\Sigma_{Q}\) and $\Sigma_{ZG}=\rho_{G}\Sigma_{Z}\Sigma_{G}$ are  given in Lemmas \ref{CovZQ} and \ref{CovZG}, respectively.

\subsection{State-Dependent Control Constraints}
\label{state-depend}
The microgrid considered in this paper is subject to various constraints, including box constraints to the battery state of charge and to the fuel tank level given by Equation \eqref{Constraint_state}, that is, the battery level $Q(t) \in [0,1]$, for all $t\in[0,T]$, where $Q(t)=0$ and $Q(t)=1$ represent an empty and a full battery level, respectively. In addition to that, the fuel tank level $G(t)$ is required to take values in a bounded set, that is, $G(t) \in [0,1]$, for all $t\in[0,T]$, where $G(t)=0$ and $G(t)=1$ represent an empty and a full fuel tank, respectively.  

In a continuous-time model, where the controls continuously change over time, such restrictions mean that discharging the battery or starting the generators is no longer permitted when the battery or the fuel tank is empty, while charging the battery can no longer be selected for full storage. However, due to Assumption \ref{Ass-co}, we are limited to constant controls between two consecutive time points; that is, a control chosen at time step $n$ can only be changed at time step $n+1$. Hence, contrary to continuous time, here charging or discharging is no longer allowed when the battery is “almost” full or empty. Similarly, starting the generator is no longer possible when the fuel tank is “almost” empty. 
Therefore, for each time step $n \in \{0,\ldots, N-1\}$, one has to derive a subset of the set of all feasible actions $\mathcal{U} \subset \controlset$ given in \eqref{Controlset}, which contains the feasible actions available to the controller depending on the state $X_n$ at time $n$. This should be defined in such a way that the battery does not become full or empty within the next period $  [ t_n,t_{n+1} ) $. In addition, the control must be chosen so that the fuel tank does not become empty in the next period $  [t_n,t_{n+1} )$. 
This leads to the following implicit definition of the set of state-dependent control constraints
\[ \mathcal{U}(n, x) = \{ a \in U \mid \mathcal{T}_n(x,a,\mathcal{E}_{n+1}) \text{ satisfies the state constraint}\},\]
\(n=0,\ldots,N\text{ and}~ x\in\mathcal{X}\).
This definition is intuitive but requires additional clarification. Now, we show how this intuitive definition can be formulated mathematically in a rigorous way.
Recall that from the recursion \eqref{Trans_operator}, we have $Q_{n+1}=\mathcal{T}^Q_n(X_n, \alpha_n,\mathcal{E}_{n+1})$ and $G_{n+1}=\mathcal{T}^G_n(X_n, \alpha_n,\mathcal{E}_{n+1})$. In addition, the conditional distribution of the battery level $Q_{n+1}$ and the fuel tank level $G_{n+1}$ given $X_n$ is Gaussian. This procedure prevents us from satisfying the box constraint $Q_{n+1} \in [0, 1]$ and $G_{n+1} \in [0, 1]$ with certainty and requires a relaxation. As a realization of a Gaussian random variable, these values are potentially unbounded, making it impossible to define reasonable relaxed constraints to $Q_{n+1}$ and $G_{n+1}$.
Therefore, for the battery state of charge, we allow over- and undershooting; that is, when $Q_{n+1} >1$, charging is allowed, and when $Q_{n+1} <0$, discharging is allowed. Similarly, for the fuel tank, we allow undershooting; that is, when meeting the demand using the generator, $G_{n+1}<0$ is allowed. However, we constrain the probabilities for these events by some ``small'' tolerance value $0<\epsilon \ll 1$. {\revi Here, \(\epsilon\) is chosen sufficiently small so that the battery and fuel tank levels satisfy each of the above constraints with high probability at the next time step.}

In view of state discretization, this relaxation appears to be acceptable and reasonable, since the grid points on the boundaries of the truncated state space, in particular the points on $Q_{n+1}=1$, $Q_{n+1}=0$, and $G_{n+1}=0$, will represent all points in the state space with $Q_{n+1} \geq 1$, $Q_{n+1} \leq 0$, and $G_{n+1} \leq 0$, respectively. {\revi Summarizing, the desired set of feasible actions is then given by the intersection}
\begin{align}
    \mathcal{U}(n,x) = \mathcal{U}_Q(n,x) \cap \mathcal{U}_G(n,x),
    \label{union_actions}
\end{align}
{\revi where both sets are defined on the same set $\mathcal{U}_r(n,x)\subset U$ of actions that are compatible with the sign of the residual demand $r=\mu_R(t_n)+z$ at time $n$,
\[
\mathcal{U}_r(n,x)=\begin{cases}\{u^C,u^O\}, & r<0,\\ \{u^W,u^{DL},u^D,u^{FL},u^F\}, & r>0,\end{cases}
\]
where $u^{DL}$ and $u^{FL}$ are only available for $r\geq R_{Q0}$ (resp.\ $r\geq R_{G0}$), and}
\begin{align}
\mathcal{U}_Q(n,x) &= \{a\in \mathcal{U}_r(n,x)~|~~\mathbb{P}(\mathcal{T}_n^Q(x,a,\mathcal{E}_{n+1})<0)<\epsilon, \nonumber\\& \qquad \qquad \quad~~ \mathbb{P}(\mathcal{T}_n^Q(x,a,\mathcal{E}_{n+1})>1) < \epsilon \},\label{Constraint_Q}\\
\mathcal{U}_G(n,x) &= \{a \in \mathcal{U}_r(n,x)~|~ ~\mathbb{P}(\mathcal{T}_n^G(x,a,\mathcal{E}_{n+1})<0)<\epsilon  \} .\label{Constraint_G}
\end{align}
{\revi An action is therefore feasible if it satisfies the battery constraint and the fuel constraint at the same time. An action that does not operate a storage leaves its level unchanged (up to the self-discharge of the battery) and thus satisfies the corresponding constraint automatically: the generator actions $u^{FL},u^F$ belong to $\mathcal{U}_Q(n,x)$ and the battery actions $u^C,u^{DL},u^D$ belong to $\mathcal{U}_G(n,x)$. The sets $\mathcal{U}_Q$ and $\mathcal{U}_G$ are then not restricted to battery-only and generator-only actions, respectively, and the feasible set is the intersection of the two.}
\begin{figure}[h!]
	\centering
    \includegraphics[width=0.49\linewidth,height=0.24\linewidth]{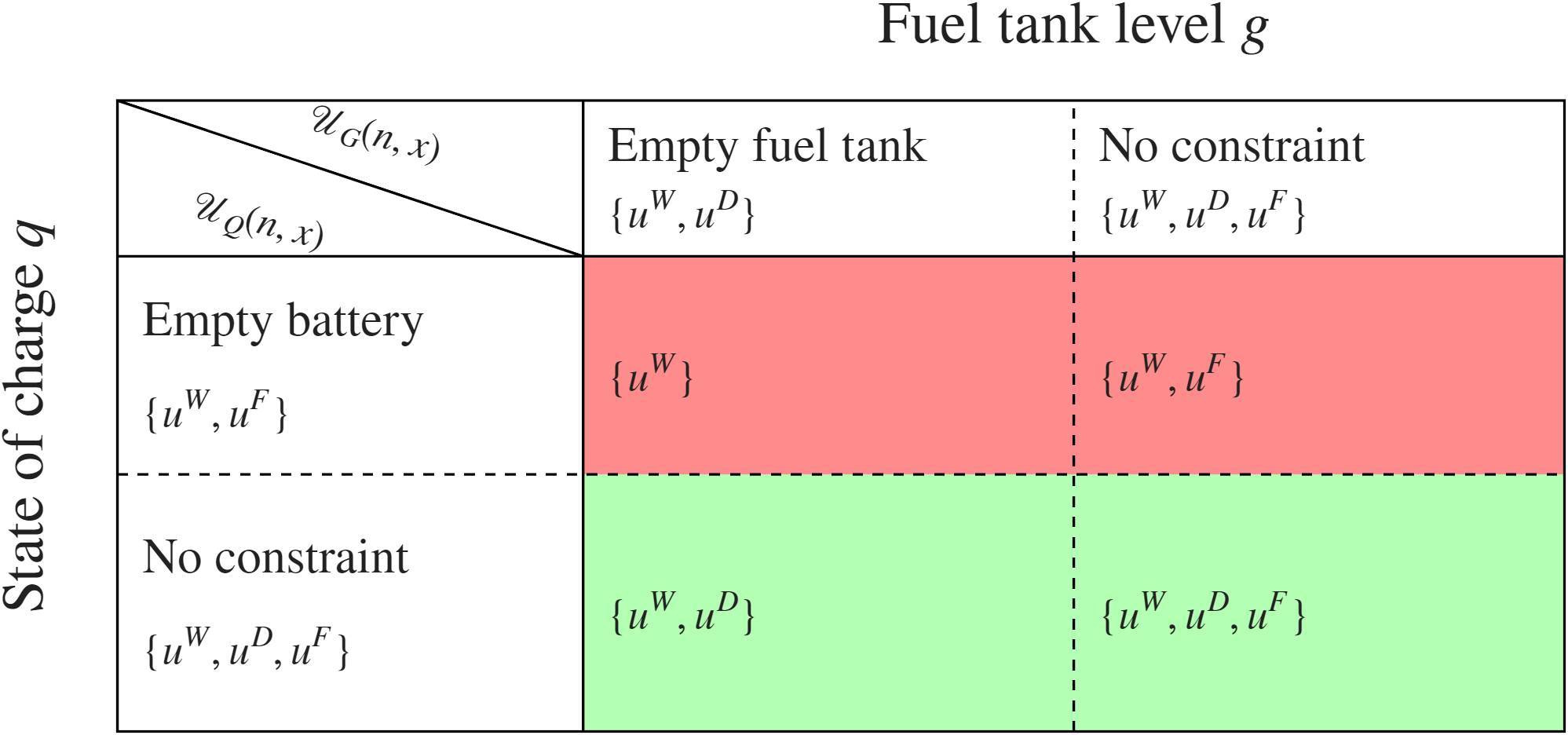}
\includegraphics[width=0.49\linewidth,height=0.24\linewidth]{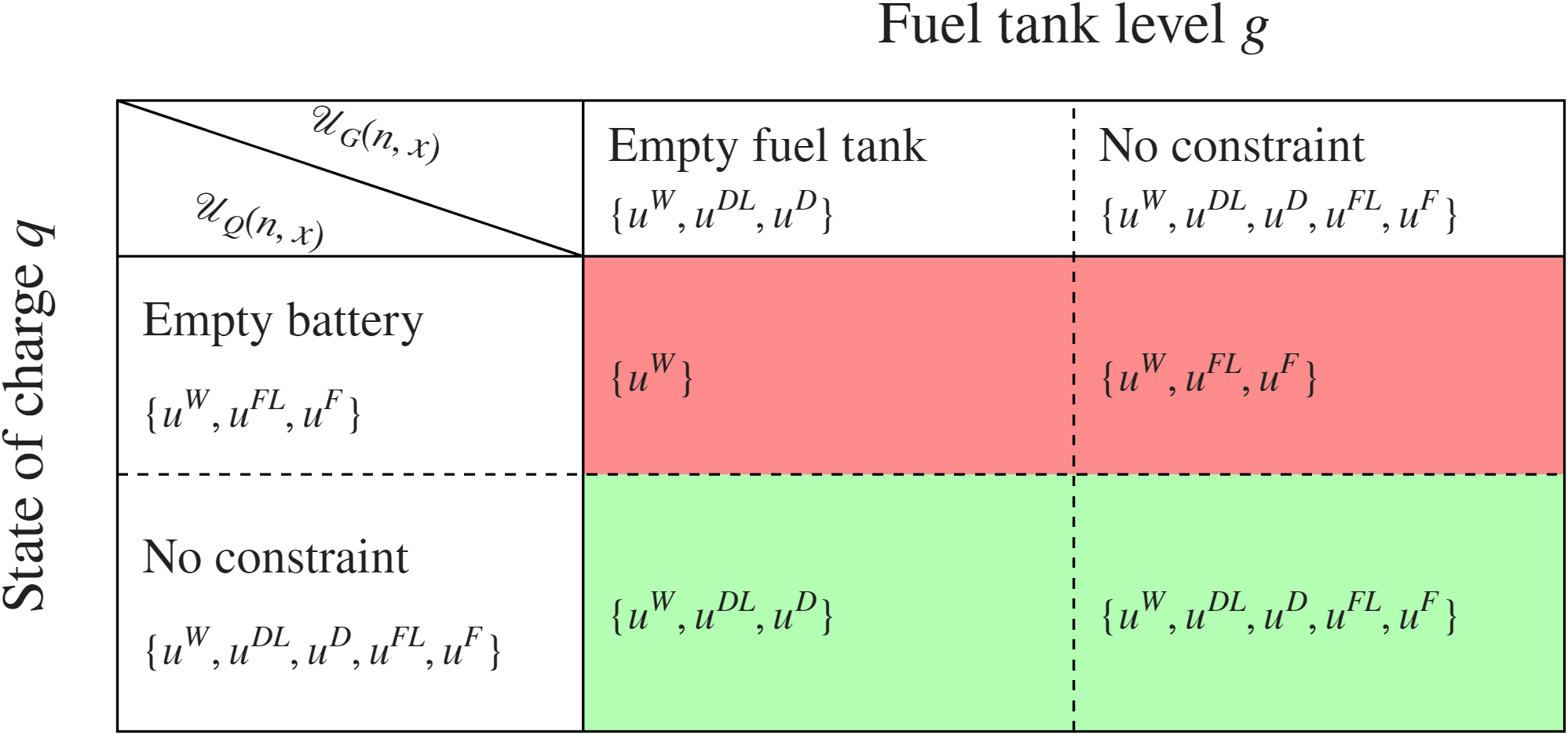}
	\caption{Set of feasible controls $\mathcal{U}(n,x)=\mathcal{U}_Q(n,x) \cap\mathcal{U}_G(n,x)$ for a positive residual demand $(r>0)$ and $R_{Q0}=R_{G0}$. {\revi The rows give $\mathcal{U}_Q(n,x)$ and the columns $\mathcal{U}_G(n,x)$, both including the actions that do not operate the corresponding storage; each cell is the intersection of the corresponding row and column.} \\Left: $r< R_{Q0}$.  Right: $r\geq R_{Q0}$}

	\label{Fig_actionspace}
\end{figure}
 Furthermore, the recursion \eqref{Trans_operator} for $Q$ and $G$ shows that  the Gaussian distributions of $Q_{n+1}$ and $G_{n+1}$ depend on the residual demand $R_n$ at time $n$.
Therefore, based on the sign of the residual demand, the set $\mathcal{U}(n,x)$ can be explicitly characterized. For example, when the residual demand at time $n$ is negative (overproduction), starting the generator or discharging the battery is not feasible. {\revi Then $\mathcal{U}_r(n,x)=\{u^C,u^O\}$; since $u^C$ and $u^O$ do not use the fuel, $\mathcal{U}_G(n,x)=\{u^C,u^O\}$, and hence $\mathcal{U}(n,x)=\mathcal{U}_Q(n,x)\subset \{u^C,u^O\}$.}
Similarly, when the residual demand at time step $n$ is positive (unsatisfied demand), charging the  battery and over-spilling are no longer feasible. {\revi Then $\mathcal{U}_r(n,x)=\{u^W,u^{DL},u^D,u^{FL},u^F\}$. The battery constraint excludes $u^{DL}$ and $u^D$ when the battery is empty, the fuel constraint excludes $u^{FL}$ and $u^F$ when the fuel tank is empty, and neither constraint restricts the actions of the other storage. For example, for an empty battery and a non-empty tank we obtain $\mathcal{U}_Q\cap\mathcal{U}_G=\{u^W,u^{FL},u^F\}\cap\{u^W,u^{DL},u^D,u^{FL},u^F\}=\{u^W,u^{FL},u^F\}$, and for an empty battery and an empty tank $\mathcal{U}_Q\cap\mathcal{U}_G=\{u^W,u^{FL},u^F\}\cap\{u^W,u^{DL},u^D\}=\{u^W\}$. Summarizing, the set of feasible controls for $r>0$ is illustrated in Fig.~\ref{Fig_actionspace}.}
\begin{remark}
    In the above description of the set of feasible controls, it is assumed that the sign of the residual demand does not change in a small time interval $[t_n, t_{n+1})$. When the residual demand at time $t_n$ is strongly negative or positive above the threshold $R_{Q0}$ (resp. $R_{G0}$), it is more likely that its sign will not change in the interval $[t_n, t_{n+1})$. However, when the residual demand is close to zero at time $t_n$, it is likely that the sign changes in the interval $[t_n, t_{n+1})$. In this case, the set of feasible controls is reduced to doing nothing; that is, we set $\mathcal{U}(n,x)=\{u^W\}$.
\end{remark}

\subsection{Markov Decision Process}
\label{MDP}
We consider a Markov decision process with a finite time horizon $T$, finite action space $\mathcal{U}(n,x)$, and the state $\Statespace \subset \R^3$ described above.  We now want to derive the discrete-time version of the performance criterion from the continuous-time performance given in {\revi \eqref{performance}}. This summarizes the expected total discounted costs from the operation of the system and the expected discounted terminal  costs.  

 \paragraph{Admissible Controls}  Let \(\mathbb{G} = (\mathcal{G}_n)_{n=0,\ldots,N}\) denote the filtration generated by the sequence of independent identically distributed random vectors \((\mathcal{E}_n)_{n=1, \ldots, N}\). We denote by $\mathcal{G}_n:=\sigma(\mathcal{E}_1,\ldots,\mathcal{E}_n)$ a \(\sigma\)-algebra generated by the first $n$ random vectors $\mathcal{E}_1,\ldots,\mathcal{E}_n$, and $\mathcal{G}_0=\{\varnothing,\Omega\}$ the trivial \(\sigma\)-algebra. \\  
 We define the set of admissible controls \(\alpha = (\alpha _0, \dots, \alpha _{N-1})\), denoted by \(\mathcal{A}\), for which we want to define the performance criterion below. Since we want to apply dynamic programming methods to solve the optimization problem, we restrict ourselves to Markov or feedback controls defined by $\alpha_n=\tilde{\alpha}(n,x)$ with a measurable function   $\tilde{\alpha}:\{0,\ldots,N-1\}\times \Statespace\to \mathcal{U}$, which is called a \textit{decision rule}. This decision rule is selected such that the controlled state process \(X=X^{\alpha}\) remains within the state space \(\mathcal{X}\) with high probability  for all \(n=0,\ldots,N\). Formally, the set of admissible controls is defined as
    \begin{align*}
        \mathcal{A} = \bigg\{ \alpha = (\alpha _0, \dots, \alpha _{N-1}) \mid \alpha_n = \tilde{\alpha}(n, X_{n}) \text{ for all } n = 0, \dots, N-1, \\
        \tilde{\alpha }(n,x) \in \mathcal{U}(n,x) \text{ for all } (n, x) \in \{0, \dots, N-1\} \times \mathcal{X} \bigg\}.
    \end{align*}
    
\paragraph{Performance criterion}
The performance criterion $J: \{0,\ldots,N\} \times \mathcal{X} \times \mathcal{A} \to \R$ given in \eqref{performance} can be rewritten in terms of the sequence of  values $(X_n^{\alpha})_{n=0,\ldots,N}$ obtained from sampling of the continuous-time state process $(X^u(t))_{t\in[0,T]}$. Since the admissible controls are of the feedback type, the state process is also a Markov process. 
Consider a control \(\alpha = (\alpha_0, \dots, \alpha_{N-1})  \in \mathcal{A}\) and $x=({\revi z},q,g) \in \Statespace$. 
The discrete-time performance criterion
is given by the following lemma:
\begin{lemma} Let $x=({\revi z},q,g)$ and $a=\alpha_n$.
The discrete-time version of the performance criterion \eqref{performance} can be written as
    \begin{align}
 J{^D}(n,x; \alpha)  =\mathbb{E}_{n,x}\Bigg[\sum_{k=n}^{N-1} \Psi^D(k,X_k^{\alpha},\alpha_k) +{\termcostD}(X_N^{\alpha}) \bigg],
    \label{performance_D}
\end{align}
where for all  $k=n,\ldots, N-1, ~ x \in \Statespace$, and $a \in \mathcal{U}$ 
\begin{align}
    \qquad ~\Psi^D(k,x,a)=\mathrm{e}^{-\rho (k-n)\Delta_N}\Psi(k,x,a) ~~\text{ and }~~  {\termcostD}(x) = \mathrm{e}^{-\rho (N-n)\Delta_N}\phi_N(x)
    \label{phi_d}
\end{align}
with
\begin{align}
     \Psi(k,x, a) &= \mathbb{E}\left[\int_{t_k}^{t_{k+1}} \mathrm{e}^{-\rho (s-t_k)}\psi(s,X^u(s),a)ds \mid \mathcal{F}_k\right], ~\text{with } \mathcal{F}_k = \mathcal{F}_{t_k}. 
        \label{one-step}
\end{align}
\end{lemma}
Here, $\mathbb{E}_{n,x}(\cdot) =  \mathbb{E}(\cdot| \mathcal{F}_n )=\mathbb{E}(\cdot| X_n = x) $ denotes the conditional expectation given that at time $t_n$ the state is $X_n=x$, and $\mathbb{E}(\cdot| \mathcal{F}_n )$ denotes the conditional expectation given the available information up to time n. $ \rho > 0$ denotes the discount rate, \(\psi\) represents the running cost, and \(\phi\) is the terminal cost given in Subsection~\ref{Stoc-OPtCont}. 
\begin{proof}
The discrete-time performance criterion is given as
    \begin{equation}
    J(n, x; \alpha ) = \mathbb{E}_{n,x}\left[\int_{t_n}^{t_N} \mathrm{e}^{-\rho (s-t_n)}\psi(s,X^u(s),u(s))ds + \mathrm{e}^{-\rho (t_N-t_n)}\phi_N(X_N^{\alpha})\right].
\end{equation} 

Let \({ \termcostD}(x)= \mathrm{e}^{-\rho (N-n)\Delta_N}\phi_N(x)\). By applying the tower property of conditional expectation, we can rewrite the performance criterion as follows:
\begin{align*}
	J(n, x; \alpha ) &=\mathbb{E}_{n,x}\bigg[
	\sum_{k=n}^{N-1} 	\mathbb{E}\Big[\int_{t_k}^{t_{k+1}}\mathrm{e}^{-\rho (s-t_n)}\psi(X^u(s),u(s))\mathrm{d}s\Big| \mathcal{F}_{t_k}\Big] + \mathrm{e}^{-\rho (t_N-t_n)}\phi_N(X_N^{\alpha})\bigg] \\
     &= \mathbb{E}_{n,x}\bigg[\sum_{k=n}^{N-1} \mathrm{e}^{-\rho(t_k - t_n)}  \mathbb{E} \bigg[ \int_{t_k}^{t_{k+1}} e^{-\rho(s-t_k)} \psi(s, X^u(s), u(s)) ds  | \mathcal{F}_{t_k}\bigg] + \termcostD(X_N^{\alpha})\bigg],\notag \\
   &= \mathbb{E}_{n,x}\left[\sum_{k=n}^{N-1} \mathrm{e}^{-\rho (k-n)\Delta_N}\Psi(k, X^{\alpha}_k, \alpha_k)+ \termcostD(X^{\alpha}_N)\right].
\end{align*}
 For  $\Psi^D(k,x,a)=\mathrm{e}^{-\rho (k-n)\Delta_N}\Psi(k,x,a)$ with  $\Psi$ given by \eqref{one-step}, the result follows.
 
\end{proof}
 
The next lemma shows that the conditional expectation appearing in Equation \eqref{one-step} can be written in closed-form. Therefore, the discrete-time approximation of the performance criterion \eqref{performance} does not have additional discretization errors.

\begin{lemma}
\label{Conditional_C}
     For $k=0,\ldots, N-1$, $x=({\revi z},q,g)$, and $a=\alpha_k$, the conditional expectation appearing in \eqref{one-step} can be computed in closed-form as follows:
     \begin{align}
     \Psi(k,x,a)&=\begin{cases}
        \left[(c_0+c_1\mu_{R,k})\zeta_1+ c_1 z\zeta_2\right]F_0, & \text{ } \quad a= u^{F}, \\[1ex]
        \left( (c_0 + c_1 R_{G0})F_0 + k_0(\mu_{R,k} - R_{G0})^2 + \frac{\sigma_R^2k_0}{2\beta_R} \right)\zeta_1\\ + 2z k_0(\mu_{R,k} - R_{G0})\zeta_2+ k_0\left(z^2 - \frac{\sigma_R^2}{2\beta_R}\right) \zeta_3, & \text{ } \quad a = u^{FL}, \\[1ex]
       \gamma_{deg} \left( \mu_{R,k}\zeta_1 +  z\zeta_2\right),& \text{}~\quad a= u^{D}, \\
        \left(\gamma_{\text{deg}}R_{Q0}  + k_0(\mu_{R,k} - {\revi R_{Q0}})^2 + \frac{\sigma_R^2k_0}{2\beta_R} \right)\zeta_1 \\\quad+ 2z k_0(\mu_{R,k} - {\revi R_{Q0}}) \zeta_2+k_0\left(z^2 - \frac{\sigma_R^2}{2\beta_R}\right) \zeta_3, & \text{ }\quad  a= u^{DL}, \\[1ex]
       -\gamma_{deg} \left( \mu_{R,k}\zeta_1 +  z\zeta_2 \right),& \text{ } \quad a= u^{C}, \\[1ex]
       k_0\bigg[\zeta_1\left(\mu^2_{R,k}+ \frac{\sigma_R^2}{2\beta_R} \right) + 2z \zeta_2\mu_{R,k}+\left(z^2 - \frac{\sigma_R^2}{2\beta_R}\right) \zeta_3\bigg] & \text{ } \quad a = u^{W}, \\[1ex]
        0, & \text{} ~\quad a = u^{O},
     \end{cases}
         \label{conditional-psi}
     \end{align}    
 where
  \begin{equation}\label{zeta}
 	\zeta_1= \frac{1 - \mathrm{e}^{-\rho \Delta_N}}{\rho},~~\zeta_2=\frac{1 - \mathrm{e}^{-(\rho + \beta_R)\Delta_N}}{\rho + \beta_R}, ~~ \text{ and }  ~~\zeta_3  =\frac{1 - \mathrm{e}^{-(\rho + 2\beta_R)\Delta_N}}{\rho + 2\beta_R}. 
 \end{equation}
\end{lemma}

\begin{proof}
    The proof of this lemma can be found in Appendix \ref{appen_Proof_condiC}.
\end{proof}
 \paragraph{Optimal Control Problem} 
 The objective of the system's manager is to find an admissible control process $\alpha$ defined by the associated decision rule $\widetilde{\alpha}$ that minimizes the expected total discounted costs arising from the operation of the system and the evaluation of the stored electricity in the battery and the fuel in the tank at the terminal time.
In other words, we want to minimize the performance criterion \(J(0, x, \alpha)\), as defined in \eqref{performance_D}, on the set of all admissible controls \(\mathcal{A}\).  The value function \(V (n, x)\) for all \(x \in \mathcal{X}\) and \(n = 0, \dots, N-1\) is defined as:

\begin{align}
\label{Val_function}
    V (n,x) = \inf_{\alpha \in \mathcal{A}} J{^D}(n, x, \alpha).
\end{align}

An admissible control \(\alpha^* = (\alpha_0^*, \dots, \alpha_{N-1}^*) \in \mathcal{A}\) is called an \emph{optimal control strategy} if it satisfies
\(V(n,x) = J{^D}(n, x, \alpha^*)\) for $(n,x) \in \{0,\ldots,N-1\}\times \Statespace$.

\paragraph{Dynamic programming equation}
The Bellman equation, also known as the dynamic programming equation, provides a necessary condition for optimality in dynamic programming problems.  It characterizes the value function and forms the basis for a backward recursion algorithm to compute an optimal control policy. For more details on the MDP theory, we refer to Bäuerle and Rieder \cite{bauerle2011markov}, Puterman \cite{puterman2014markov}, Hern{\'{a}}ndez-Lerma and Lasserre \cite{HernandezLerma1996}, and the references therein.

The Bellman equation is formally stated in the following theorem.
\begin{theorem}[Bellman Equation]
\label{Bellman}
For all $x \in \mathcal{X}$, the value function \(V\) satisfies the Bellman equation:
\begin{align}
     \label{dyn}
       V(N,x) &= \Phi(x), &  \\[1ex]
        V(n,x)& = \inf_{a \in \mathcal{U}(n,x)} \left\{ {\revi \Psi(n,x,a) + \mathrm{e}^{-\rho\Delta_N}}\, \mathbb{E}[V(n+1,\mathcal{T}_n(x,a,\mathcal{E}_{n+1}))] \right\}, &~ \, n = 0, \dots, N-1.
\end{align}
{\revi Here $\Psi(n,x,a)$ is the one-step cost \eqref{one-step}, which is already discounted within the step $[t_n,t_{n+1}]$, and the factor $\mathrm{e}^{-\rho\Delta_N}$ discounts the value at $t_{n+1}$ to $t_n$. Thus $V(n,x)$ is expressed in values at time $t_n$, and $\Psi^D(n,x,a)=\Psi(n,x,a)$ for $k=n$ in \eqref{phi_d}.}
The optimal decision rule \(\tilde{\alpha}^*\) is defined as:
\begin{equation} \label{opt_rule}
    \tilde{\alpha}^*(n,x) = \arg\min_{a \in \mathcal{U}(n,x)} \left\{ {\revi \Psi(n,x,a) + \mathrm{e}^{-\rho\Delta_N}}\, \mathbb{E}[V(n+1,\mathcal{T}_n(x,a,\mathcal{E}_{n+1}))] \right\}.
\end{equation}
\end{theorem}
The dynamic programming equation \eqref{dyn} can be solved using a backward recursion algorithm  starting at the terminal time $N$.
 
The algorithm is presented below:

\begin{algorithm}[H]
	\SetAlgoLined
    \KwData{ Given the terminal cost $\Phi(x)$  for all $x \in \Statespace$
}
	\KwOut{Find the value function $V$ for all $(n,x)  \in \{N, \ldots,0\} \times \Statespace$and the optimal strategy $\mathbf{\alpha}^*$ for all $(n,x)  \in \{N-1, \ldots,0\} \times \Statespace$.} 
	\begin{enumerate}
		\item [$\mathbf{1.}$] Compute for all $x \in \mathcal{X}$
		$$V(N,x)=\Phi(x).$$
		\item [$\mathbf{2.}$] For $n:=N-1,\ldots, 0$ compute for all $x \in \mathcal{X}$
		$$V(n,x)=\inf_{a \in  \mathcal{U}(n,x)}\bigg\{\Psi(n,x,a)+\mathbb{E} \big[V(n+1,\mathcal{T}_n(x,a,\mathcal{E}_{n+1})\big]\bigg\}.$$
		Compute the minimizer $\alpha_n^*$  given by
		$$\widetilde{\alpha}^*(n,x)=\underset{a \in  \mathcal{U}(n,x)}{\mathrm{argmin}}\bigg\{\Psi(n,x,a)+\mathbb{E} \big[V(n+1,\mathcal{T}_n(x,a,\mathcal{E}_{n+1})\big]\bigg\}.$$
	\end{enumerate}
	\caption{Backward recursion algorithm}
	\label{bellman_backward_algo}
\end{algorithm}

\label{State-discrete}
\subsection{State-Discretization}
Solving MDP with continuous state spaces can be computationally expensive, and it becomes numerically intractable when the dimension is high.  Note that, due to state-dependent control constraints, the closed-form expressions of expectation $\mathbb{E} \big[V(n+1,\mathcal{T}_n(x,a,\mathcal{E}_{n+1}))\big] $, appearing in the Bellman equation \eqref{dyn}, cannot be expected.	\\
To overcome these problems, we describe below a computationally usable approximation of this expectation based on state discretization.
We must first truncate the unbounded state into a bounded domain.
The interval \(\R\) of values of the deseasonalized residual demand is reduced to a closed interval \(\overline{\stZ} = [\underline{z},\overline{z}]\in \mathbb{R}\), within which the values of the random process \(Z\) lie with high probability. Since the closed-form solution of the SDE governed by the deseasonalized residual demand is a Gaussian Ornstein-Uhlenbeck process for which the marginal distribution of $Z(t)$ converges asymptotically for $t\to\infty$ to the stationary distribution $\mathcal{N}(0, \sigma_R^2/(2\beta_R))$, where \(\beta_R\) is the mean reversion speed and \(\sigma_{R}\) is the standard deviation. We apply the $3\sigma$-rule and choose the truncated interval, which carries $99.97\%$ of the probability mass of this distribution as follows:
 \[ [\underline{z},\overline{z}]=\bigg[ -\frac{3\sigma_{R}}{\sqrt{2\beta_R}},  ~\frac{3\sigma_{R}}{\sqrt{2\beta_R}}\bigg].\]
The continuous state space is divided into a subset using a set number of points \(N_Z,N_Q\) and \(N_G\), where we define \(z_i=i\Delta_z,\)  \(i\in \{0,\cdots,N_Z\}\), \(q_j=j\Delta_q\), \(j\in \{0,\cdots,{\revi N_Q}\}\), and \(g_k=k\Delta_g\), \(k\in \{0,\cdots,N_G\}\), where \(\Delta_z = \frac{\overline{z}-\underline{z}}{N_Z}\), \(\Delta_q = \frac{1}{N_Q}\), and \(\Delta_g = \frac{1}{N_G}\) are equidistant step sizes in the \(z,~q\), and \(g\) directions, respectively.
Then, the 3-dimensional discretized state space is given by
\begin{align*}
	\widetilde{\mathcal{X}} = \widetilde{\stZ}\times\widetilde{\mathcal{Q}}\times\widetilde{\mathcal{G}} =  \{z_0,\cdots,z^{}_{N_Z}\}\times\{q_0,\cdots,q^{}_{N_Q}\}\times \{g_0,\cdots,g^{}_{N_G}\}.
\end{align*}
Now, let us denote by \( \mathcal{N}_Z = \{0,\cdots,N_Z\}, ~~~\mathcal{N}_Q = \{0,\cdots,N_Q\}, \text{ and}~~~\mathcal{N}_G = \{0,\cdots,N_G\},\)
the set of indices for \(Z,Q\) and \(G\), respectively. Let $\widetilde{\mathcal{N}}$ be the set of 3-tuples of multi-indices defined by $\widetilde{\mathcal{N}} = \mathcal{N}_Z\times \mathcal{N}_Q\times \mathcal{N}_G = \{(i,j,k), i \in \mathcal{N}_Z, j \in \mathcal{N}_Q, k \in \mathcal{N}_G\}.$
A point \(x_m \in \widetilde{\mathcal{X}}\), with \(m = (i,j,k) \in \widetilde{\mathcal{N}}\) is defined by
\(x_m = (z_i,q_j,g_k),\) with \(z_i \in \widetilde{\stZ},q_j \in \widetilde{\mathcal{Q}}, \text{ and } g_k \in \widetilde{\mathcal{G}}\).\\
\(\mathcal{X}\) is then converted into a finite-state MDP with state space \(\widetilde{\mathcal{X}}\), inheriting the Markov property from \(\mathcal{X}\), and becomes a discrete-state Markov process. \\ Let us denote by $\mathcal{N}_n = \{0,\cdots,N\}$ the set of time indices for discrete time points \(t_0,\cdots,t_{N}\). For \((n,x_m)\in \mathcal{N}_n\times\widetilde{\mathcal{X}}\), we define the approximation of the value function and the decision rule at a point \(x_m =(z_i,q_j,g_k) \) at time \(n \in \mathcal{N}_n\) by
\[V^D(n,x_m) \simeq V(t_n,z_i,q_j,g_k) \text{ and } \alpha_n = \widetilde{\alpha} (n,x_m) \simeq \widetilde{\alpha}(t_n,z_i,q_j,g_k),\]
respectively. \\ 
We define \(X^{\alpha,D}_{n} = (Z_{n}^D, Q_{n}^D, G_{n}^D) \in \widetilde{\mathcal{X}}\) as the discrete-state process at time \(n\). Assuming that at time \(n\in \mathcal{N}_n\) the state is at the grid point \(x_{m_1} \in \tilde{\mathcal{X}}\) and that action \(a \in \mathcal{U}\) is performed, the state moves to the grid point \(x_{m_2} \in \tilde{\mathcal{X}}\) at time \(n + 1\) with some probability \(\mathcal{P}_{x_{m_1},x_{m_2}}^a\). This probability is called transition probability, which is the probability that the state moves from \(x_{m_1} \) at time \(n\) to \(x_{m_2}\) at time \(n + 1\) under the action $a$, and it is defined by \[ \mathcal{P}_{x_{m_1},x_{m_2}}^a= \mathbb{P}(X_{n+1}^{\alpha,D}= x_{m_2} ~|~ X_n^{\alpha,D} = x_{m_1}, \alpha_n = a)=\mathbb{P}(T_n(x_{m_1},a,\mathcal{E}_{n+1})= x_{m_2}).\]

Then the expectation of the value function present in algorithm \ref{bellman_backward_algo} is approximated by
\begin{align*}
    \mathbb{E}[V^D(n+1,T_n(x_{m_1},a,\mathcal{E}_{n+1})] 
    &=\sum_{x_{m_2}\in \widetilde{\mathcal{X}}}\mathcal{P}^a_{x_{m_1},x_{m_2}} V^D(n+1,x_{m_2}).
\end{align*}
\paragraph{Transition Probabilities} 
 We recall that given the state \(X_n\) and the decision rule \(\alpha_n\) at time \(n\), the deseasonalized residual demand \(Z_{n+1}\), the battery's state of charge \(Q_{n+1}\), and the fuel tank level \(G_{n+1}\) are Gaussian random variables. In {\revi addition}, (\(G_{n+1}, Z_{n+1}\)) for \(a \in \{u^F,u^{FL}\}\) and (\(Q_{n+1}, Z_{n+1}\)) for \(a \in \{u^C,u^D,u^{DL}\}\) are bivariate Gaussian. In order to practically compute the transition probabilities, we define the neighborhood of \(z_{i}\), \(q_j\), and \(g_k\) as follows: \\ 
 For the inner grid points:
\begin{align*}
	S_{z_i} &= \left(z_i -\frac{1}{2} \Delta_{z-1} , z_i +\frac{1}{2} \Delta_z\right] ~~= \left(\frac{1}{2}(z_i + z_{i-1}) , \frac{1}{2} (z_i + z_{i+1})\right], \qquad ~i = 1, \ldots, N_Z - 1,\\
	S_{q_j}&= \left(q_j -\frac{1}{2} \Delta_{q-1} , q_j +\frac{1}{2} \Delta_q\right] = \left(\frac{1}{2}(q_j + q_{j-1}) , \frac{1}{2} (q_j + q_{j+1})\right], \quad j = 1, \ldots, N_Q - 1,\\
	S_{g_k}&= \left(g_k -\frac{1}{2} \Delta_{g-1} , g_k +\frac{1}{2} \Delta_g\right] = \left(\frac{1}{2}(g_k + g_{k-1}) , \frac{1}{2} (g_k + g_{k+1})\right], \quad k = 1, \ldots, N_G - 1.
\end{align*}
Similarly, we define the neighborhood of \(z_0\) and \(z^{}_{N_Z}\) for the boundary grid points as follows:
\begin{align*}
	S_{z_0} &= \left(-\infty, z_0 +\frac{1}{2}\Delta_z\right] \qquad ~= \left(-\infty, \frac{1}{2} (z_0 + z_{1})\right],\\
	S_{z^{}_{N_Z}} &= \left(z^{}_{N_Z}-\frac{1}{2}\Delta_{z^{}_{N_Z}-1},+\infty\right)= \left(\frac{1}{2}(z^{}_{N_Z}+z^{}_{N_Z-1}),+\infty\right).
    \end{align*}
    For \(n \in \mathcal{N}_n\), the following relations between the discrete-state process and the continuous-state process hold true:
\begin{align*}
 Z^D_n &= z_i \longleftrightarrow Z_n \in S_{z_i},& \forall i = 0, \ldots, N_Z ,\\
      Q^D_n &= q_j \longleftrightarrow Q_n \in S_{q_j},&\forall  j = 0, \ldots, N_Q ,\\
      G^D_n &= g_k \longleftrightarrow G_n \in S_{g_k}, & \forall  k = 0, \ldots, N_G .
\end{align*}
According to Assumption \ref{ass}, it is not possible to operate the battery and the generator simultaneously. Therefore, we will consider three cases: first, the battery is being used; second, the generator is being operated; and third, neither the battery nor the generator is being operated. All the computations below are for the inner points \(z_{i_2},q_{j_2},g_{k_2}\), where \(i_2 = 1,\cdots,N_Z-1, ~j_2 = 1,\cdots,N_Q-1, \text{ and } k_2 = 1,\cdots,N_G-1.\)

\begin{itemize}

\item \textbf{Case 1: Operating the battery with the generator in idle mode.}
\par 
Recall that from Remark~\ref{Rem-Tkernel}, the battery and the generator cannot operate simultaneously.  Then, for \(\alpha_n=a \in \{u^C,u^D,u^{DL}\}\), \(Q_{n+1}\) and \(Z_{n+1}\) are correlated and \(G_{n+1}=G_n\) is independent of (\(Q_{n+1},Z_{n+1}\)).
	Hence, given two points \(x_{m_1} = (z_{i_1},q_{j_1},g_{k_1}) \in \widetilde{\mathcal{X}}\)  and \(x_{m_2} = (z_{i_2},q_{j_2},g_{k_2}) \in \widetilde{\mathcal{X}}\), the transition probability that the state moves from \(x_{m_1}\) at time \(n\) to \(x_{m_2}\) at time \(n+1\) under the action \(a \in \{u^C,u^D,u^{DL}\}\) is given by
\begin{align*}
&\mathcal{P}^a_{x_{m_1},x_{m_2}}= \mathbb{P}(\mathcal{T}_n(x_{m_1},a,\mathcal{E}_{n+1})=x_{m_2})
    = \mathbb{P}(\mathcal{T}_n(x_{m_1},a,\mathcal{E}_{n+1})= (z_{i_2},q_{j_2},g_{k_2}))\\
    &=  \mathbb{P}(( \mathcal{T}_n^Z(x_{m_1},a,\mathcal{E}_{n+1}),\mathcal{T}_n^Q(x_{m_1},a,\mathcal{E}_{n+1}))\in S_{z_{i_2}}\times S_{q_{j_2}}  )\times \mathbb{P}(\mathcal{T}_n^G(x_{m_1},a,\mathcal{E}_{n+1})\in S_{g_{k_2}}). 
\end{align*}
According to Proposition~\eqref{prop-G}, the generator is in idle mode, this implies that \(G_{n+1}\) is Dirac, therefore the probability  that at time \(n+1\) the state \(G_{n+1}\) is in the neighborhood \(S_{g_k}\) of \(g_k\) given that at time \(n\), \(X_n^{\alpha,D} = x_{m_1}\) and the action \(\alpha_n=a\) is taken, is given by
\[ \mathbb{P}(\mathcal{T}_n^G(x_{m_1},a,\mathcal{E}_{n+1})\in S_{g_{k_2}}) = 
\begin{cases}
	1&\text{ if } \mathcal{T}_n^G(x_{m_1},a,\mathcal{E}_{n+1})\in S_{g_{k_2}},\\
	0 &\text{otherwise}.
\end{cases}\]
The conditional probability that at time \(t_{n+1}\) the pair \((Q_{n+1},Z_{n+1}) \in S_{z_{i_2}}\times S_{q_{j_2}}\) given the state process \(X_n^{\alpha,D} = (z_{i_1},q_{j_1},g_{k_1})\) and the action \(\alpha_n = a\) is given by
\begin{align*}
	\mathbb{P}(( \mathcal{T}_n^Z(x_{m_1},a,\mathcal{E}_{n+1}),\mathcal{T}_n^Q(x_{m_1},a,\mathcal{E}_{n+1}))\in S_{z_{i_2}}\times S_{q_{j_2}} )
 = \int_{S_{z_{i_2}}}\int_{S_{q_{j_2}}}f_{ZQ}(z,q)\,dq\,dz,
\end{align*}
where \(f_{ZQ}(z,q)\) is the {\revi joint} density of the {\revi bivariate} normal distribution.

	\item \textbf{Case 2: Operating the generator with a battery in idle mode.}

\par When the generator operates, the battery cannot be used. Therefore, for \(\alpha_n = a \in \{u^F,u^{FL}\}\), \(G_{n+1}\) and \(Z_{n+1}\) are correlated, and \(Q_{n+1}\) is a deterministic function, which is independent of \((G_{n+1},Z_{n+1})\). \\
Now, given two points \(x_{m_1} = (z_{i_1},q_{j_1},g_{k_1}) \in \widetilde{\mathcal{X}}\) and \(x_{m_2} = (z_{i_2},q_{j_2},g_{k_2}) \in \widetilde{\mathcal{X}}\), the transition probability that the state moves from \(x_{m_1}\) at time \(n\) to \(x_{m_2}\) at time \(n+1\) under action \( a\in \{ u^F,u^{FL}\}\) is given by:
\begin{align*}
	&\mathcal{P}^a_{x_{m_1},x_{m_2}} 
= \mathbb{P}(\mathcal{T}_n(x_{m_1},a,\mathcal{E}_{n+1})=x_{m_2})
    = \mathbb{P}(\mathcal{T}_n(x_{m_1},a,\mathcal{E}_{n+1})= (z_{i_2},q_{j_2},g_{k_2}))\\
    &=  \mathbb{P}(( \mathcal{T}_n^Z(x_{m_1},a,\mathcal{E}_{n+1}),\mathcal{T}_n^G(x_{m_1},a,\mathcal{E}_{n+1}))\in S_{z_{i_2}}\times S_{g_{k_2}}  )\times \mathbb{P}(\mathcal{T}_n^Q(x_{m_1},a,\mathcal{E}_{n+1})\in S_{q_{j_2}}). 
	\end{align*}
 	According to Proposition \eqref{prop-Q}, when the battery is in idle mode, the conditional variance of \(Q_{n+1}\) is zero; this implies that \(Q_{n+1}\) is degenerated (Dirac). Therefore, the probability that at time \(n+1\) the state \(Q_{n+1}\) is in the neighborhood \(S_{q_{j_2}}\) of \(q_{j_2}\) given that at time \(n\), \(X_n^{\alpha,D} = x_{m_1}\) and the action \(\alpha_n=a\) is taken is given by
	\[\mathbb{P}(\mathcal{T}_n^Q(x_{m_1},a,\mathcal{E}_{n+1})\in S_{q_{j_2}}) = 
\begin{cases}
	1&\text{ if } \mathcal{T}_n^Q(x_{m_1},a,\mathcal{E}_{n+1})\in S_{q_{j_2}},\\
	0 &\text{otherwise}.
\end{cases}\]
	The conditional probability that \(({\revi Z_{n+1},G_{n+1}}) \in {\revi S_{z_{i_2}} \times S_{g_{k_2}}}\) given the state process \(X_n^{\alpha,D} = (z_{i_1},q_{j_1},g_{k_1})\) and the action \(\alpha_n = a\) is
\begin{align}
	\mathbb{P}( \mathcal{T}_n^Z(x_{m_1},a,\mathcal{E}_{n+1}),\mathcal{T}_n^G(x_{m_1},a,\mathcal{E}_{n+1}))\in S_{z_{i_2}}\times S_{g_{k_2}} )
	 = \int_{S_{z_{i_2}}}\int_{S_{g_{k_2}}} f_{ZG}(z,g )dgdz,\\
     \label{cond_proba}
\end{align}
where \(f_{ZG}(z,g)\) is the conditional density function of \(Z_{n+1}\) and \(G_{n+1}\). The double integral appearing in equation ~\eqref{cond_proba} can be computed using {\revi MATLAB} built-in functions or can further be simplified to a single integral, see \cite{takam_PhD_2023}.

\item \textbf{Case 3: Battery and generator in idle mode.}
\par When the battery and the generator are both in idle mode, the variables \(Z_{n+1}\), \(Q_{n+1}\), and \(G_{n+1}\) are independent of each other. In addition, \(Q_{n+1}\) and \(G_{n+1}\) are deterministic functions. Therefore, given two points \(x_{m_1} = (z_{i_1},q_{j_1},g_{k_1}) \in \widetilde{\mathcal{X}}\) and \(x_{m_2} = (z_{i_2},q_{j_2},g_{k_2}) \in \widetilde{\mathcal{X}}\), the transition probability that the state moves from \(x_{m_1}\) at time \(n\) to \(x_{m_2}\) at time \(n+1\) under action \(a \in \{u^W,u^O\}\) is just the product of the probabilities, that is,
	\begin{align*}
	&\mathcal{P}^a_{x_{m_1},x_{m_2}} 
= \mathbb{P}(\mathcal{T}_n(x_{m_1},a,\mathcal{E}_{n+1})=x_{m_2})
    = \mathbb{P}(\mathcal{T}_n(x_{m_1},a,\mathcal{E}_{n+1})= (z_{i_2},q_{j_2},g_{k_2}))\\
    &=  \mathbb{P}(\mathcal{T}_n^Z(x_{m_1},a,\mathcal{E}_{n+1})\in S_{z_{i_2}}) \times \mathbb{P}(\mathcal{T}_n^Q(x_{m_1},a,\mathcal{E}_{n+1})\in  S_{q_{j_2}}  )\times \mathbb{P}(\mathcal{T}_n^G(x_{m_1},a,\mathcal{E}_{n+1})\in S_{g_{k_2}}). 
	\end{align*}
 \end{itemize}
Here, \(Q_{n+1}\) and \(G_{n+1}\) are degenerate (Dirac), so 
the above transition probability
can be written as follows
\begin{align*}
	\mathbb{P}^a_{x_{m_1},x_{m_2}} =\begin{cases}
		 \mathbb{P}( \mathcal{T}_n^Z(x_{m_1},a,\mathcal{E}_{n+1})\in S_{z_{i_2}}),& \text{ if }~~ \mathcal{T}_n^Q(x_{m_1},a,\mathcal{E}_{n+1}) \in S_{q_{j_2}} \text{ and }~~G_{n+1}\in S_{g_{k_2}},\\
		 0, &\text{otherwise}.
	\end{cases}
\end{align*}
\begin{remark}
	Similar reasoning can be adopted for the computations of the transition probabilities for the
	boundary points.
\end{remark}

\section{Numerical Results}
\label{Num-rsult}
\subsection{Experimental Setting}
 To gain a concrete understanding of the optimal control and the behavior of the value function for microgrid operation, we now turn to the numerical solution of the optimal control problem. This section presents a visualization of the optimal decision rules over a time horizon of \(T=7\) days and the corresponding numerically approximated value function computed using the backward recursion approach, presented in Algorithm \ref{bellman_backward_algo}.
\\
  For numerical simulations, we consider the residual demand denoted by $R(t)= \mu_R(t) +Z(t)$, where the seasonality function \(\seasonality(t) = \mu_0^R + \kappa_1^{R} \cos\left(\frac{2\pi(t-t_1^R)}{\delta_1}\right) +  \kappa_2^{R} \cos\left(\frac{2\pi(t-t_2^R)}{\delta_2}\right) \) has parameters $t_1^R =t_2^R=0,~ \delta_1 = 365 \text{ days},~ {\revi \delta_2} = 1 \text{ day},~\mu_0^R = 0.1$, \(\kappa_1^{R} =0.1,\) and \( \kappa_2^{R} =1\). The deseasonalized residual demand $Z$ is modeled as a mean-reverting to zero OU process with a mean reversion speed $\beta_R = 0.2$ and a volatility of $\sigma_R = 0.45$. This results in a practical operating range for the deseasonalized residual demand within the interval $[ \underline{z}, \overline{z}]=[-2.13, 2.13]$, chosen using the 3-sigma rule, that is, $\underline{z}=-\frac{3\sigma_{R}}{\sqrt{2\beta_R}}=-2.13$ and $\overline{z}=\frac{3\sigma_{R}}{\sqrt{2\beta_R}}=2.13$. Then, the operating range of the residual demand is chosen using the worst-case scenario as follows: $\underline{r}=\displaystyle\min_{t \in [0,T]}\mu_{R}(t)+\underline{z}=-3$ and $\overline{r}=\displaystyle\max_{t \in [0,T]}\mu_{R}(t)+\overline{z}=3$. The finite time interval $[0,T]$ is subdivided into $N = 168$ sub-intervals of length \(\Delta_N = 1\) hours. In the direction of r, the state is discretized such that the critical value 0 and the threshold $R_{Q0}$ (resp. $R_{G0}$) are not grid points but midpoints of two consecutive grid points. The case of the critical value 0 is achieved by choosing $N_r=N_Z$ odd and $\Delta_r=6/N_r$. The threshold $R_{Q0}$ (resp. $R_{G0}$) above which the economical mode can be activated is chosen as the midpoint of two consecutive grid points around $\frac{\overline{r}}{2}$. The continuous state space \(\Statespace\) is discretized with sub-intervals $N_Z=N_R= 17$, $N_Q = {\revi 15}$, and $N_G = 10$. Hence, for $\overline{r}=3$ and $N_R=17$, the threshold is given by $R_{Q0}=R_{G0}=1.4118$. The battery's capacity is calibrated so that a full battery can meet demand all night and an empty battery can store excess production all day; the details are in Appendix~\ref{Calib_batC}. $C_Q = 18$ kWh is the proper capacity given the residual demand parameters mentioned above.
 Furthermore, we presume that the battery self-discharges using a parameter calibrated so that, in idle mode, a full battery's state of charge drops by 2\% in 4 days; Appendix~\ref{Calib_self} provides details. Consequently, the self-discharge rate is determined to be $\eta_0 = 2.1044 \times 10^{-4}$. We assign a penalty to the remaining unsatisfied demand, called the discomfort cost. To emphasize the severity of the discomfort cost, we model it using a pure quadratic function, $f_p(x)=k_0x^2$, with $k_0=0.575~€/(kW)^2h$.\\ The fuel tank capacity is chosen so that th
e generator can meet the average demand throughout the simulation period. According to the calibration of the fuel tank parameters given in Appendix \ref{calib_fuelC},  $C_G=20$ liters is the appropriate fuel tank capacity for a  simulation period of seven days.
The average price of diesel in Germany is used to parameterize the fuel price at the time of analysis \(F_0 = 1.5\,€/\ell\).  At the terminal time, we charge a penalty price $\gamma_{pen}^Q = 0.8 ~€/kW$ for every kilowatt of electricity needed to compensate for the deficit when the battery level is below the reference level $q_{ref} = 80\%$, and there is no reward for the surplus, that is, $\gamma_{liq}^Q =0$. Finally, a liquidation price of $\gamma_{liq}^G = 1.25$ €/$\ell$ is fixed for the remaining fuel in the tank at the terminal time. The parameters are summarized in table~\ref{Tab_paramet}.
Given this parameter, the running cost as a function of the residual demand $r$ is illustrated in Fig. ~\ref{fig_cost}. Here, the solid blue line represents the cost to operate the generator in the full mode, given by $\Psi_{F}(r)=F_0(c_0 + c_1 r)$, while the black dotted line represents the cost to operate the generator in the limited mode, given by $\Psi_{FL}(r)=F_0(c_0 + c_1 R_{G0})+k_0 (r-R_{G0})^2$. The solid purple line represents the cost to operate the battery in full mode, or the degradation cost, given by $\Psi_{D}=\gamma_{deg}r$, while the dotted green line shows the cost to operate the battery in limited mode, given by $\Psi_{DL}=\gamma_{deg}r+ k_0 (r-R_{Q0})^2$. The solid orange line is the cost of discomfort due to unmet demand, defined as $k_0 r^2$.  The red dotted line at $R_{Q0}$ is the threshold above which the limited mode can be activated.
\begin{figure}[htbp!]
    \centering 
    \hspace*{-1.cm}
    \includegraphics[width=0.9\textwidth,height=0.4\linewidth]{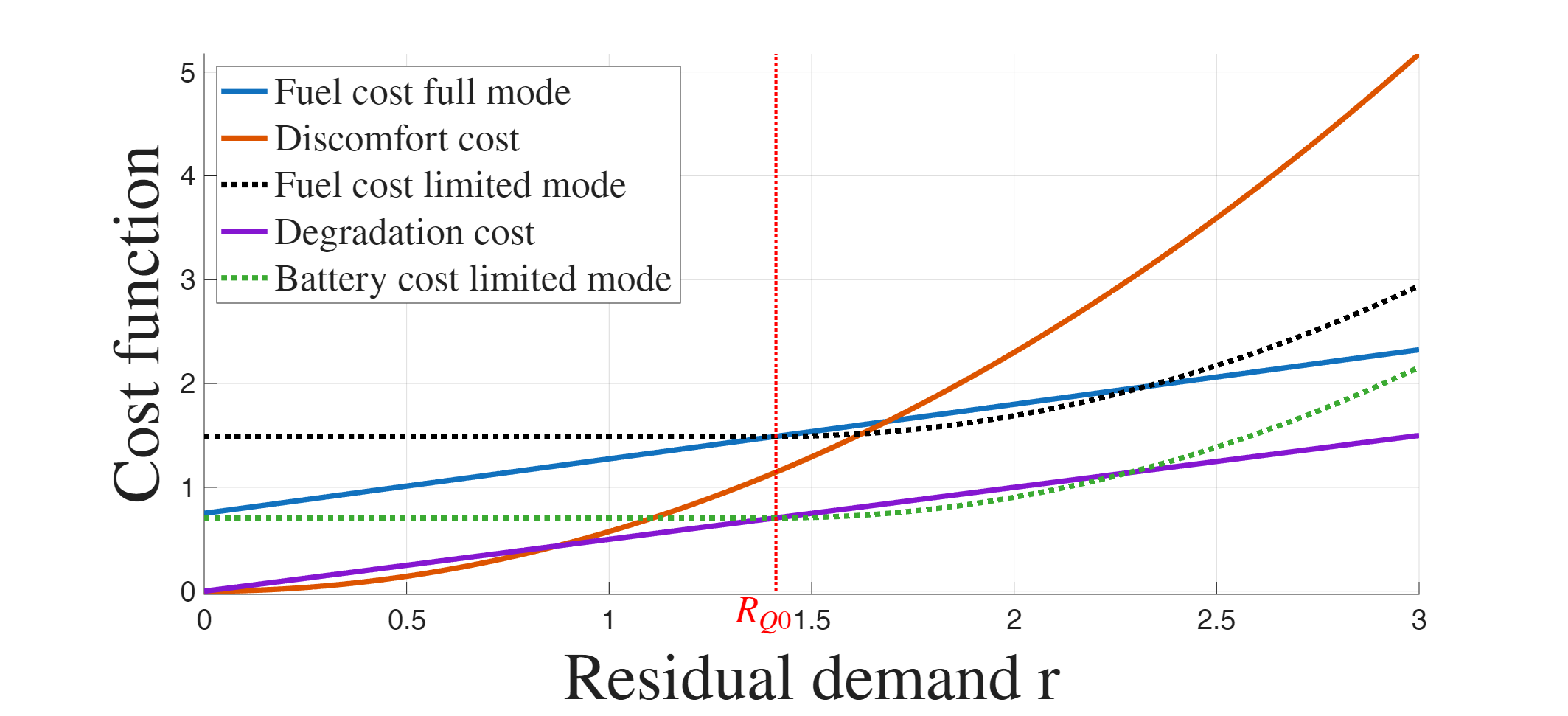}
    \caption{Running cost as a function of a positive residual demand $r>0$.}
 \label{fig_cost}
\end{figure}

\paragraph{Physical metric for unmet demand}
{\revi While the discomfort cost $\Psi_P(t,x,\nu)$ of Subsection~\ref{Stoc-OPtCont} provides a tractable quadratic penalty used to drive the optimization, we additionally evaluate the optimal policy using the \emph{unsatisfied demand}, a physically interpretable metric expressed in kWh. For time step $n=0,\ldots,N-1$, with residual demand $R_n = \mu_{R,n}+Z_n$ and action $\alpha_n$, the energy not served over $[t_n,t_{n+1})$ is given by
\begin{equation}
	\Delta E_n(\alpha_n, R_n) = \begin{cases}
		(R_n)^+ \, \Delta_N, & \text{if } \alpha_n = u^W, \\
		(R_n - R_{Q0})^+ \, \Delta_N, & \text{if } \alpha_n = u^{DL}, \\
		(R_n - R_{G0})^+ \, \Delta_N, & \text{if } \alpha_n = u^{FL}, \\
		0, & \text{otherwise}.
	\end{cases}
	\label{ENS_step}
\end{equation}
The total expected unsatisfied energy over the planning horizon under the optimal decision rule $\alpha^*$ (or under any other strategy, with its own actions $\alpha_n$) is then
\begin{equation}
	E_u = \mathbb{E}\left[\sum_{n=0}^{N-1} \Delta E_n(\alpha_n^*, R_n)\right] \quad [kWh],
	\label{EENS}
\end{equation}
and the corresponding percentage of total demand left unserved over the horizon is
\begin{equation}
	E_{u}{\%} = \frac{E_u}{\mathbb{E}\left[\sum_{n=0}^{N-1} (R_n)^+ \Delta_N\right]} \times 100.
	\label{EENS_perc}
\end{equation}
These two metrics complement the discomfort-cost values already reported and are used below to quantify the practical reliability of the optimal management strategy under each simulated weather scenario.}

\begin{table}[h!]
	\centering	
	\begin{tabular}[T]{|lc|rrr|} \hline
		Parameters&&&Values& Units\\
		\hline
		\textbf{Discretization} &&& & \\
        
		\hspace*{1em} Time horizon, sub-intervals & $T, N$ && $7$ days, $168$ & \\
		\hspace*{1em} Time step&  $\Delta_N$ && $1$&$ ~h$\\	
		\hspace*{1em}	Number of grid points in $r$- direction &$N_R$&& $17$ &\\ 
		\hspace*{1em}	\qquad\qquad\qquad\qquad\qquad $q$- direction &$N_Q$&& $15$ &\\ 
		\hspace*{1em}	\qquad\qquad\qquad\qquad\qquad$g$- direction &$N_{G}$&& $10$ &\\ 
		\hspace*{1em}	Discretization step size in  $r$- direction &$\Delta_{R}$&& $0.353$ &$kW$\\ 
		\hspace*{1em}	\qquad\qquad\qquad\qquad\qquad ~ $q$- direction &$\Delta_{Q}$&& $0.067$ &\\
        \hspace*{1em}	\qquad\qquad\qquad\qquad\qquad  ~~$g$- direction &$\Delta_{G}$&& $0.1$ &\\
	\hspace*{1em}	Maximum residual demand  & $\overline{r}$ &&$3$ & $kW$\\
        \hspace*{1em}	Minimum residual demand  & $\underline{r}$ &&$-3$ & $kW$\\
       \hline		
		\textbf{Residual demand} & && &\\
		\hspace*{1em}	Volatility coefficient & $\sigma_R$  &&$0.45$ & $kW/\sqrt{h}$\\
		\hspace*{1em}	Mean reversion speed & $\beta_R$  &&$0.2$ & $1/h$\\
        \hspace*{1em}	Long term mean & $\mu_0^R$  &&$0.1 $ & $kW$ \\
	\hspace*{1em}	Amplitudes of the yearly and daily seasonality& $\kappa_1,\kappa_2$ &&$0.1, ~1$ & $kW/h$ \\
		\textbf{Battery} & && &\\
		\hspace*{1em}  Capacity  & $C_Q$  &&$18$ & $kWh$ \\
		\hspace*{1em} Self-discharging rate  & $\eta_0$  &&$    2.1044\times 10^{-4}
$ & $1/h$ \\
		\hspace*{1em}	Threshold limited mode  & $R_{Q0}$  &&$1.4118$ & $kW$\\
		\hspace*{1em} Threshold terminal time & $q_{ref}$  &&$0.8$ & \\
		\textbf{Fuel level} & && &\\
		\hspace*{1em}   Capacity & $C_G$  &&$20$ & $\ell$\\
		\hspace*{1em} Threshold limited mode & $R_{G0}$  &&$1.4118$ & $kW$ \\
		\hspace*{1em}	Fuel consumption in idle mode& $c_0$  &&$0.5$ & $\ell/h$ \\
		\hspace*{1em}	Load-dependent fuel consumption & $c_1$  &&$0.35$ & $\ell/kWh$ \\ 
         \hline		
		\textbf{Cost parameters} & && &\\
			\hspace*{1em}  Battery degradation cost &$\gamma_{deg}^Q$&& $0.05$&$\texteuro/kWh$ \\
		\hspace*{1em} Discomfort cost coefficient&$k_0$&& $0.575$&$\texteuro/(kW)^2h$ \\
			\hspace*{1em} Penalty price for battery  &$\gamma_{pen}^Q$&& $0.8$&$\texteuro/kWh$ \\
	\hspace*{1em} Liquidation price for fuel &$\gamma_{liq}^G$&& $1.25$&$\texteuro/\ell$ \\
		\hspace*{1em}	Fuel price &$F_0$&&$ 1.5$& $\texteuro/\ell$\\
		\hspace*{1em}	Discount rate &$\rho$&&$ 0.03$& \\
		\hline
	\end{tabular}
\\

	\caption{ Parameters of the numerical simulations}	
	\label{Tab_paramet}
\end{table}

\subsection{Optimal Decision Rule and Value Function}
\label{sec: Decision}
Here, we present the value function and the optimal decision rule that includes a penalty for deviating from the reference state of charge (\(q_{ref}\)) at the terminal time,  and there is no reward for the surplus. The control actions are visualized using a color-coded system, where dark blue indicates over-spilling, light blue the battery charging, green for waiting or \emph{doing nothing}, yellow for discharging the battery in limited mode, orange for discharging the  battery in full mode, and light red and dark red indicate firing fuel in limited and full modes, respectively. In the optimal decision rule graph, the black dotted vertical line at 0 and the {\revi blue dotted vertical line at} $R_{Q0}=1.4118$ represent the critical values of residual demand. The red dotted horizontal line at $q_{ref}=0.8$ represents the penalty threshold for the state of charge. {\revi In the limited modes, the battery (generator) supplies only the fixed power $R_{Q0}$ ($R_{G0}$), and the remaining demand is unmet and penalized by the discomfort cost; these modes are used to save energy or fuel for later periods (Assumption~\ref{ass}(2)). In the legends, BFM and BLM denote the battery in full and limited mode, and FM and LM the generator (fuel) in full and limited mode.} In the visualizations, we plot the value function and the optimal decision rules as a function of one variable, while the others remain fixed.  For visualization with respect to the state of charge ($q$), the red dotted, blue dotted, green dashed, and magenta dashed lines represent $r=\overline{r}$, $r=2.29, r=1.23$, and $r=\underline{r}$, respectively. Similarly, for visualization with respect to $r$, the red dashed, blue dotted, and green dashed lines represent a full battery, a battery filled to $20\%$ capacity, and an empty battery, respectively.\\

\paragraph{Terminal value function (midnight of the last day)}
Fig.~\ref{fig_Term} illustrates the value function at terminal time \(N = 168\) as a function of residual demand $r$ and battery state of charge $q$ (left) and  as a function of battery level and fuel tank level $g$ (right).
The left panel shows three states of the fuel tank at terminal time; the top plot shows the case where the fuel tank is empty, the middle plot shows the case where the fuel tank is at 20\%, and the bottom plot shows the case where no fuel is used during the entire time horizon. 
We observe that these surfaces increase linearly with decreasing state of charge $q$, but are shifted downward (below zero) from empty to full fuel tank due to liquidation of the remaining fuel in the tank. The increase in the value function is due to the penalty for unmet demand, that is, when \(Q(N) < q_{\text{ref}}\). However, it remains constant with respect to $r$ or when \(Q(N) \geq  q_{\text{ref}}\) as there is no reward for the surplus. The right panel shows that the value function is above zero when the fuel tank is empty and decreases linearly as the fuel tank or the state of charge decreases.

\begin{figure}[htbp!]
    \centering 
    \hspace*{-1.cm}
    \includegraphics[width=0.49\textwidth,height=0.25\linewidth]{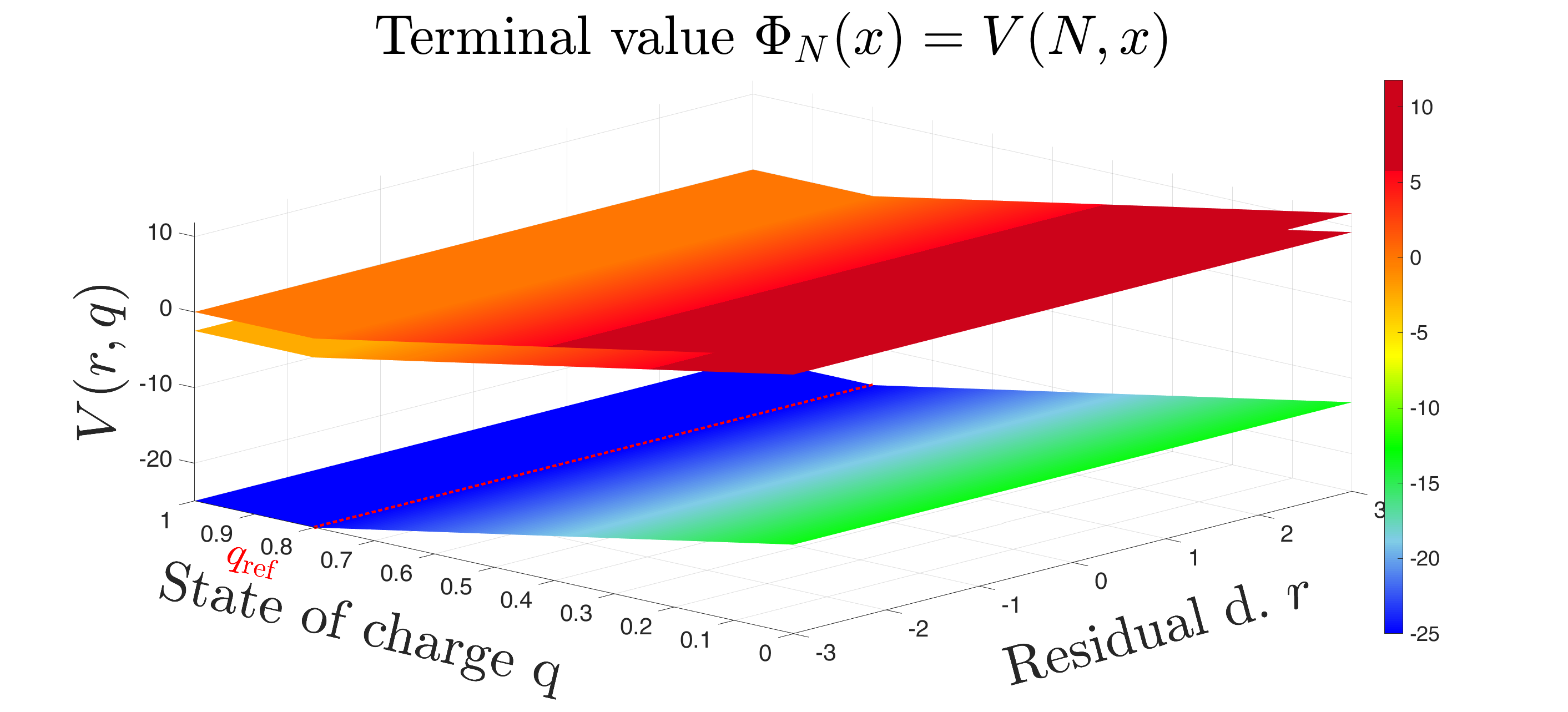}
      \includegraphics[width=0.49\textwidth,height=0.25\linewidth]{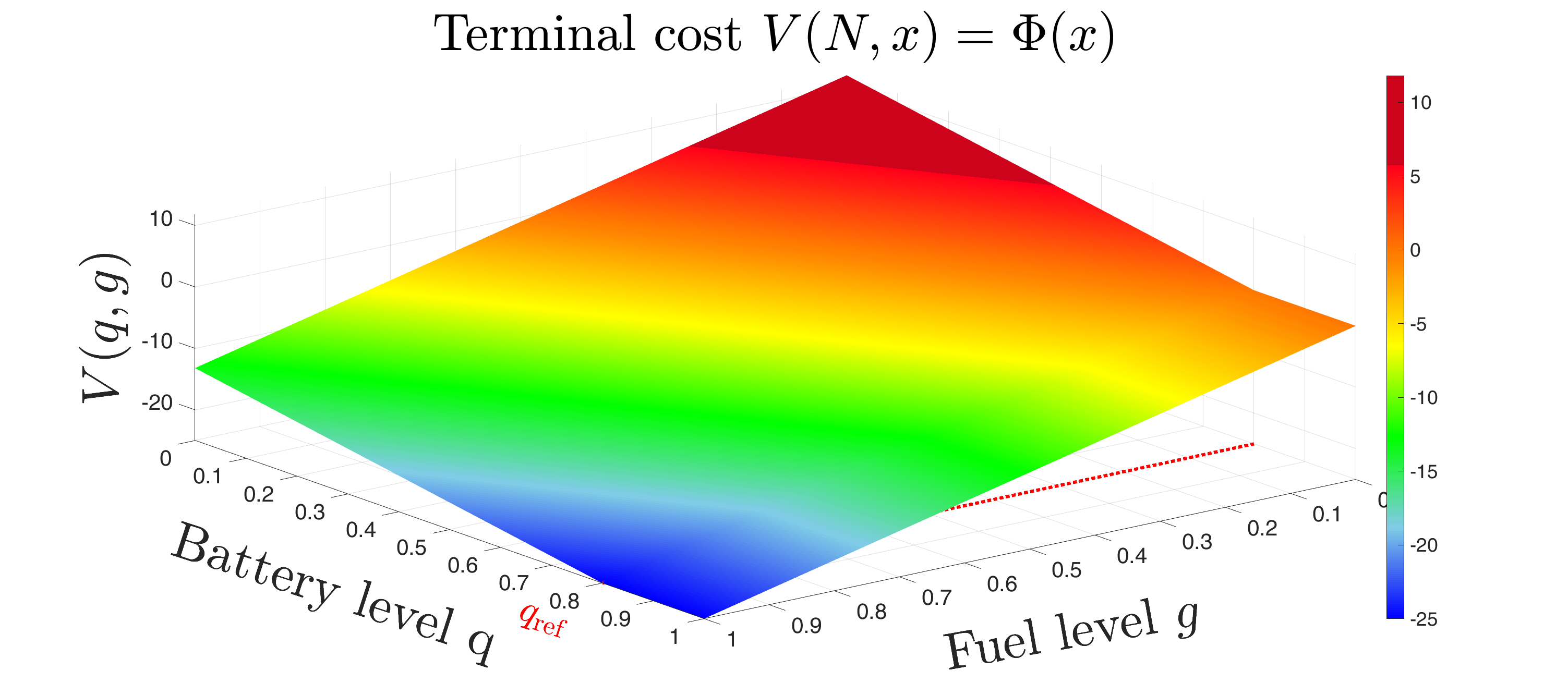}
    \caption{Terminal value function $V(N,x)=\Phi(x)$.\\
    Left:  Value function as a function of $r$ and $q$ for empty fuel tank (top graph),  fuel tank at $20\%$ (middle graph) and full fuel tank (bottom graph).  Right:  Value function as a function of $q$ {\revi and} $g$.}
 \label{fig_Term}
\end{figure}

In the following, Figures~\ref{fig_Val_n-1}, \ref{fig_Val_n-12}, and \ref{fig_Val_0} show the value function and the optimal decision rule in terms of \(r, q)\) one hour before the terminal time ($n=N-1$), at noon on the last day ($n=N-12$), and at the initial time ($n=0$), respectively. The top left graph illustrates the value function for an empty fuel tank, while the middle graph depicts the value function for a tank at 20\% capacity, and the bottom graph shows the value function for a full fuel tank. The optimal decision rule is represented for an empty fuel tank in the lower left panel, for a fuel tank filled to 20\% in the upper right panel, and for a full fuel tank in the lower right panel.  Figures~\ref{fig_VFr-N-1} and \ref{fig_VFr-N-12} show the visualization of the value function (left panel) and the optimal decision rule (right panel) in terms of the state of charge $q$ for a fixed residual demand $r=\{-3, 1.23, 2.29, 3\}$, while Figures~\ref{fig_VFq-N-1} and \ref{fig_VFq-N-12} show the visualization in terms of the residual demand $r$ for a fixed $q=\{0, 0.2, 1\}$ one hour before the terminal time, at noon on the last day, and at the initial time, respectively. 

{\revi \paragraph{Remark} Several actions have almost the same value in large parts of the state space. For the visualization and the simulations, we therefore select, among the decision rules with the smallest regret, the one that is piecewise constant in $q$ with at most three switches. Here, the regret of an action $a$ at $(n,x)$ is the excess of its expected cost, $\Psi(n,x,a)+\mathrm{e}^{-\rho\Delta_N}\,\mathbb{E}[V(n+1,\mathcal{T}_n(x,a,\mathcal{E}_{n+1}))]$, over the value $V(n,x)$ (it is zero for an optimal action), and the regret of a decision rule is the sum of these excesses over the states.}

\section*{Optimal decision rule and value function 1 hour before the terminal time}

\begin{figure}[htbp!]
    \centering 
     \includegraphics[width=0.495\linewidth,height=0.217\linewidth]{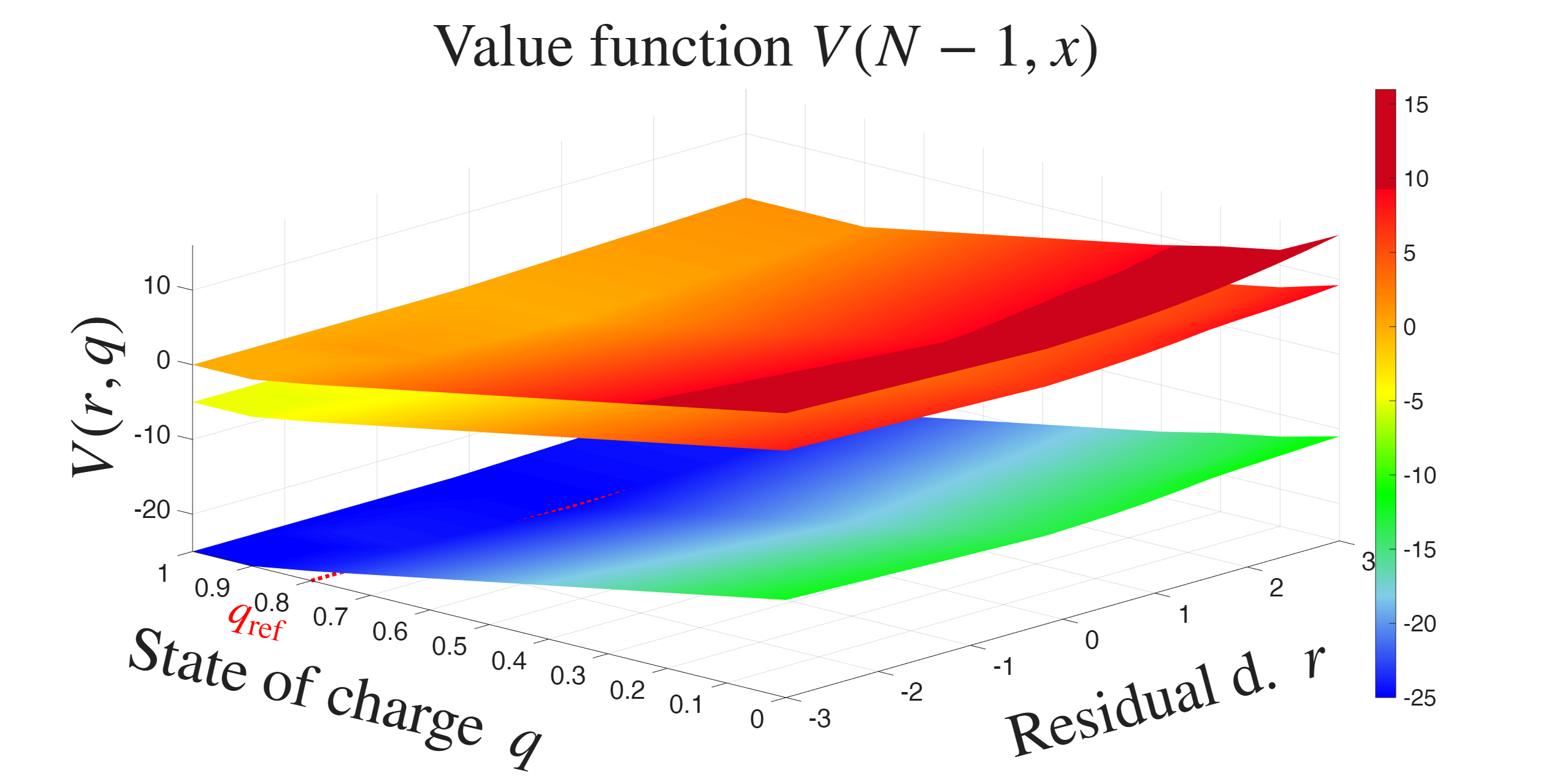} 
     \includegraphics[width=0.495\linewidth,height=0.19\linewidth]{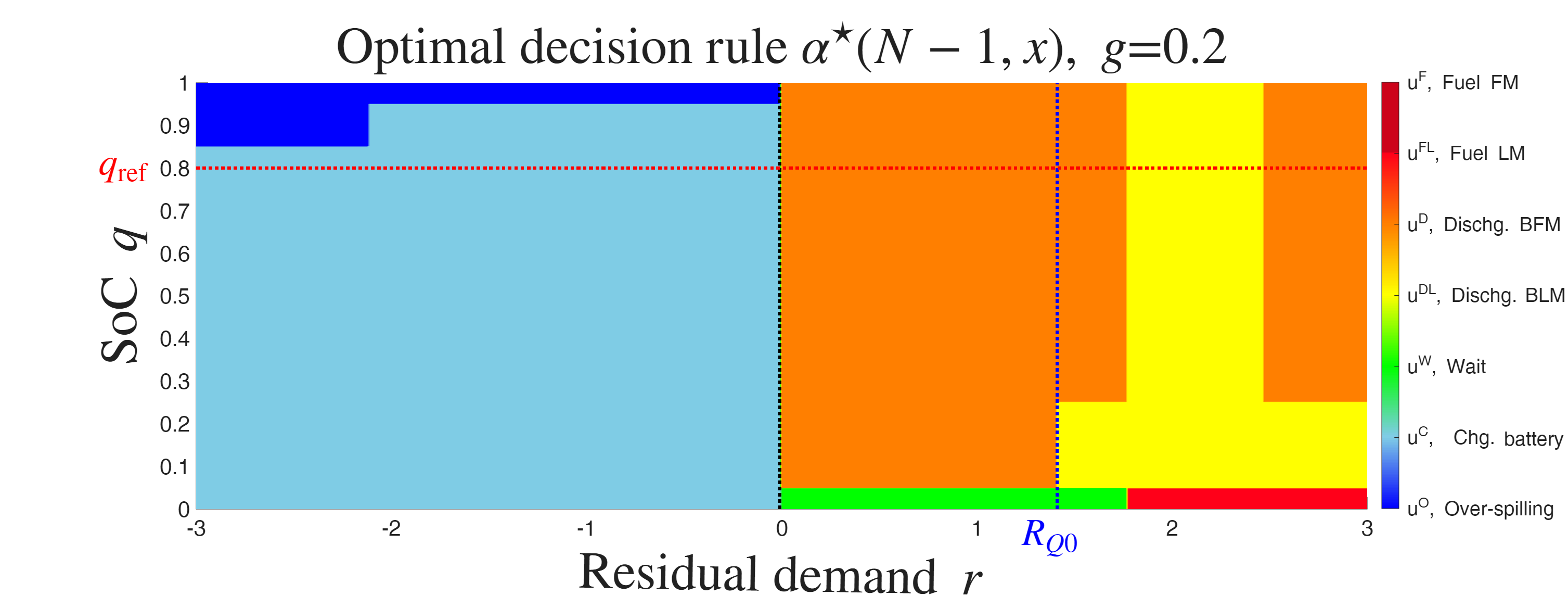}
\vspace*{0.5cm}

\includegraphics[width=0.495\linewidth,height=0.19\linewidth]{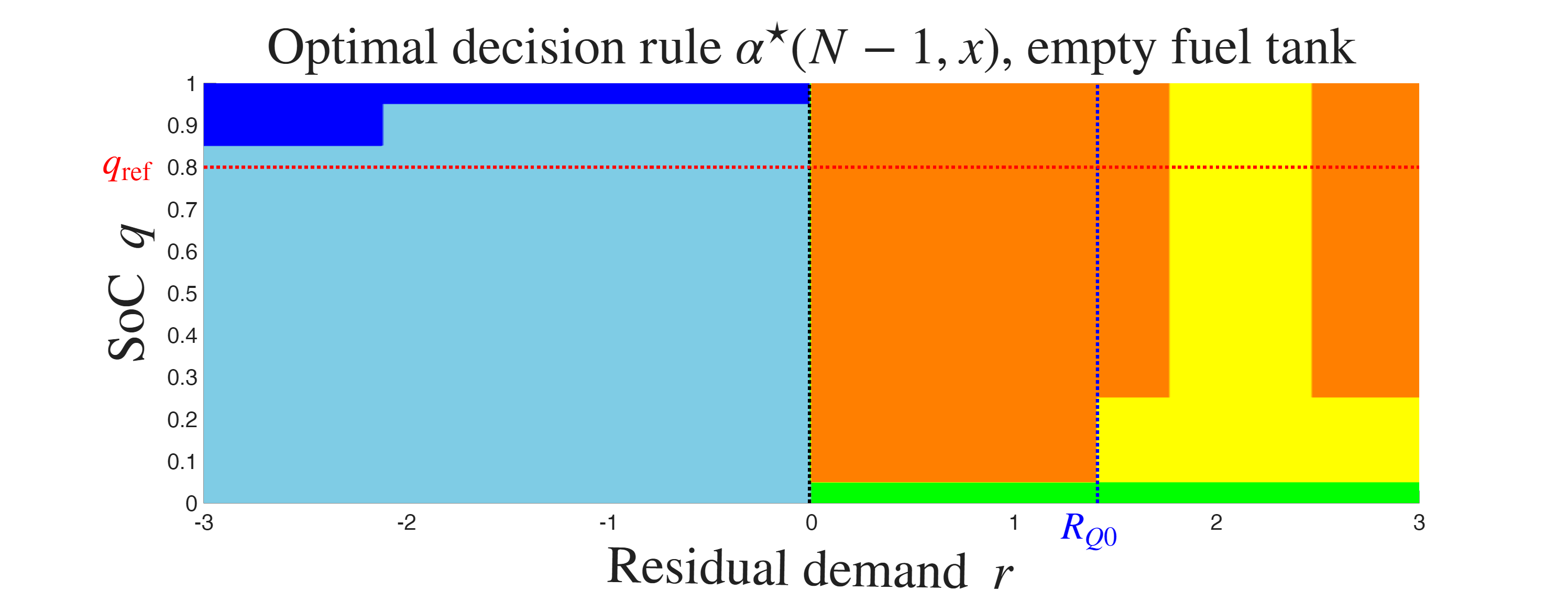}
     \includegraphics[width=0.495\linewidth,height=0.19\linewidth]{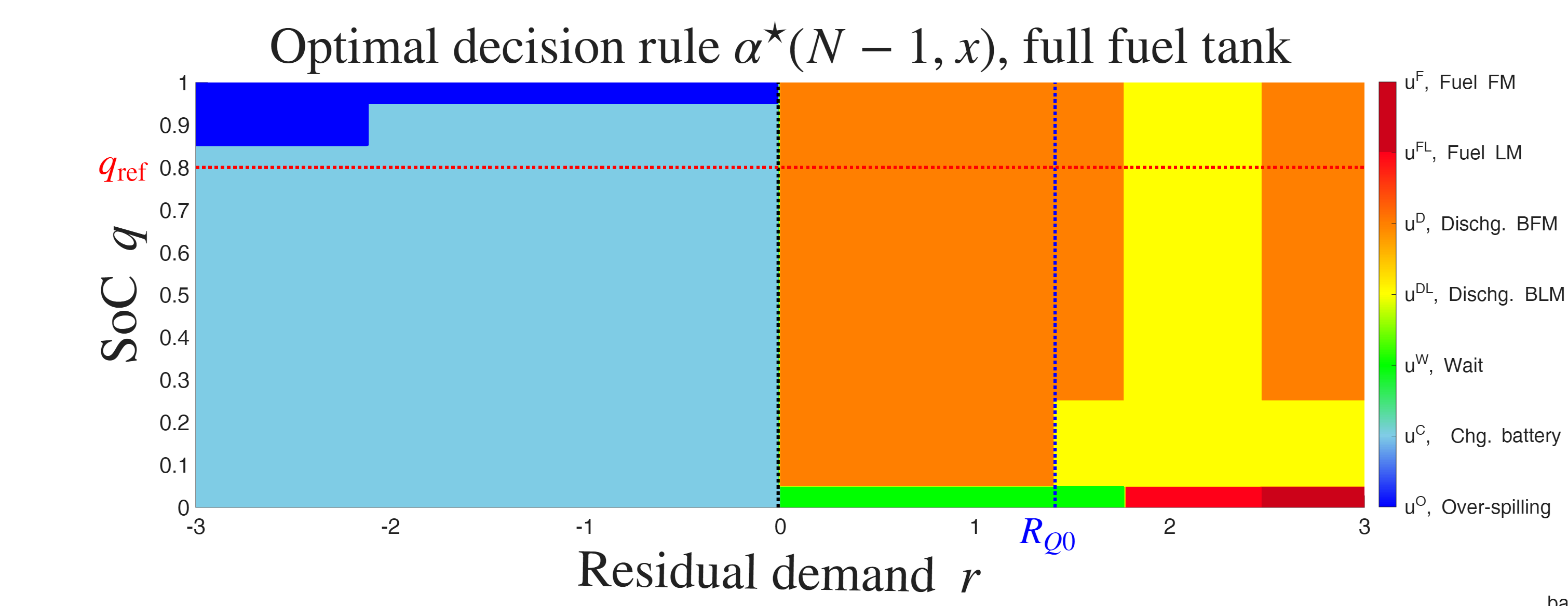}
    \caption{Value functions \(V\) and optimal decision rule \(\alpha^\star\) one hour before the terminal time ($n = N-1$) in terms of $r$ and $q$.
    Upper left panel: Value function for an empty fuel tank (upper graph), a fuel tank filled to 20$\%$ (middle graph), and a full fuel tank (bottom graph). Lower left panel: Optimal decision rule for an empty fuel tank. Lower right panel: Optimal decision rule for a full fuel tank. Upper right panel: Optimal decision rule for a fuel tank filled to 20$\%$ capacity.
    }\label{fig_Val_n-1}
\end{figure}
\begin{figure}[b!]
    \centering   
    \includegraphics[width=0.49\linewidth,height=0.215\linewidth]{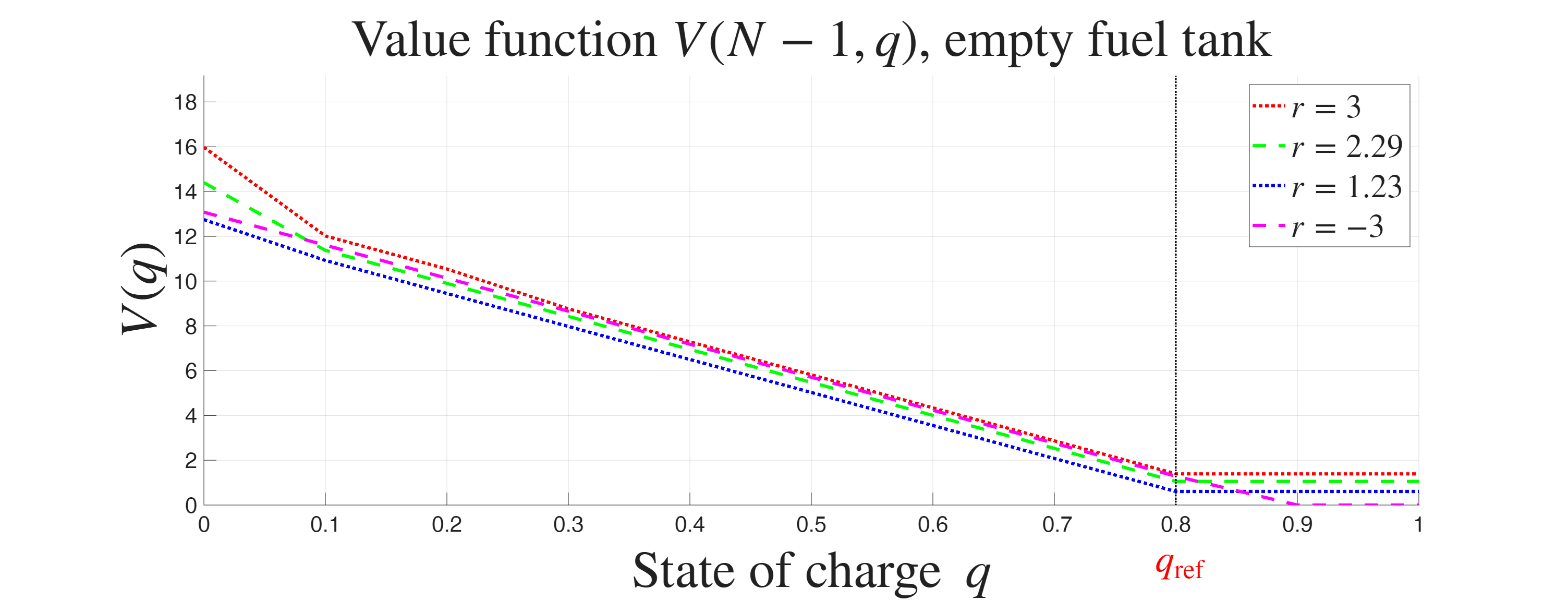}
\includegraphics[width=0.49\linewidth,height=0.215\linewidth]{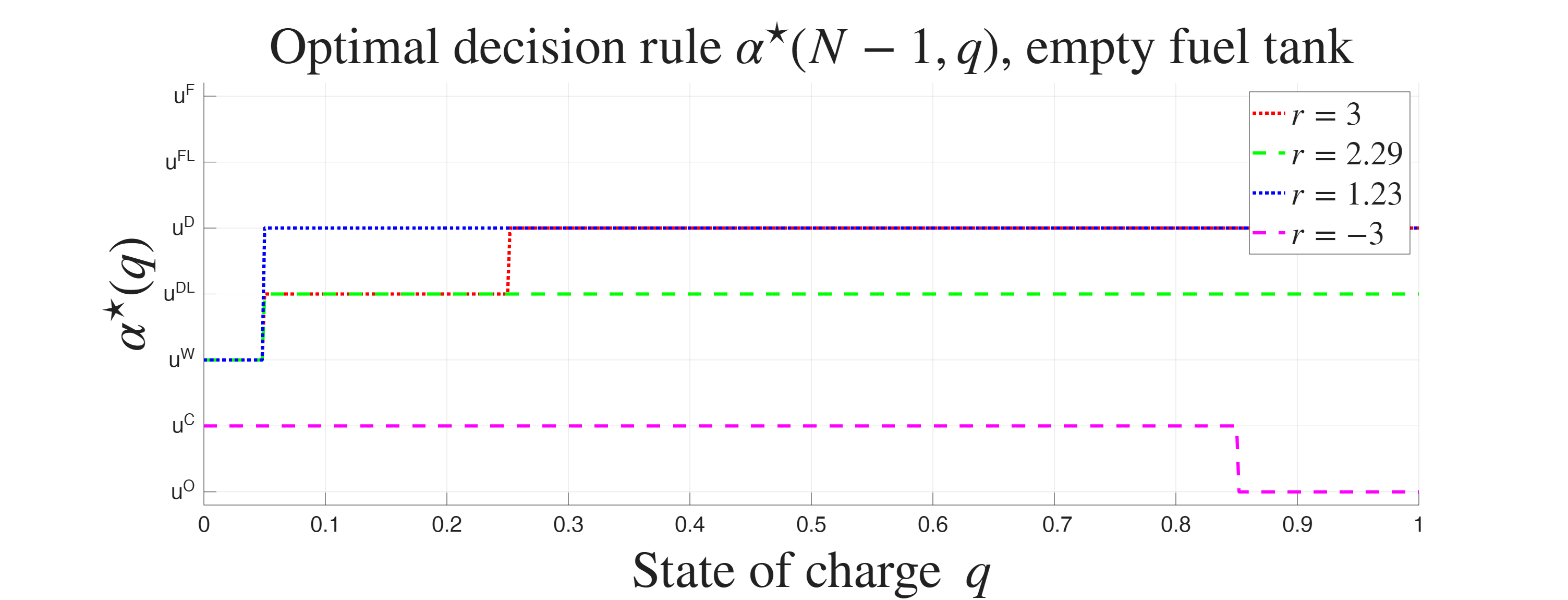}  

\vspace*{0.3cm}
 \includegraphics[width=0.49\linewidth,height=0.215\linewidth]{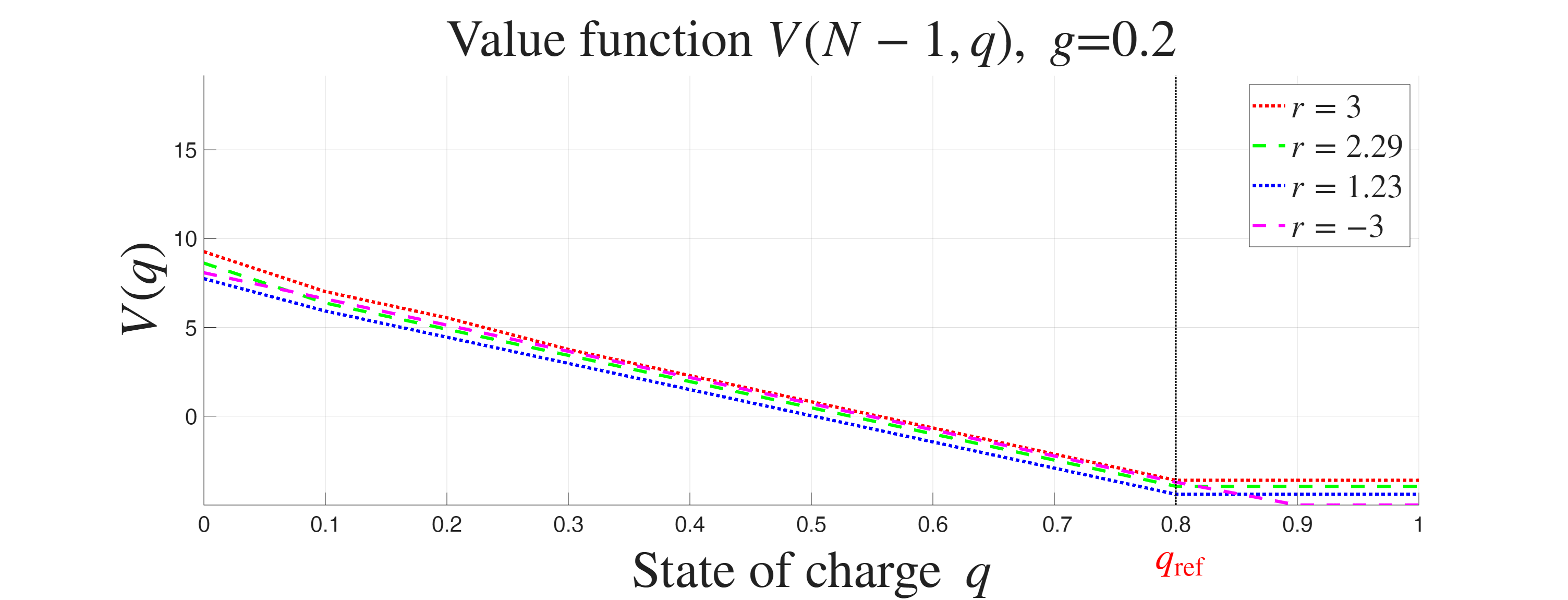}
\includegraphics[width=0.49\linewidth,height=0.215\linewidth]{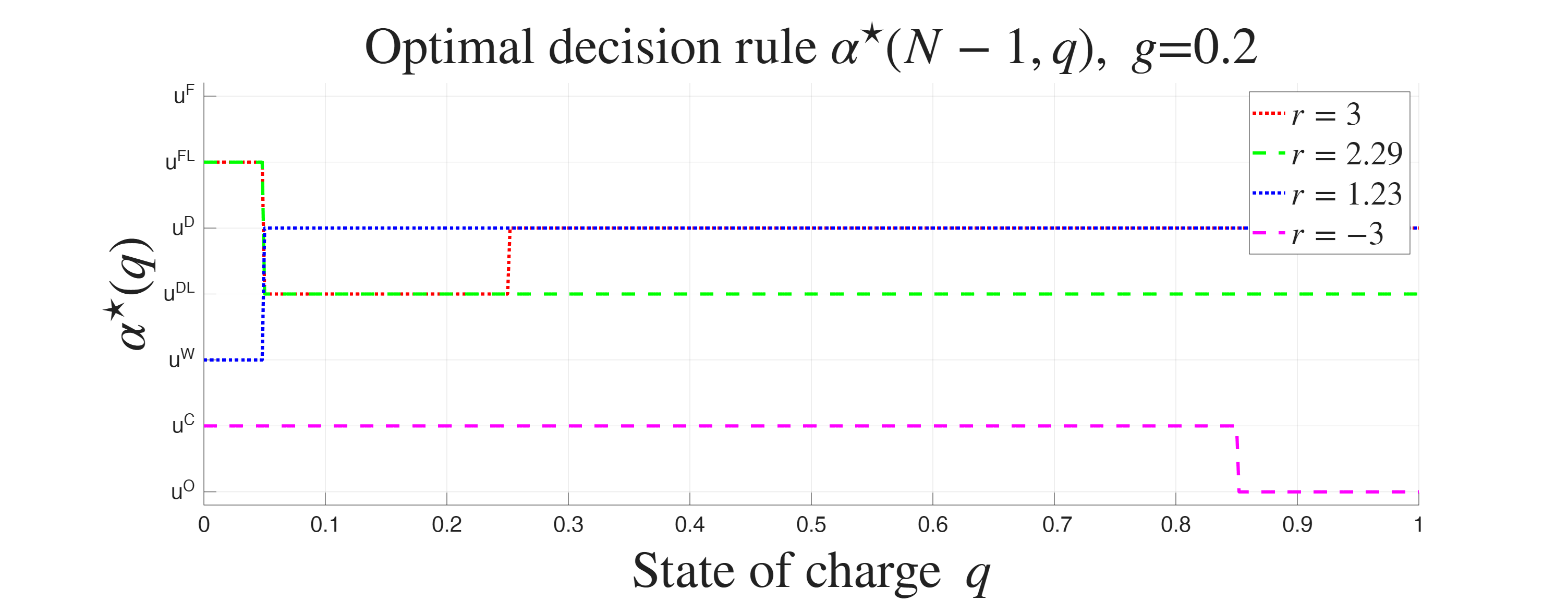}

\vspace*{0.3cm}
 \includegraphics[width=0.49\linewidth,height=0.215\linewidth]{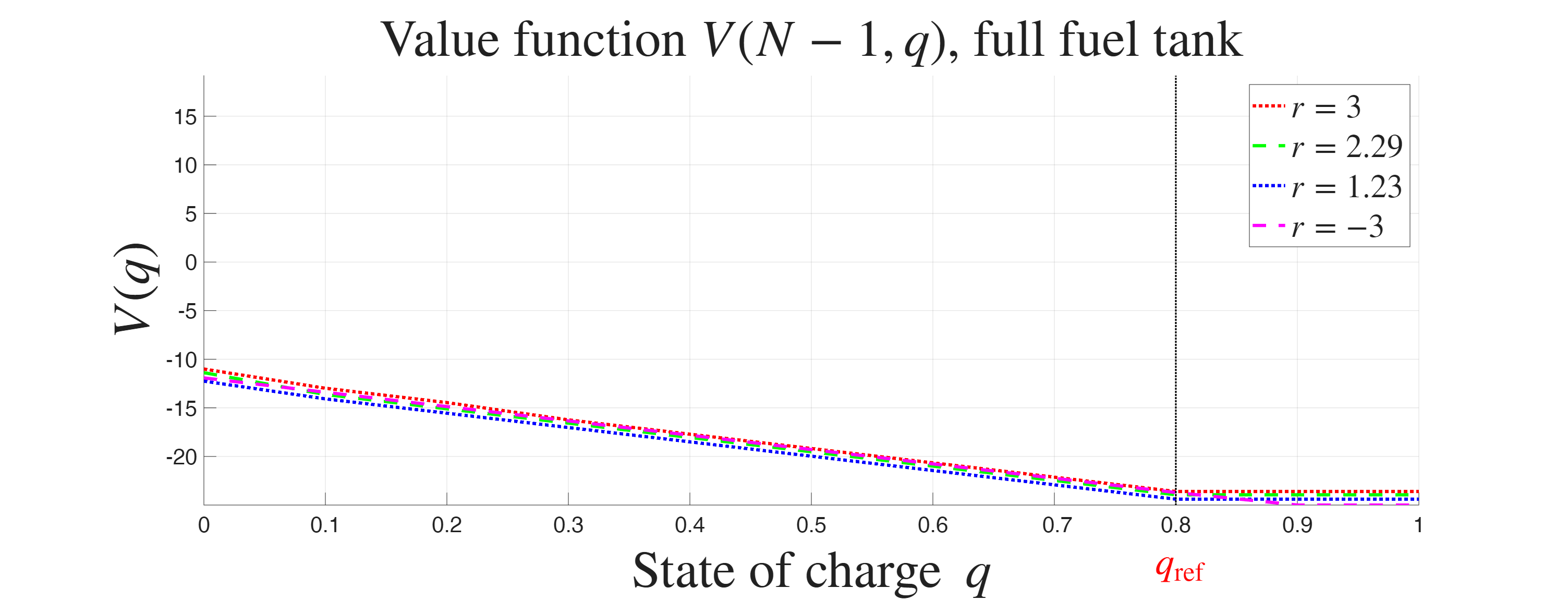}
\includegraphics[width=0.49\linewidth,height=0.215\linewidth]{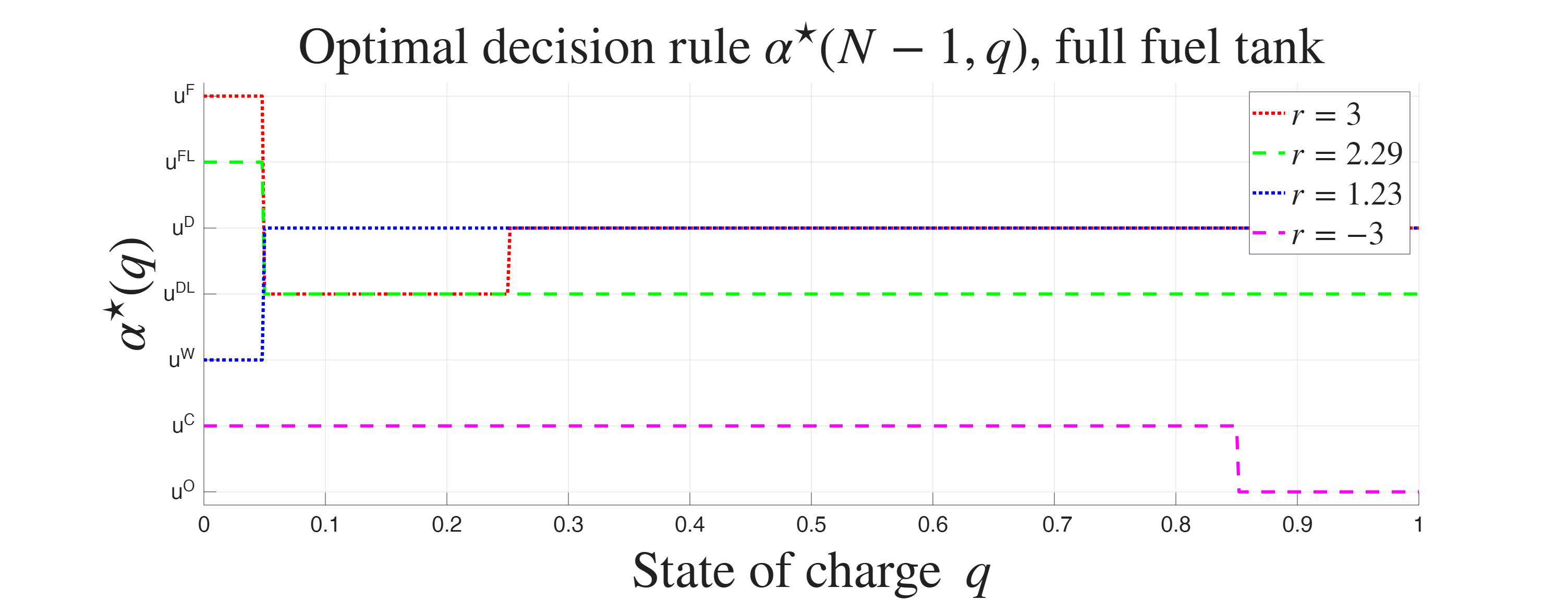} 
    \caption{Visualization of the value function (left) and optimal decision rule (right) one hour before the terminal time ($n=N-1$) as a function of the state of charge $q$ for a fixed residual demand \(r=\{-3, 1.23, 2.29, 3\}\). Upper panel: Empty fuel tank. Middle panel:  Fuel tank filled to 20$\%$ of its capacity. lower panel: Full fuel tank.}\label{fig_VFr-N-1}
\end{figure}

We observe in the upper left panel of Fig.~\ref{fig_Val_n-1} that the value function decreases sharply as the state of charge $q$ increases and remains constant when it is greater than the threshold ($q_{\text{ref}} = 80\%$). The left panel of Fig.~\ref{fig_VFq-N-1} shows that the value function increases very slowly (almost constant) when the residual demand increases. This behavior results from the trend of the terminal value function observed in Fig.~\ref{fig_Term}. {\revi The optimal decision rule in Fig.~\ref{fig_Val_n-1} and the right panel of Fig.~\ref{fig_VFr-N-1} indicates that, for a positive residual demand below the threshold $R_{Q0}$, it is optimal to discharge the battery in full mode if $q\ge q_{ref}$ and to wait otherwise. When the residual demand exceeds the threshold, it is preferable to discharge the battery in limited mode as long as it is not empty; the full mode is used only for $r>1.8$ and $q>90\%$ or $q<10\%$. When the battery is empty and $r>1.8$, it is preferable to use the generator in limited mode if the fuel tank is not empty and to wait otherwise.}\\
Similarly, the right panel of Figures ~\ref{fig_VFr-N-1} and ~\ref{fig_VFq-N-1} shows that for negative residual demand, charging the battery is optimal as long as the latter is not almost full; otherwise, it is preferable to over-spill. The bottom left panel of Fig.\ref{fig_Val_n-1} shows that when there is a positive residual demand and both the fuel tank and the battery are empty, we have no other option than to wait and pay a penalty (discomfort cost). This results in overall higher system costs, and this justifies the fact that the value function for an empty fuel tank is all-time greater than the case of a full fuel tank; see the upper left panel of Fig.~\ref{fig_Val_n-1}.

\begin{figure}[htbp!]
    \centering
\includegraphics[width=0.49\linewidth,height=0.215\linewidth]{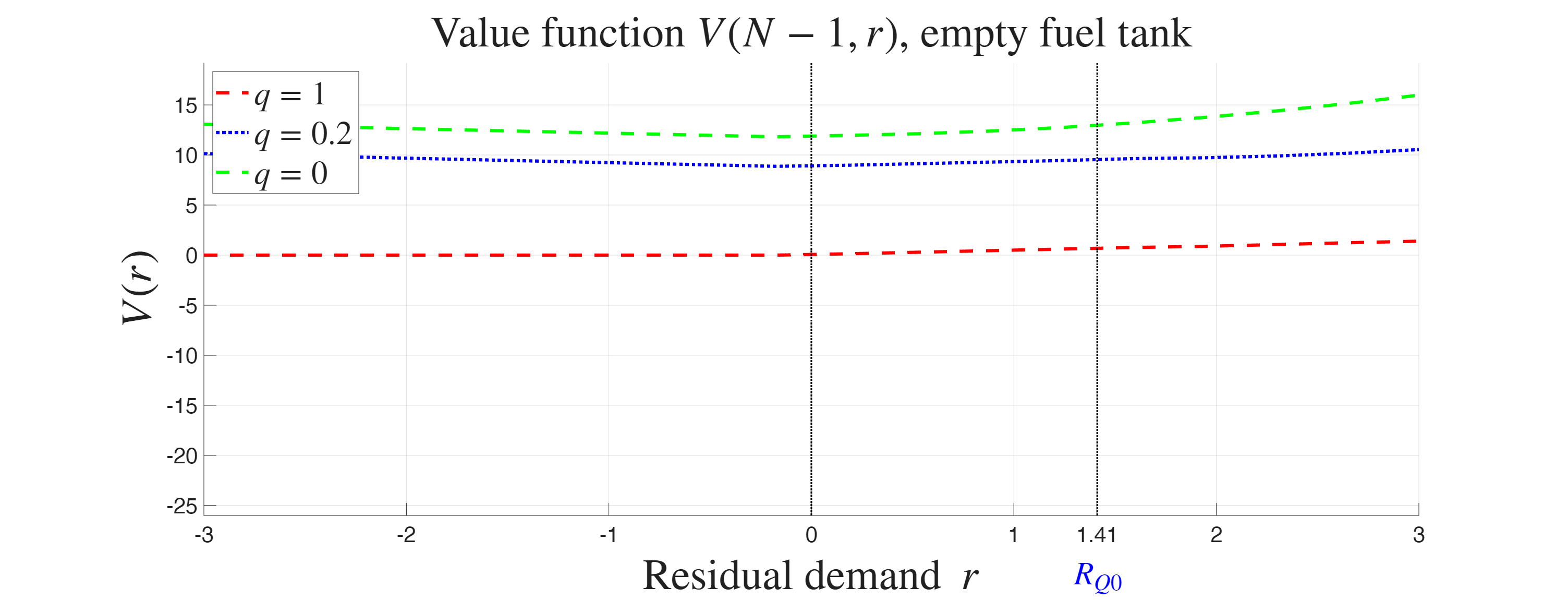}
\includegraphics[width=0.49\linewidth,height=0.215\linewidth]{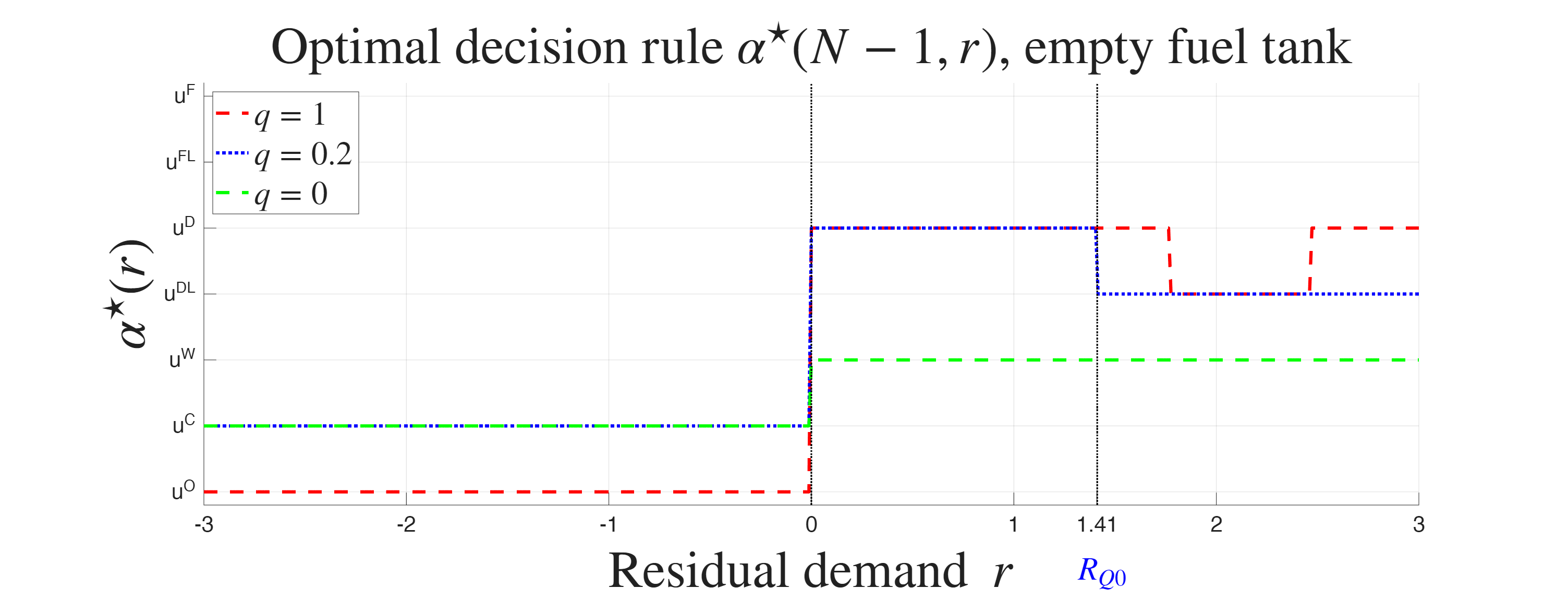}  
\vspace*{0.3cm}

\includegraphics[width=0.49\linewidth,height=0.215\linewidth]{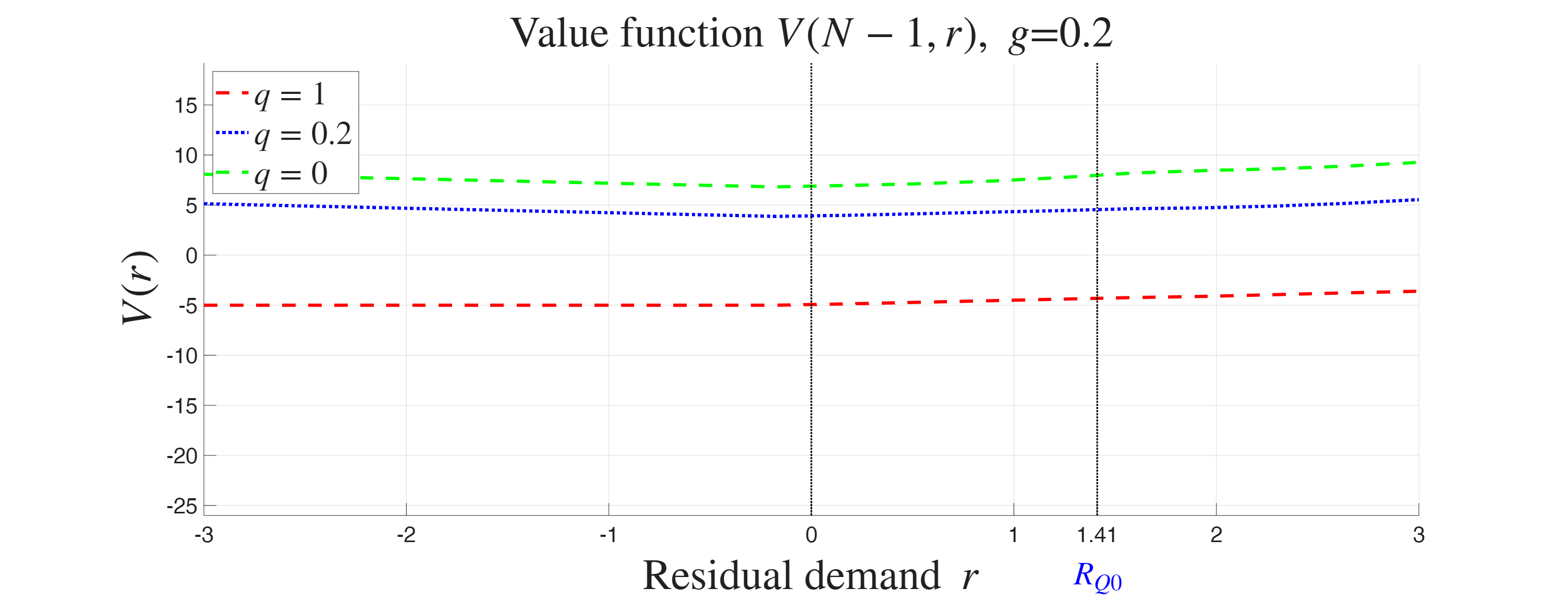}
\includegraphics[width=0.49\linewidth,height=0.215\linewidth]{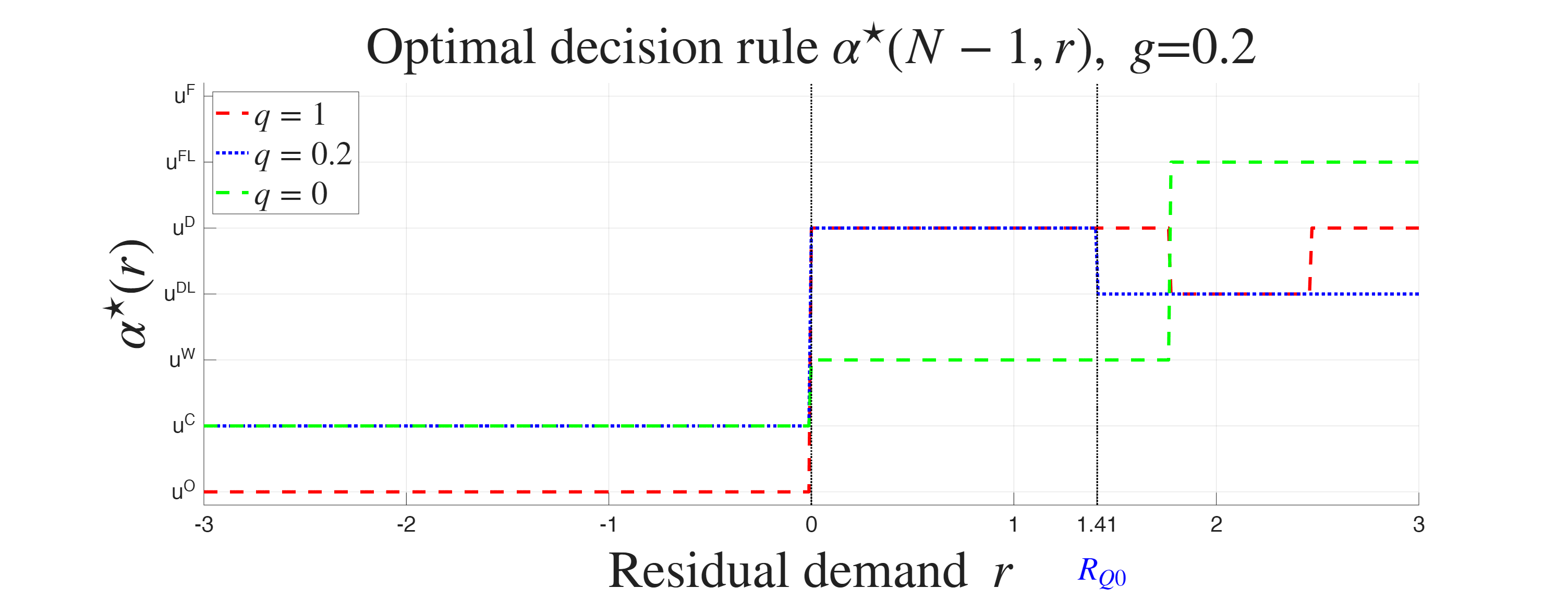}
\vspace*{0.3cm}

\includegraphics[width=0.49\linewidth,height=0.215\linewidth]{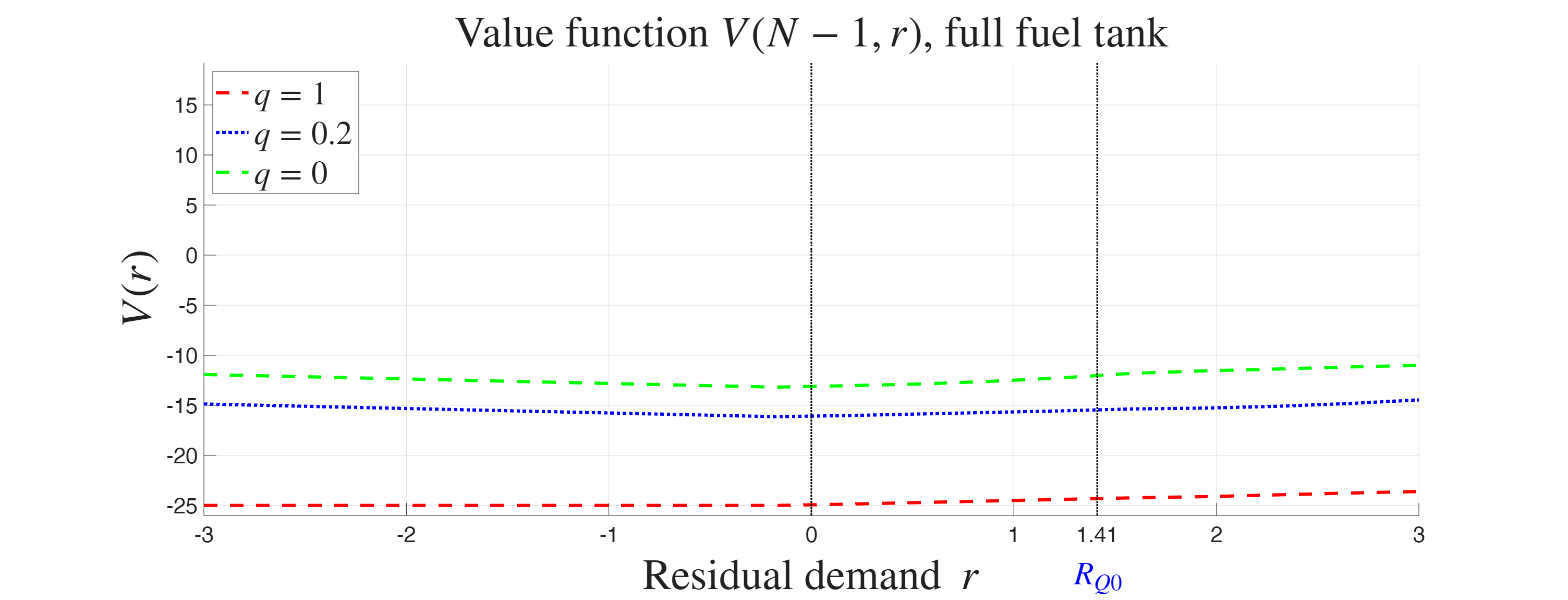}
\includegraphics[width=0.49\linewidth,height=0.215\linewidth]{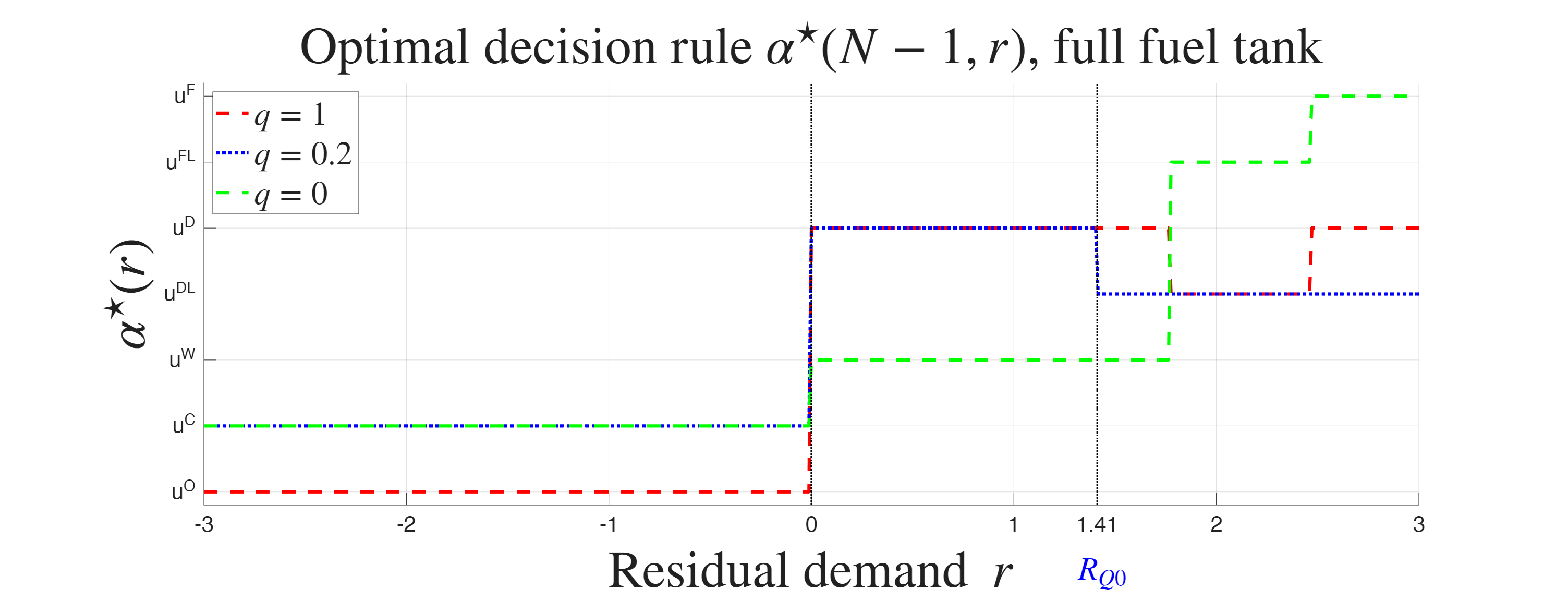} 
  \caption{Visualization of the value function (left) and optimal decision rule (right) one hour before the terminal time ($n=N-1$) as a function of {\revi the}  residual demand $r$ for a fixed  state of charge \(q=\{0,20\%,100\%\}\). Upper panel: Empty fuel tank. Middle panel:  Fuel tank filled to 20$\%$ of its capacity. lower panel: Full fuel tank.}\label{fig_VFq-N-1}
\end{figure}

  \newpage
 
\paragraph{Optimal decision rule and value function 12 hours before the terminal time}

Figure ~\ref{fig_Val_n-12} indicates that when the residual demand is negative, it is preferable to charge the battery as long as it is not almost full {\revi ($q\lesssim 85$--$95\%$, depending on $r$)}; otherwise, it is optimal to over-spill. {\revi For a positive residual demand below the threshold ($0<r<R_{Q0}$), it is preferable to wait if $q<q_{ref}$ and to discharge the battery in full mode otherwise. For $R_{Q0}\le r<1.8$, it is preferable to discharge the battery in limited mode if $q>10\%$ and to wait otherwise. For $r\ge 1.8$, it is preferable to discharge the battery in limited mode if $q>10\%$ and in full mode otherwise; for an empty battery, the generator is used in limited mode if the fuel tank is not empty.} Similarly to time $N-1$, the left panels of Figures ~\ref{fig_VFq-N-12} and ~\ref{fig_VFr-N-12} show that the value function increases with the decreasing state of charge. {\revi In contrast to time $N-1$, where it increases only slowly with $r$ (about $0.6$--$0.7$\,\texteuro/kW), it increases with the residual demand over the whole range, more steeply for $r>0$ (about $2.9$\,\texteuro/kW for a full tank and $q=0.5$) than for $r<0$ (about $1.8$\,\texteuro/kW). The value function takes values between $-16.9$ and $26.3$\,\texteuro{}, and the liquidation value of the fuel, which is discounted only over $12$ hours, separates the three fuel levels (at $r=3$, $q=0.8$: $10.1$, $6.5$, and $-7.4$\,\texteuro{} for $g=0$, $0.2$, $1$).} This behavior arises from the discomfort cost associated with unmet demand, as well as the operation of the battery and generator in a limited mode.

  \begin{figure}[htbp!]
    \centering 
\includegraphics[width=0.495\linewidth,height=0.217\linewidth]{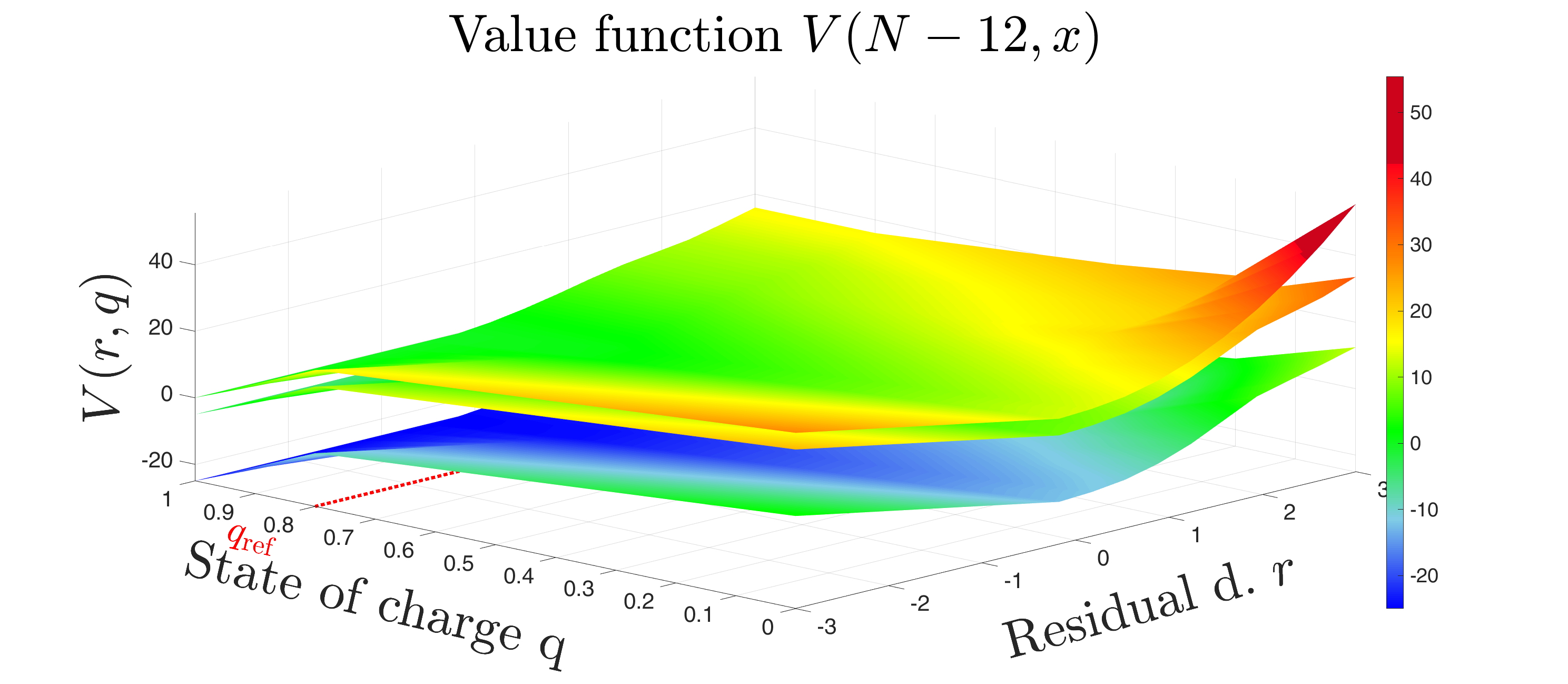} 
\includegraphics[width=0.495\linewidth,height=0.19\linewidth]{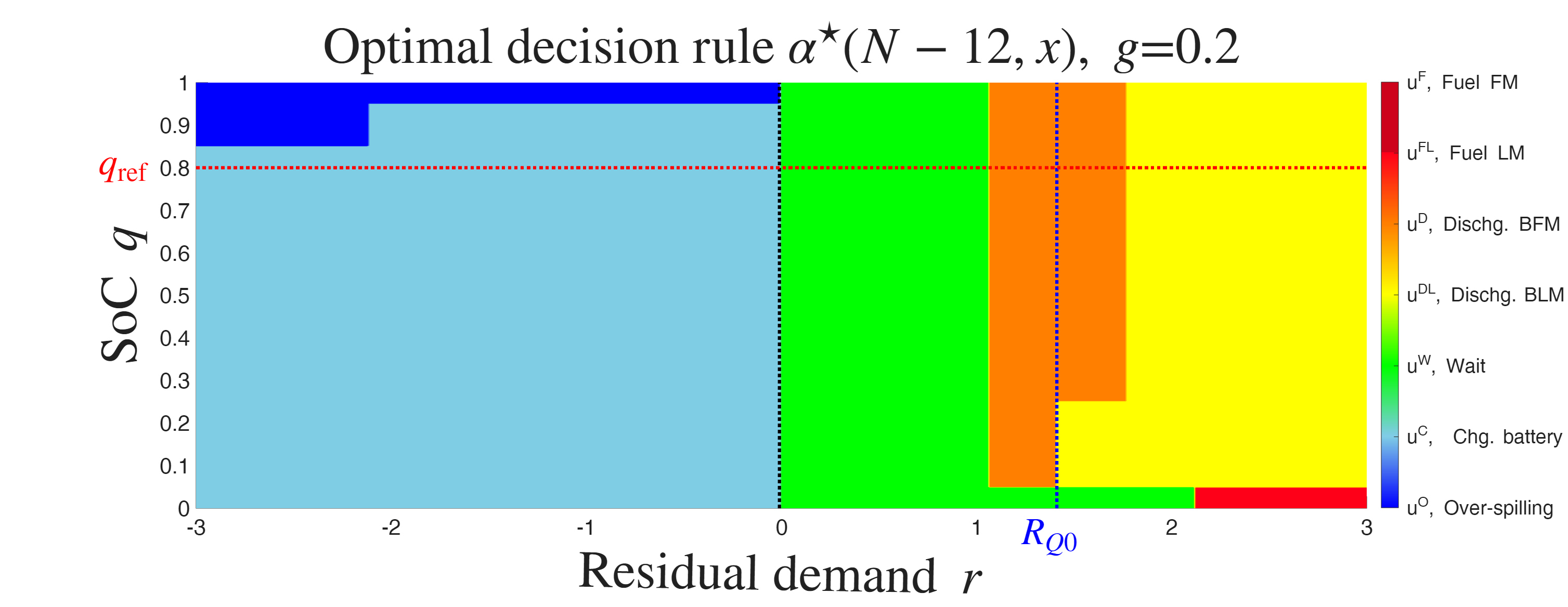}
\vspace*{0.1cm}

\includegraphics[width=0.495\linewidth,height=0.19\linewidth]{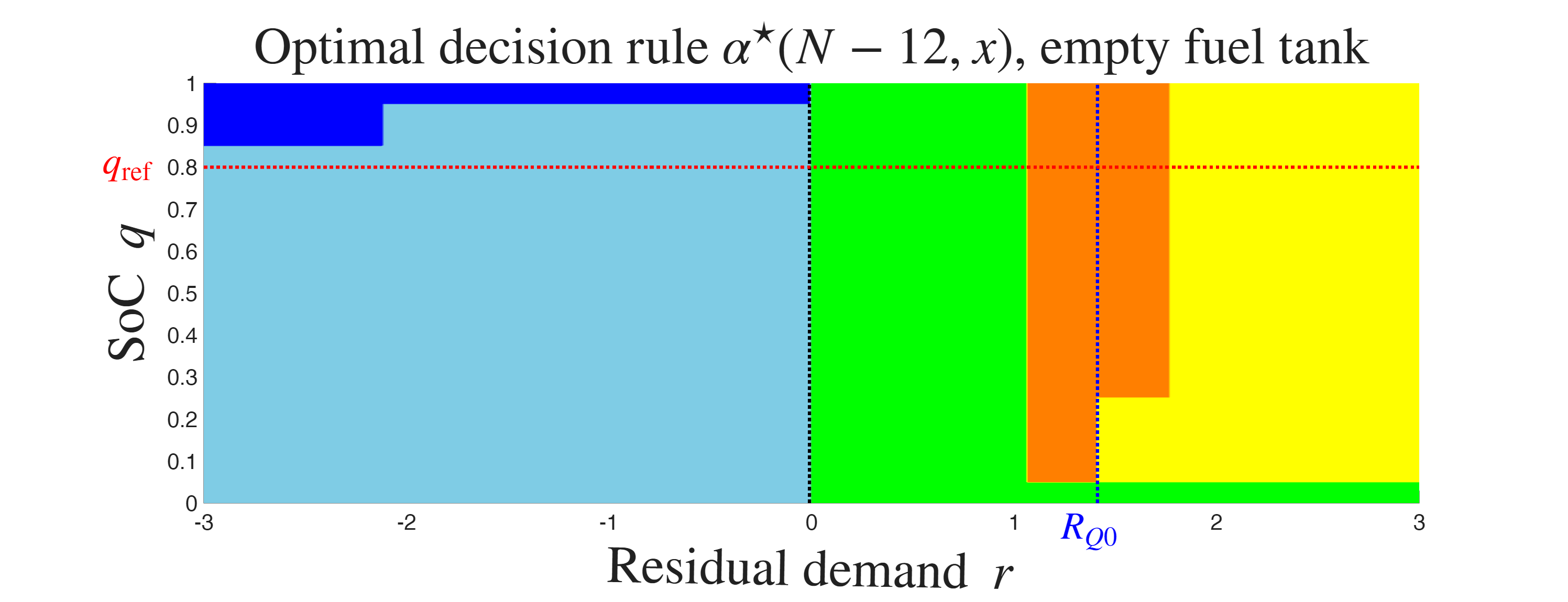}
\includegraphics[width=0.495\linewidth,height=0.19\linewidth]{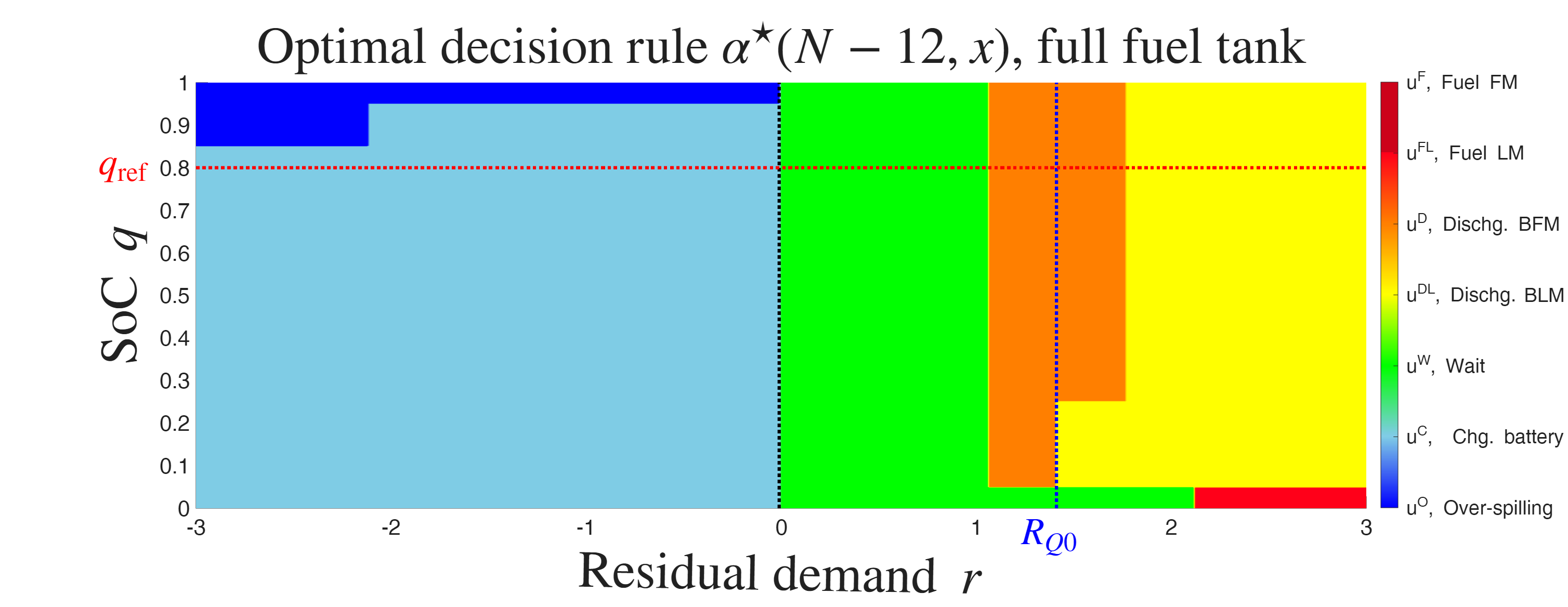}
    \caption{Value functions \(V\) and optimal decision rule \(\alpha^\star\) 12 hours before the terminal time ($n = N-12$) as a function of $r$ and $q$. 
    Upper left panel: Value function for an empty fuel tank (upper graph), a fuel tank filled to 20$\%$ (middle graph), and a full fuel tank (bottom graph). Lower left panel: Optimal decision rule for an empty fuel tank. Lower right panel: Optimal decision rule for a full fuel tank. Upper right panel: Optimal decision rule for a fuel tank filled to 20$\%$ capacity.}\label{fig_Val_n-12}
\end{figure}

\begin{figure}[b!]
    \centering   
\includegraphics[width=0.49\linewidth,height=0.215\linewidth]{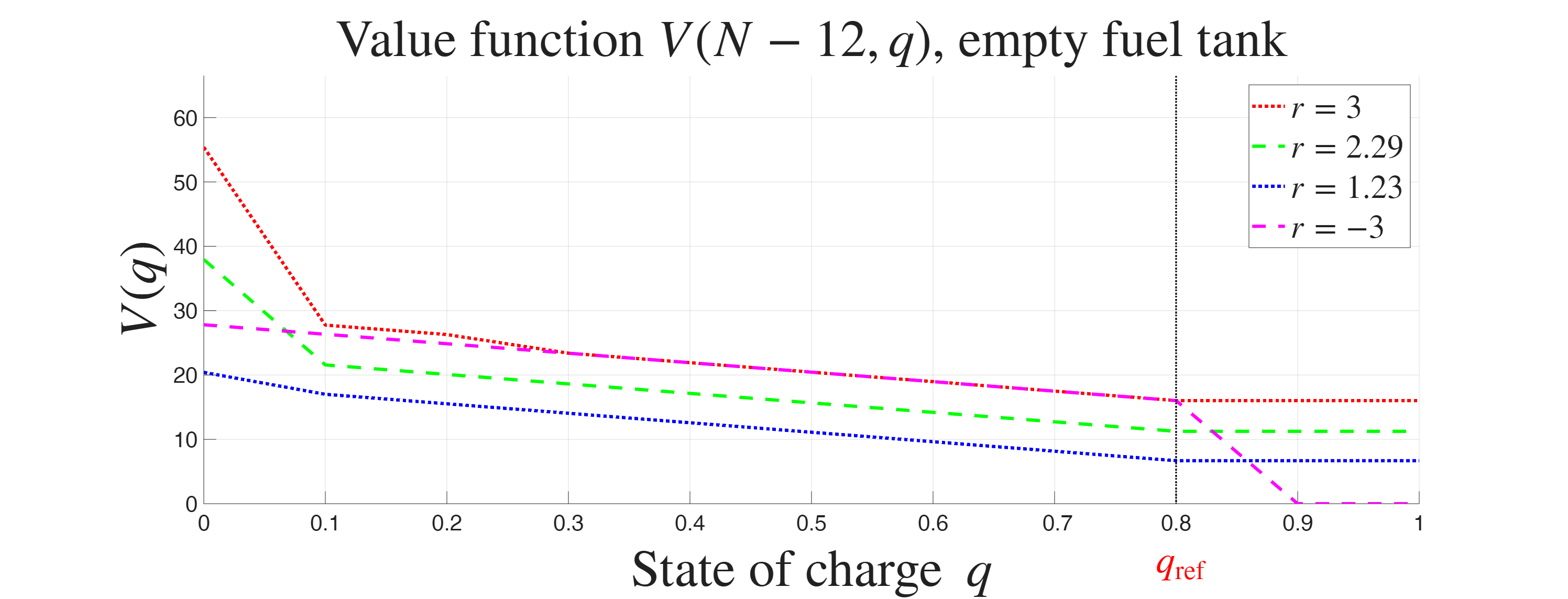}
\includegraphics[width=0.49\linewidth,height=0.215\linewidth]{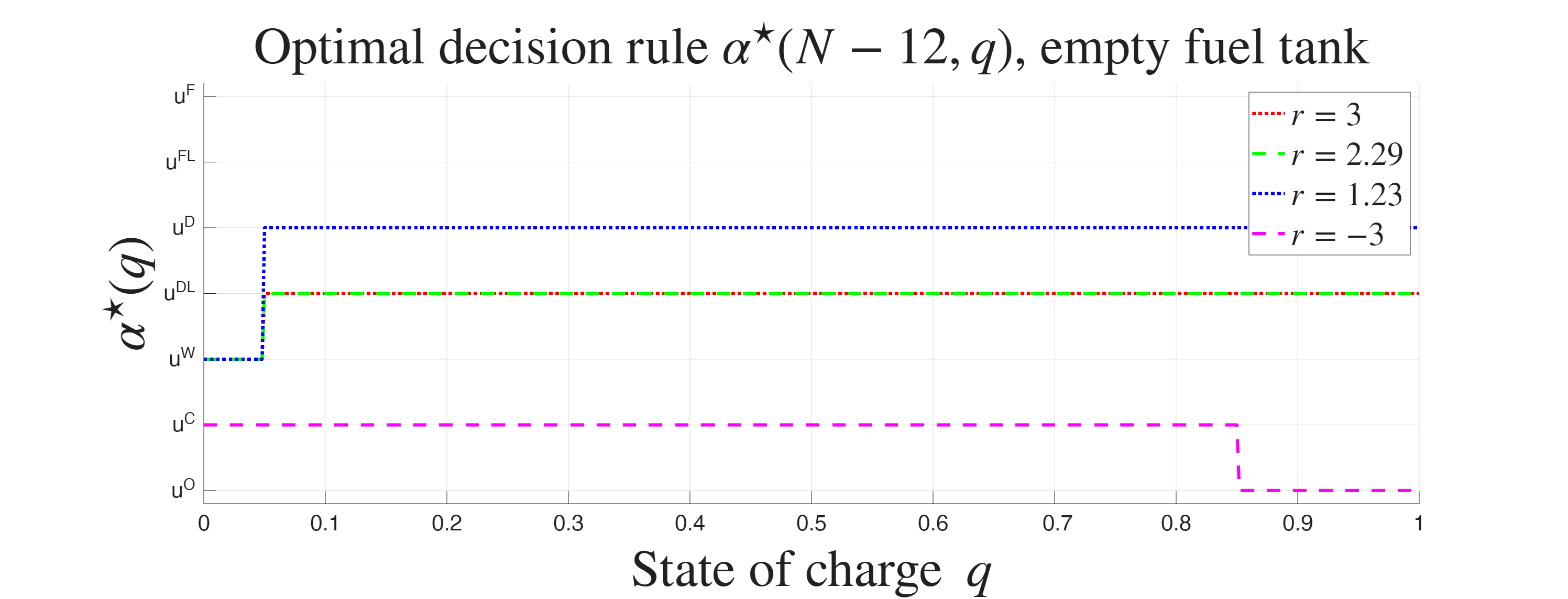}  
\vspace*{0.3cm}

\includegraphics[width=0.49\linewidth,height=0.215\linewidth]{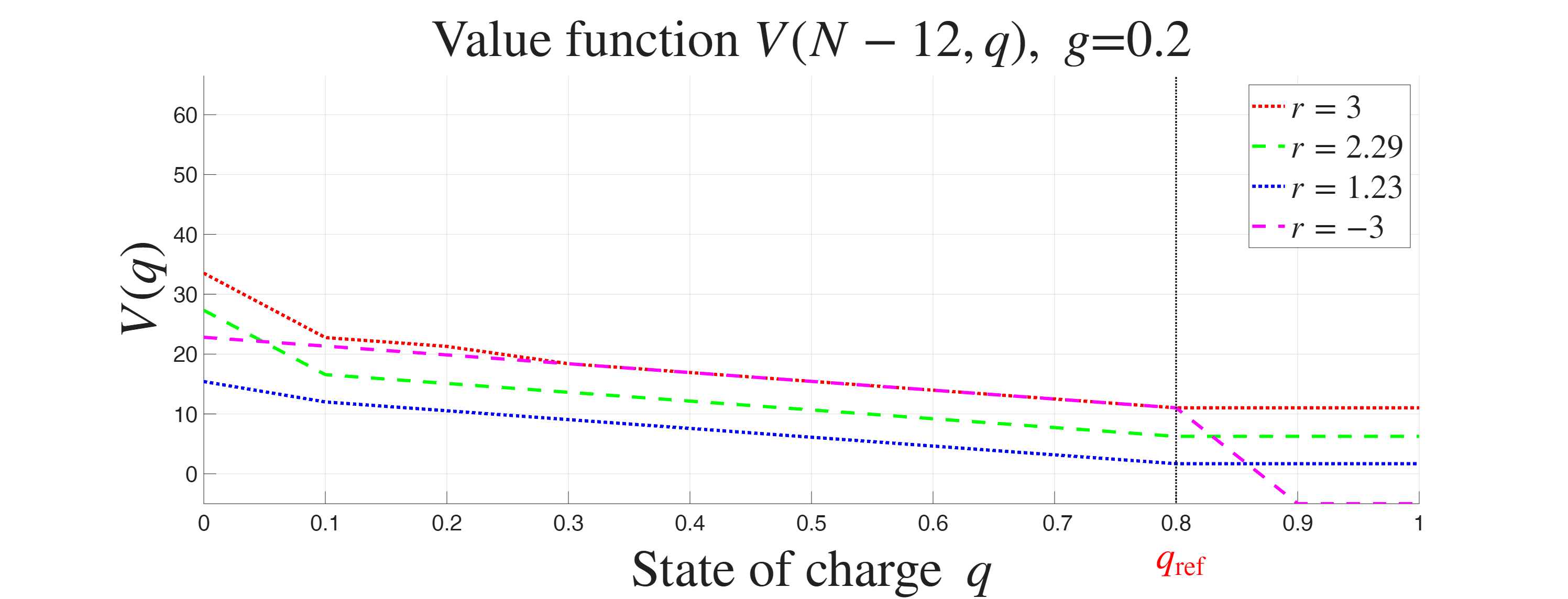}
\includegraphics[width=0.49\linewidth,height=0.215\linewidth]{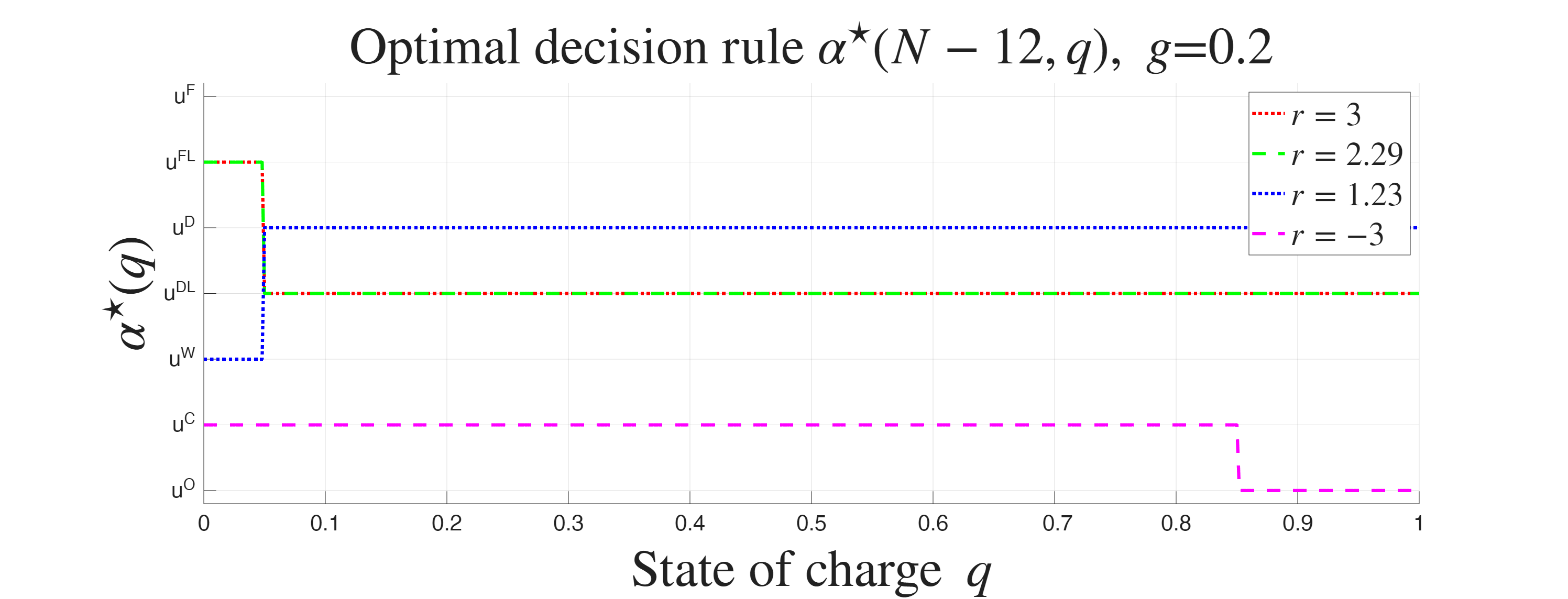}
\vspace*{0.3cm}

\includegraphics[width=0.49\linewidth,height=0.215\linewidth]{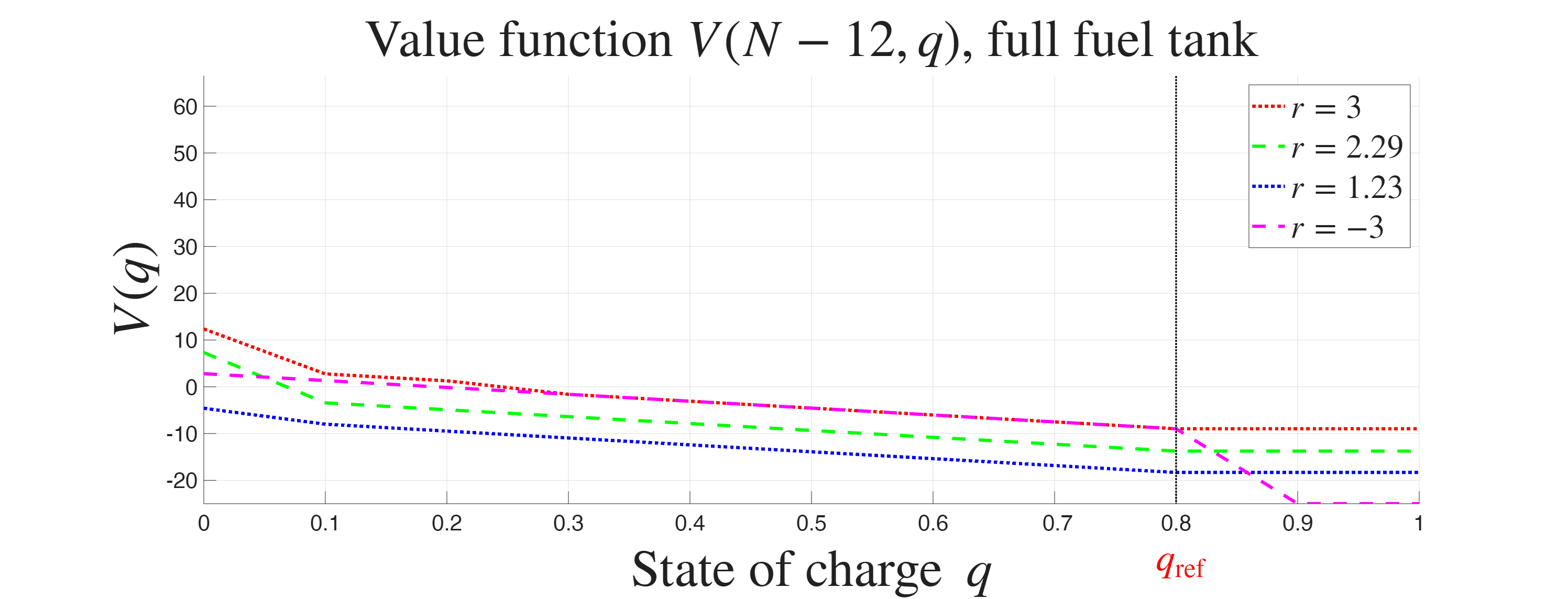}
\includegraphics[width=0.49\linewidth,height=0.215\linewidth]{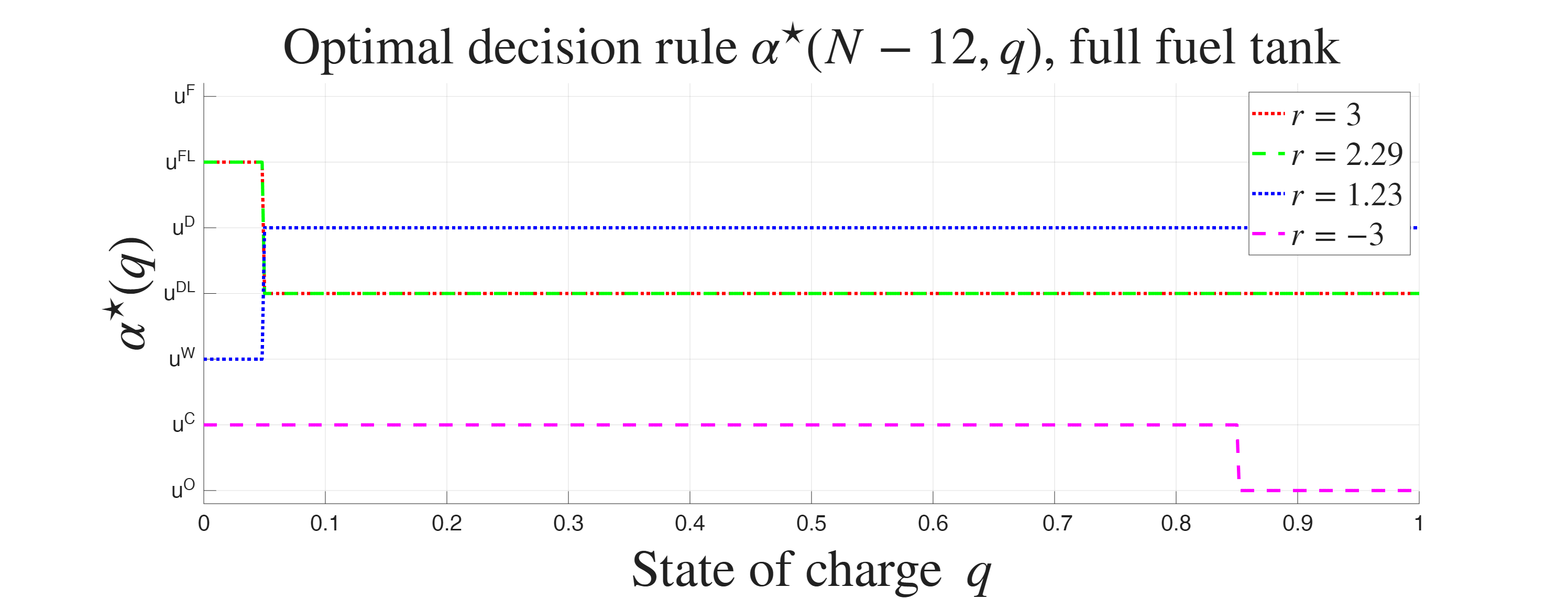} 
    \caption{Visualization of the value function (left) and optimal decision rule (right) 12 hours before the terminal time ($n=N-12$) as a function of the state of charge $q$ for a fixed residual demand \(r=\{-3, 1.23, 2.29, 3\}\). Upper panel: Empty fuel tank. Middle panel:  Fuel tank filled to 20$\%$ of its capacity. lower panel: Full fuel tank.}\label{fig_VFr-N-12}
\end{figure}

\begin{figure}[htbp!]
    \centering
    \includegraphics[width=0.49\linewidth,height=0.215\linewidth]{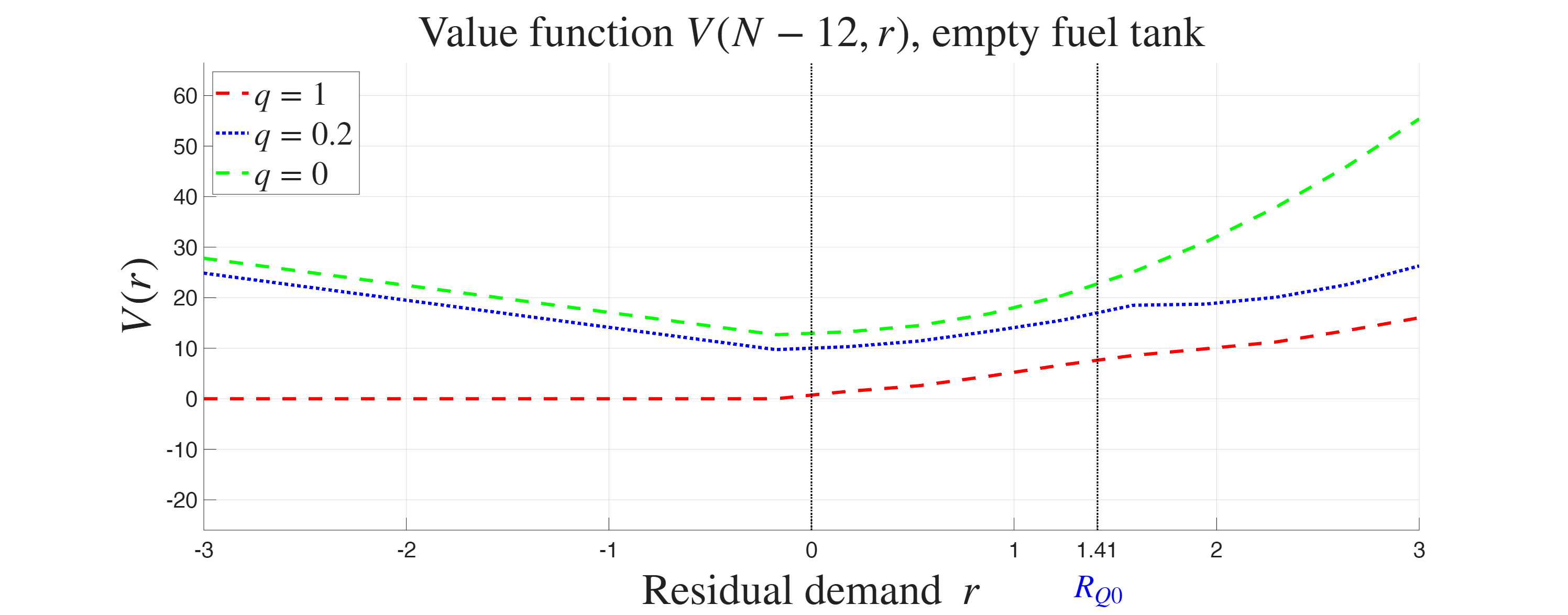}
\includegraphics[width=0.49\linewidth,height=0.215\linewidth]{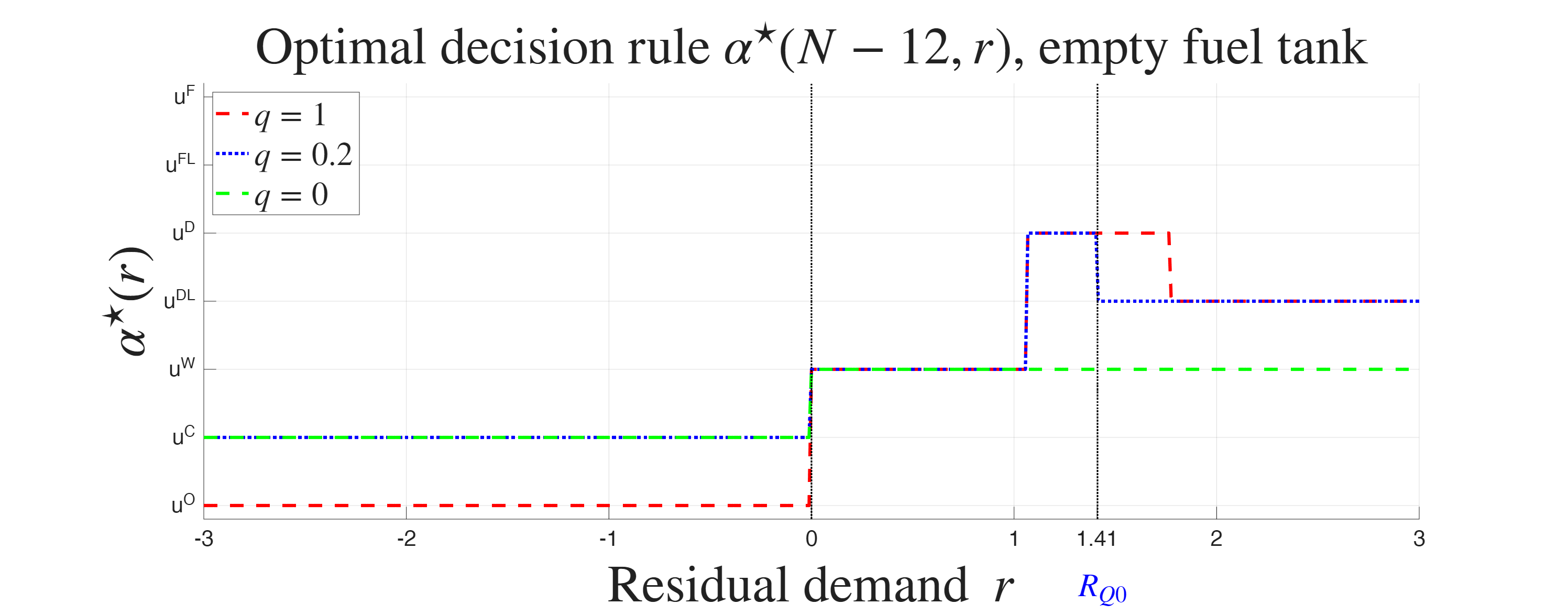}  

\vspace*{0.3cm}
 \includegraphics[width=0.49\linewidth,height=0.215\linewidth]{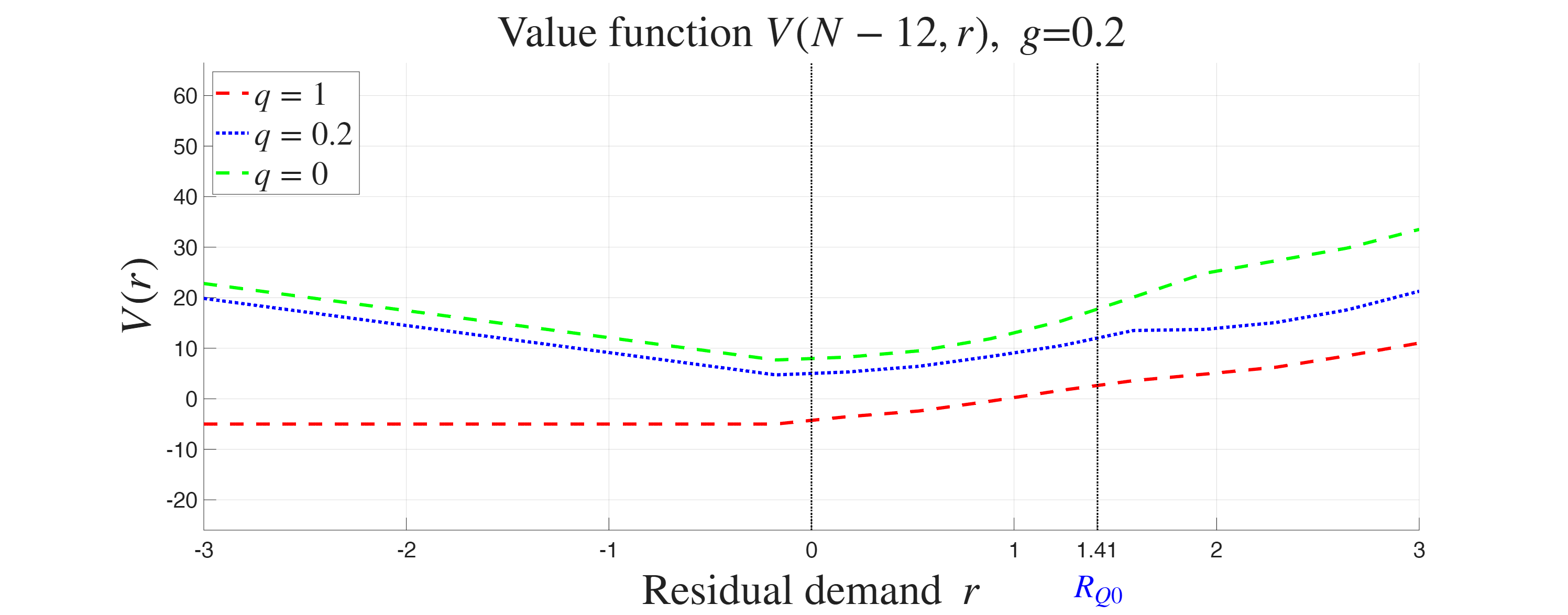}
\includegraphics[width=0.49\linewidth,height=0.215\linewidth]{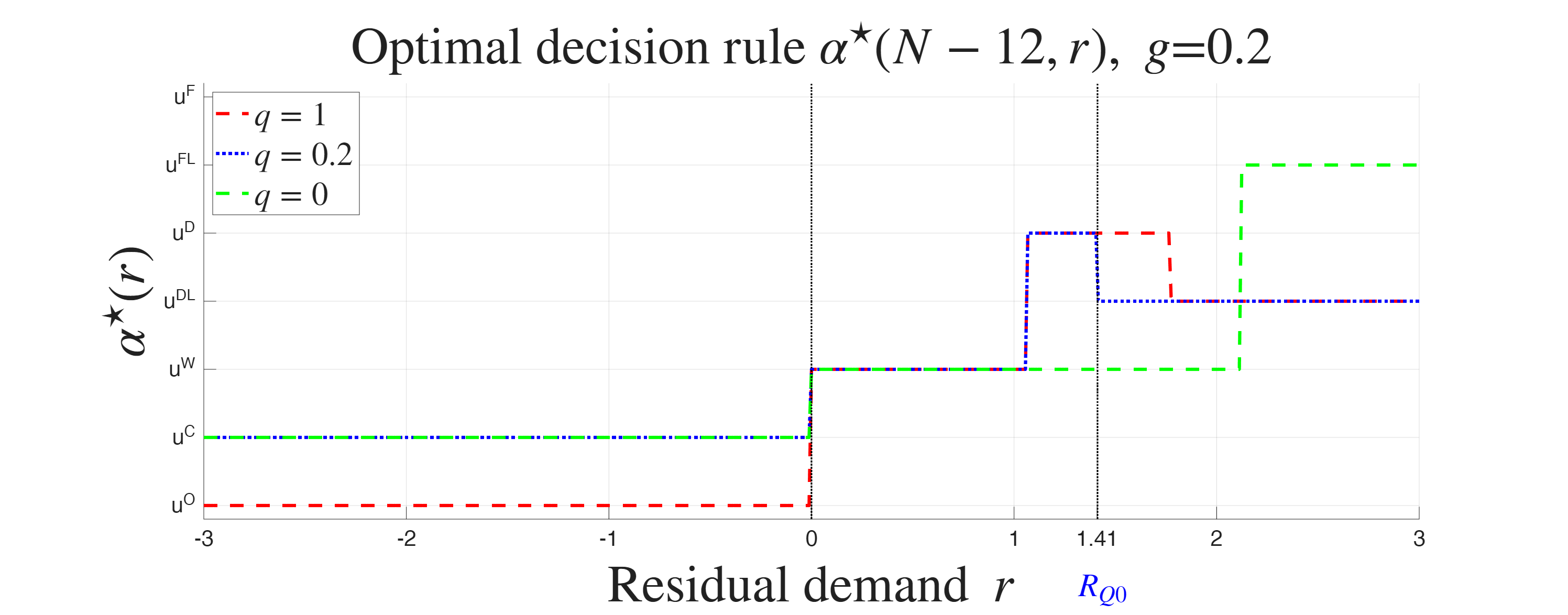}

\vspace*{0.3cm}
 \includegraphics[width=0.49\linewidth,height=0.215\linewidth]{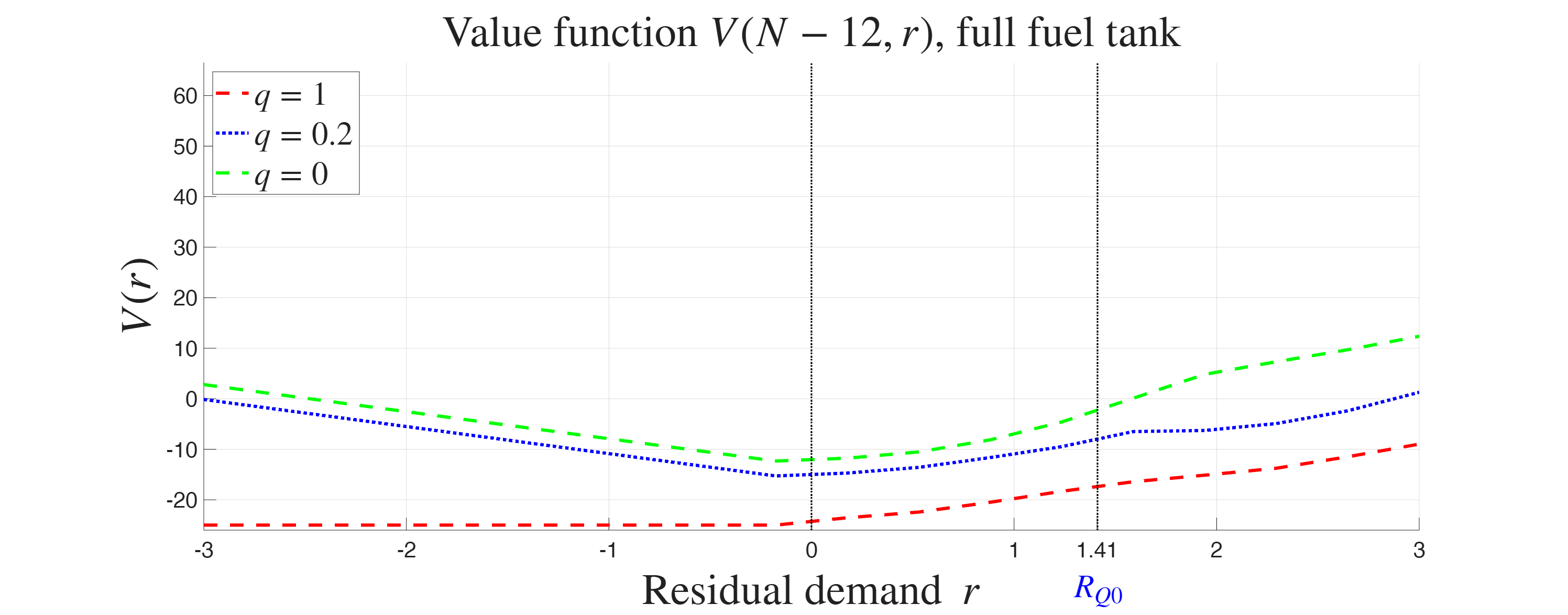}
\includegraphics[width=0.49\linewidth,height=0.215\linewidth]{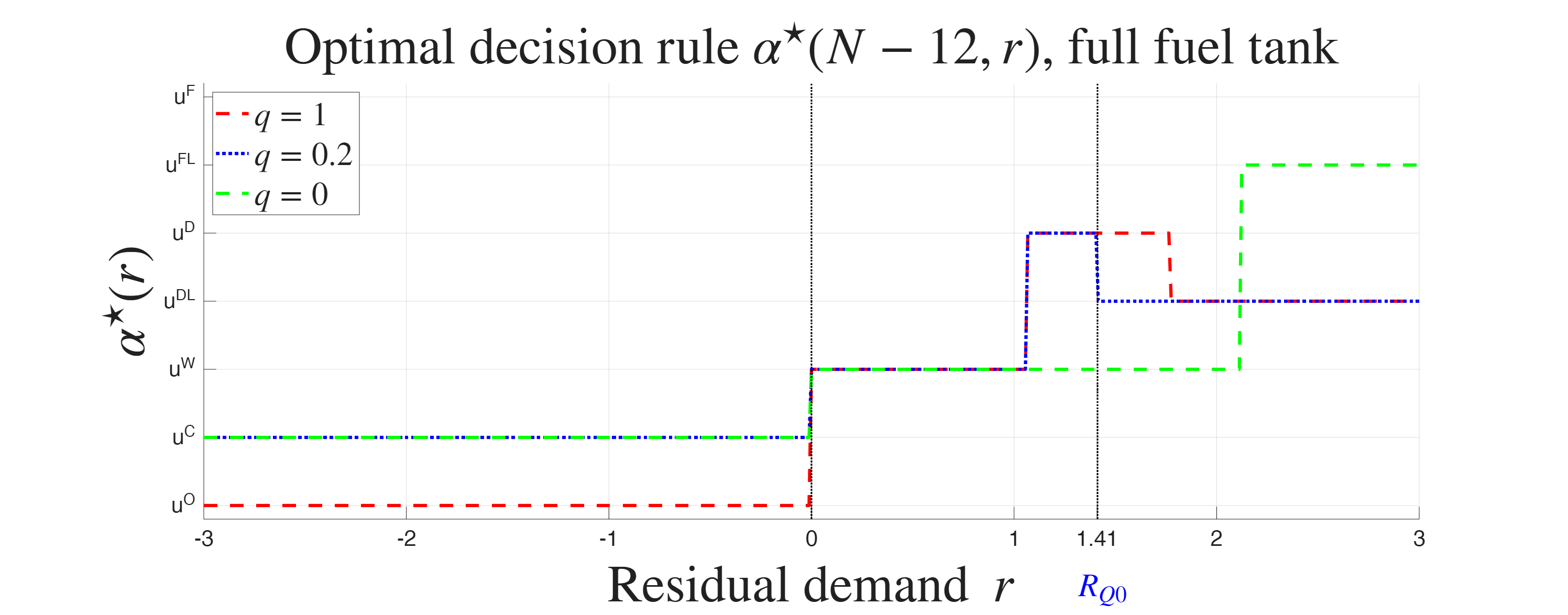} 
 \caption{Visualization of the value function (left) and optimal decision rule (right) 12 hours before the terminal time ($n=N-12$) as a function of {\revi the}  residual demand $r$ for a fixed  state of charge \(q=\{0,20\%,100\%\}\). Upper panel: Empty fuel tank. Middle panel:  Fuel tank filled to 20$\%$ of its capacity. lower panel: Full fuel tank.}\label{fig_VFq-N-12}
\end{figure}

\section*{Value function and optimal decision rule at initial time ($\mathbf{n=0})$}

 \begin{figure}[htbp!]
    \centering 
     \includegraphics[width=0.495\linewidth,height=0.217\linewidth]{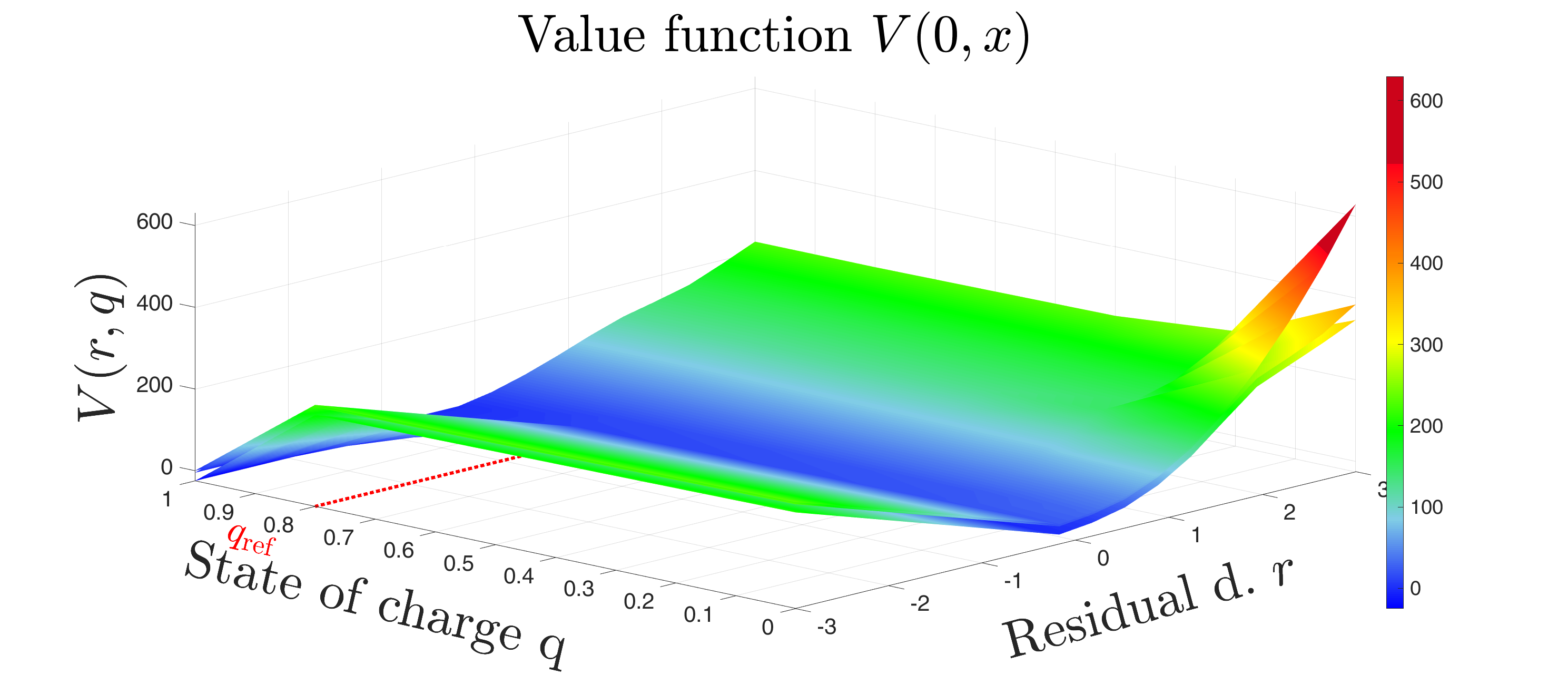} 
     \includegraphics[width=0.495\linewidth,height=0.19\linewidth]{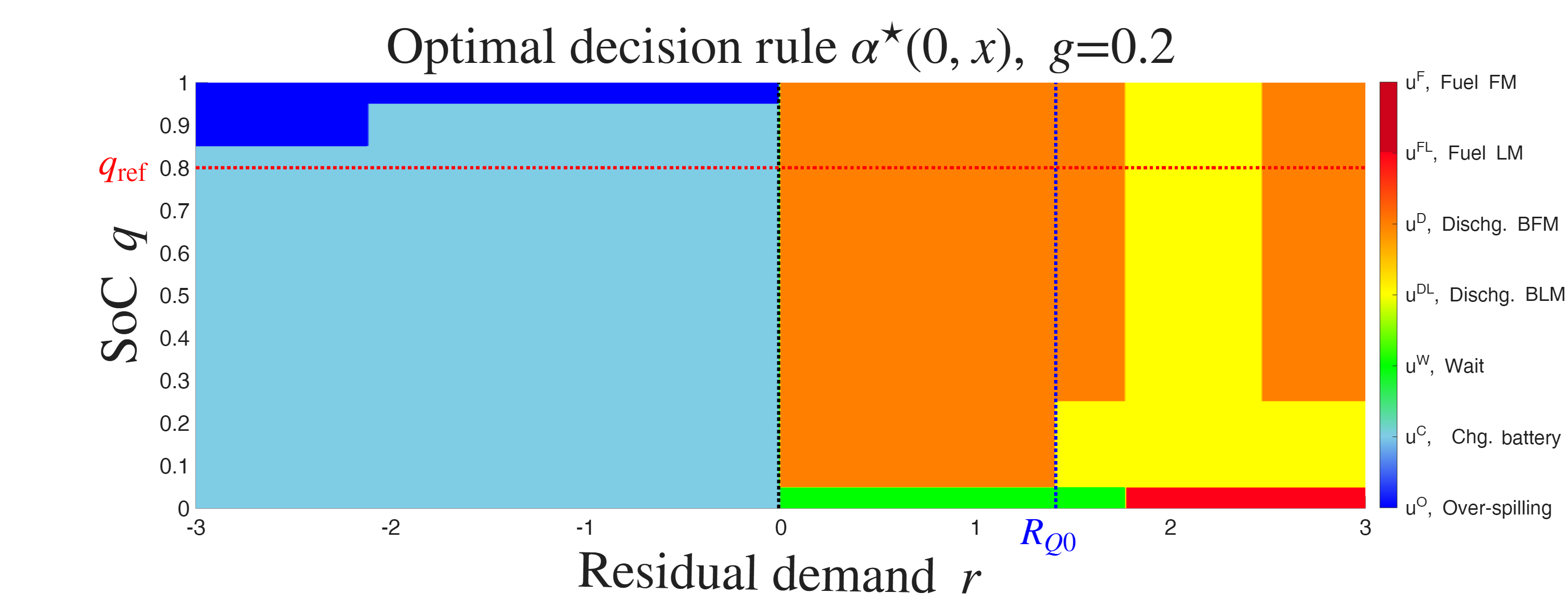}
\vspace*{0.1cm}

\includegraphics[width=0.495\linewidth,height=0.19\linewidth]{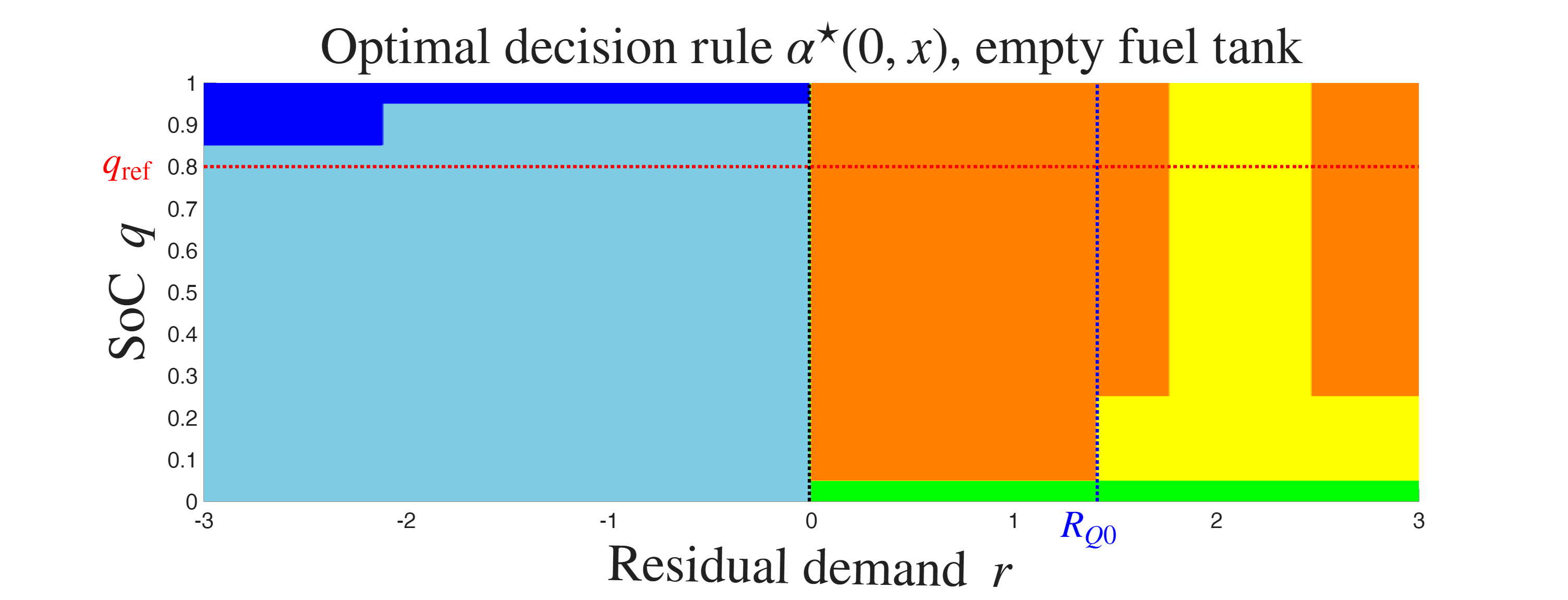}   \includegraphics[width=0.495\linewidth,height=0.19\linewidth]{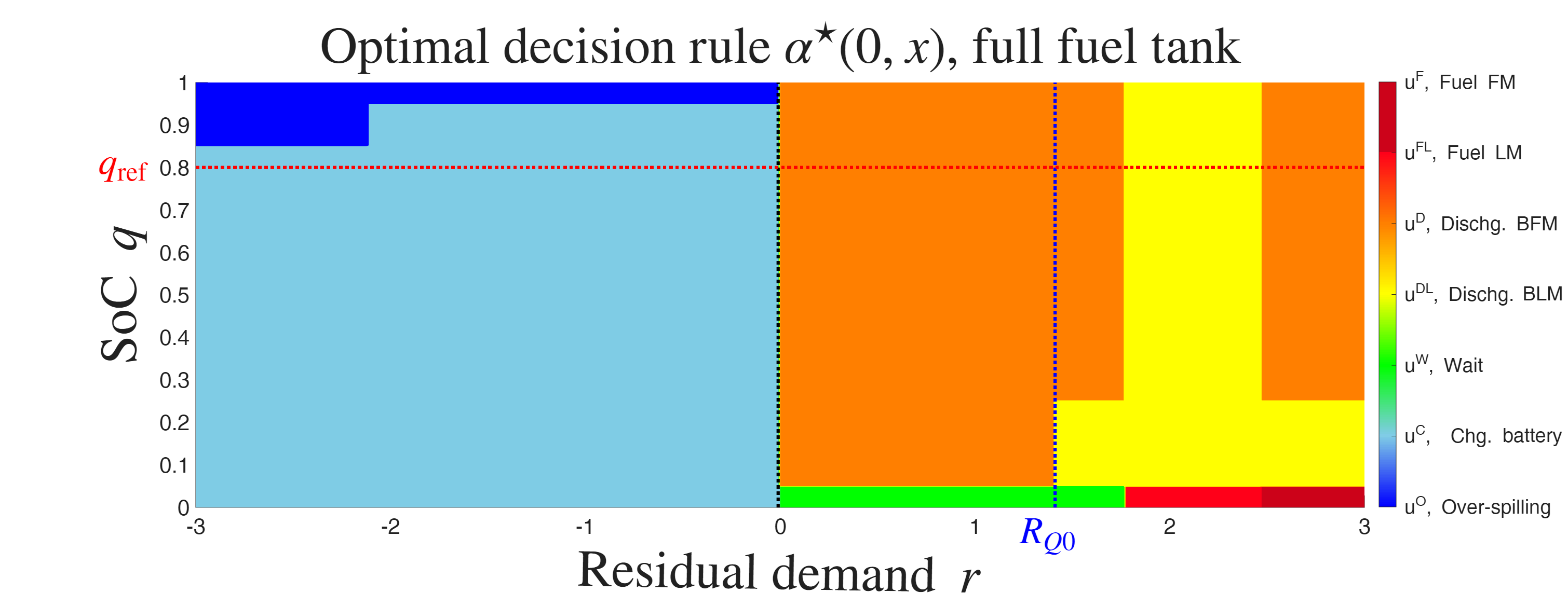}
    \caption{Value functions \(V\) and optimal decision rule \(\alpha^\star\) in terms of $r$ and $q$ at the initial time ($n= 0$). 
    Upper left panel: Value function for an empty fuel tank (upper graph), a fuel tank filled to 20$\%$ (middle graph), and a full fuel tank (bottom graph). Lower left panel: Optimal decision rule for an empty fuel tank. Lower right panel: Optimal decision rule for a full fuel tank. Upper right panel: Optimal decision rule for a fuel tank filled to 20$\%$ capacity.}\label{fig_Val_0}
\end{figure}
\begin{figure}[b!]
    \centering   
    \includegraphics[width=0.49\linewidth,height=0.215\linewidth]{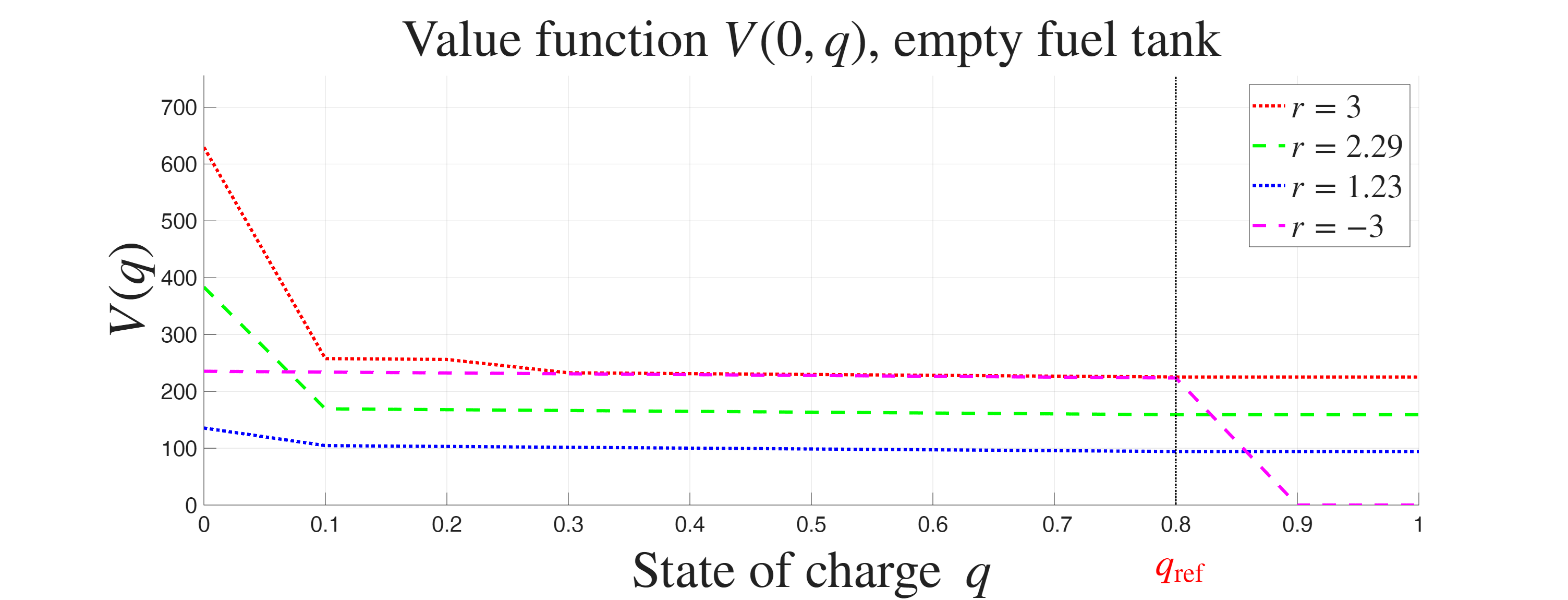}
\includegraphics[width=0.49\linewidth,height=0.215\linewidth]{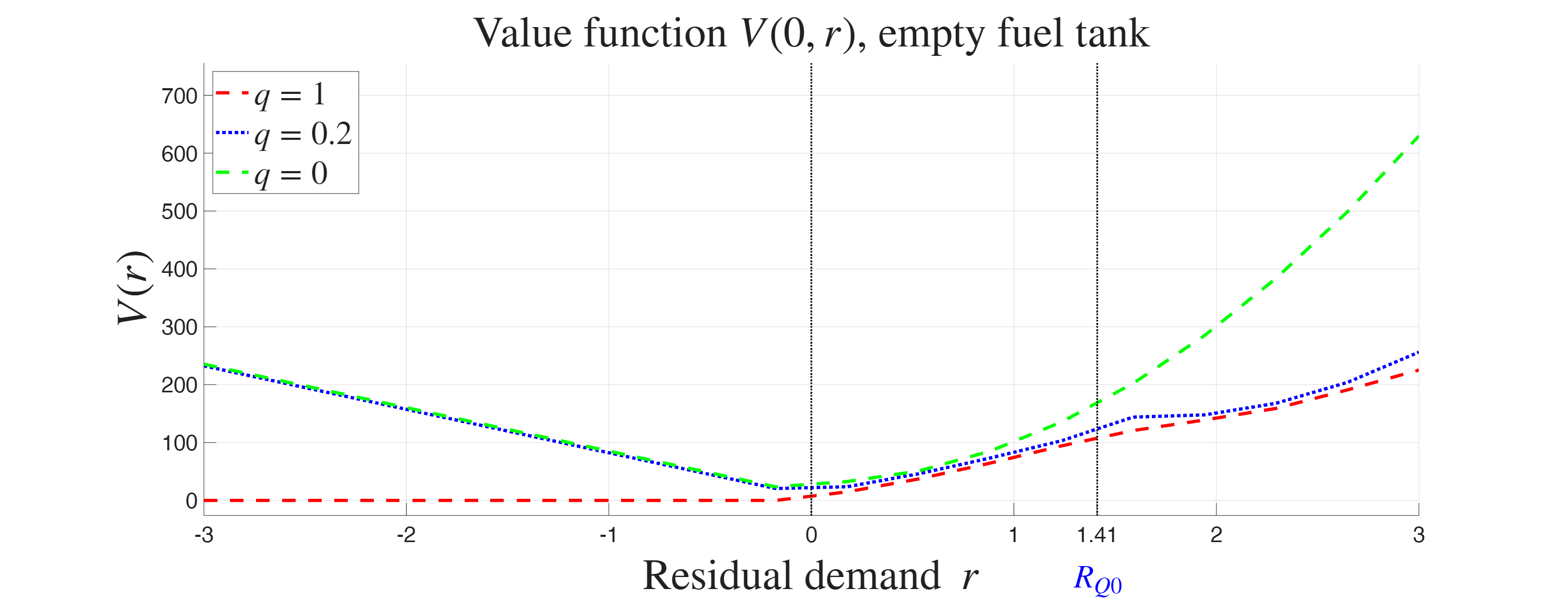}  

\vspace*{0.3cm}
 \includegraphics[width=0.49\linewidth,height=0.215\linewidth]{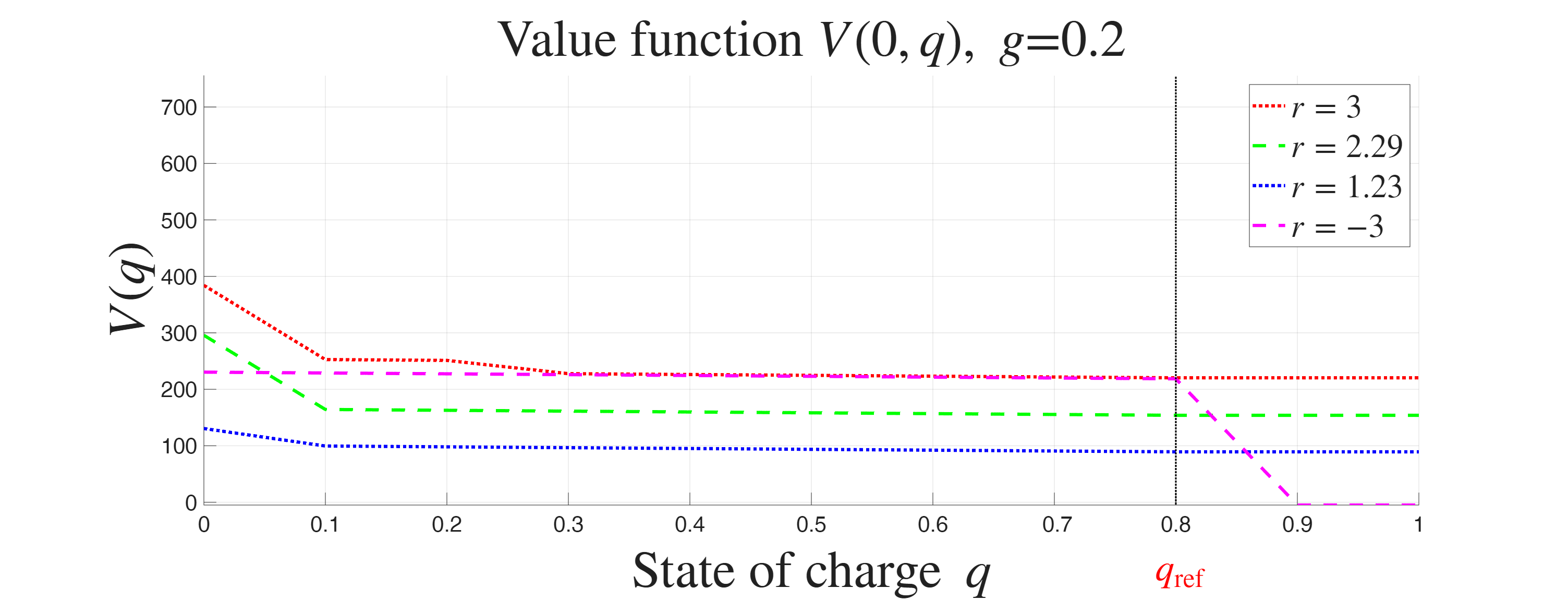}
\includegraphics[width=0.49\linewidth,height=0.215\linewidth]{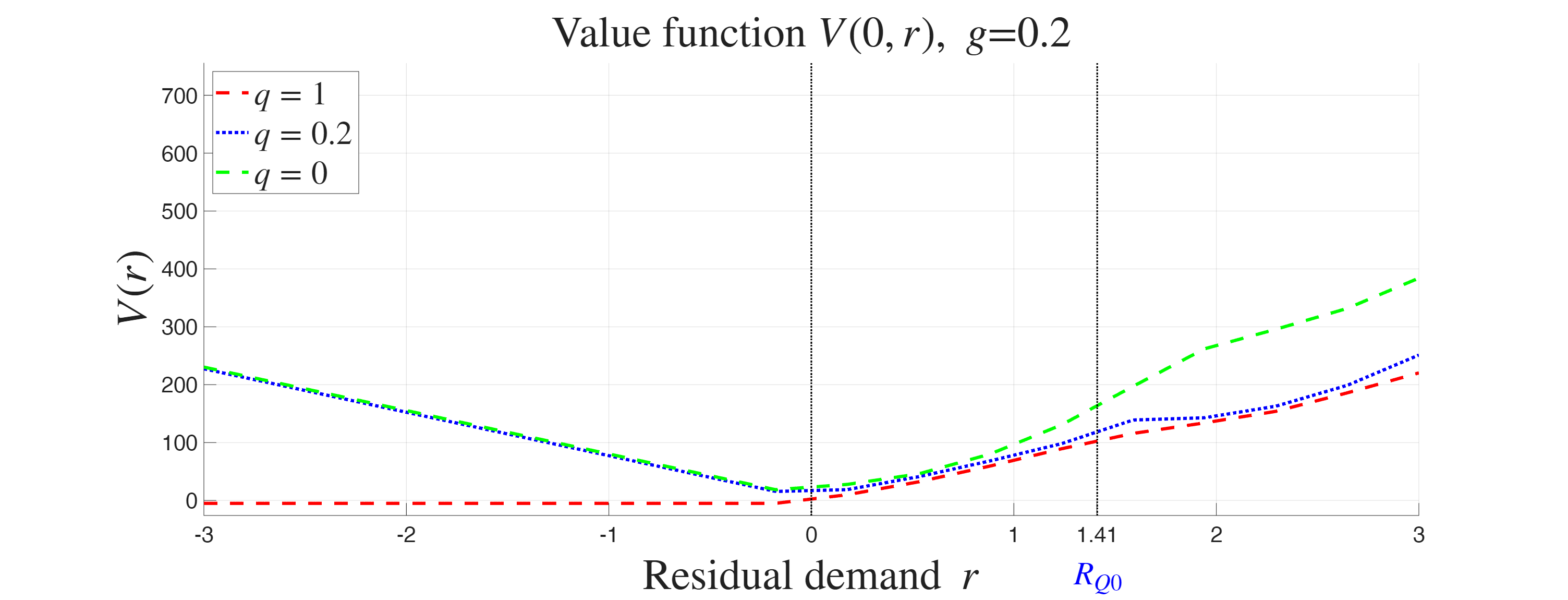}

\vspace*{0.3cm}
 \includegraphics[width=0.49\linewidth,height=0.215\linewidth]{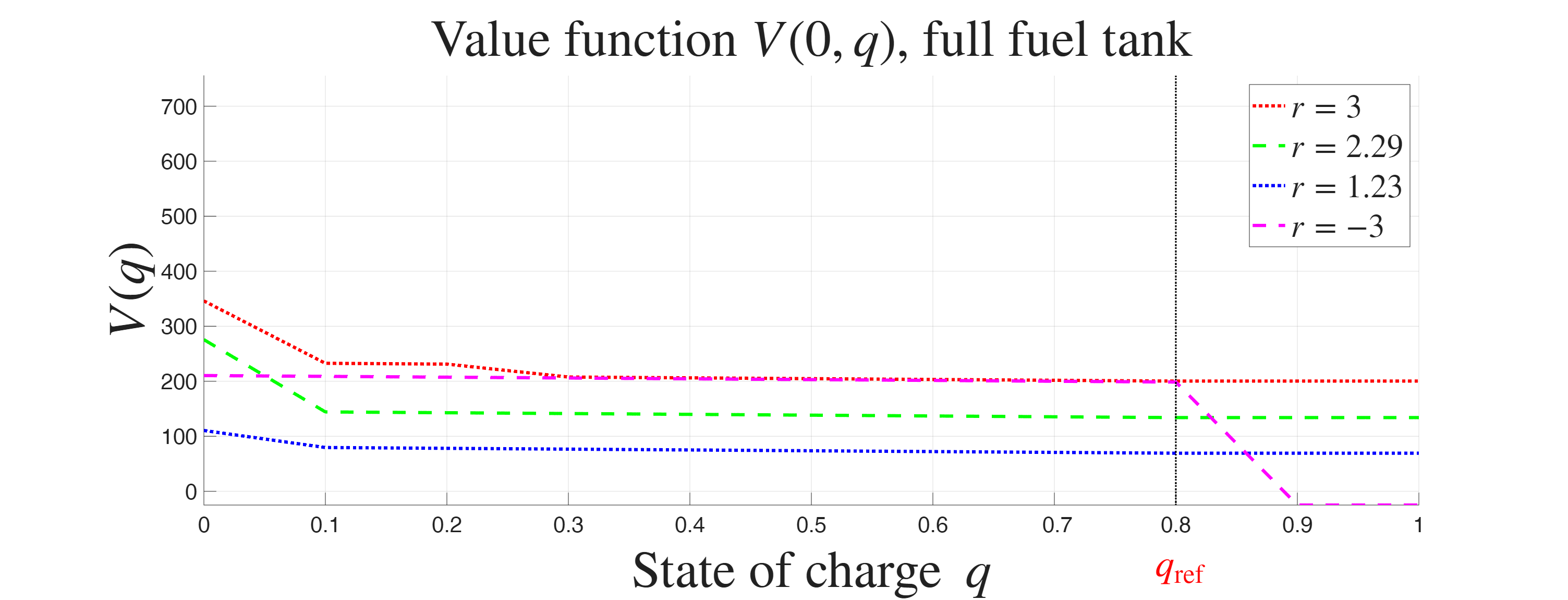}
\includegraphics[width=0.49\linewidth,height=0.215\linewidth]{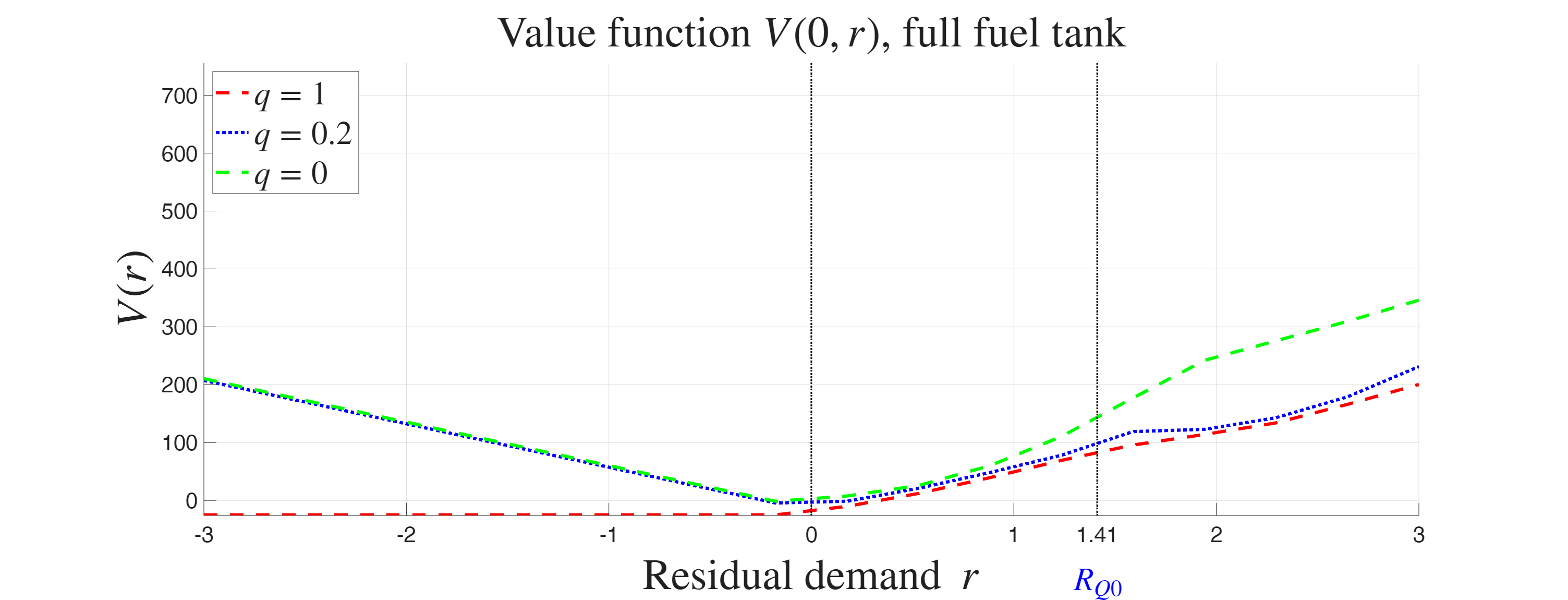} 
    \caption{Visualization of the value function at the initial time ($n=0$). 
    Left panel:  Value function as a function of the state of charge $q$ for a fixed residual demand \(r=\{-3, 1.23, 2.29, 3\}\).\\
    Right panel: Value function as a function of {\revi the}  residual demand $r$ for a fixed  state of charge \(q=\{0,20\%,100\%\}\).\\
    Upper panel: Empty fuel tank. Middle panel:  Fuel tank filled to 20$\%$ of its capacity. lower panel: Full fuel tank.}\label{fig_VFr-0}
\end{figure}

{\revi Figure~\ref{fig_Val_0} shows that the optimal decision rule at the initial time ($n=0$) has the same structure as at time $N-12$, with two differences. For $0<r<R_{Q0}$, the battery is discharged in full mode only if the state of charge exceeds a threshold that decreases with $r$ (from about $30\%$ to $17\%$), and it is optimal to wait otherwise, since the stored energy is valuable for the following periods. For $r>R_{Q0}$, the full mode is used for $r\gtrsim2.1$ and $q\gtrsim43\%$, whereas at $N-12$ it is used only for $q<10\%$; for $R_{Q0}\le r<1.8$ the battery is discharged in limited mode for $q>3\%$.} We now concentrate on visualizing the value function. In the left panel of Fig.~\ref{fig_VFr-0} we illustrate the value function as a function of the state of charge $q$ for a fixed residual demand \(r=\{-3, 1.23, 2.29, 3\}\), while the right panel shows the value function as a function of the residual demand $r$ for a fixed state of charge \(q=\{0,20\%,100\%\}\). {\revi The upper left panel of Figure~\ref{fig_Val_0} and Figure~\ref{fig_VFr-0} show that, in contrast to time $N-12$, the value function at the initial time takes smaller values (between $2.1$ and $21.4$\,\texteuro{}, against $-16.9$ and $26.3$\,\texteuro{} at $N-12$), although the costs of the remaining $168$ periods accumulate: at every step of the Bellman equation~\eqref{dyn} the future costs are discounted with $\rho=0.03$\,h$^{-1}$ (the mean discount factor over the week is $0.20$), and the liquidation value of the fuel, which separates the fuel levels at $N-12$, is discounted by $\mathrm{e}^{-\rho (N-1)\Delta_N}\approx0.007$. At $r=3$, $q=0.8$ the values for $g=0$, $0.2$, $1$ are therefore close to each other ($6.3$, $5.9$, and $5.8$\,\texteuro{}). Without discounting, the maximum would be $55.7$\,\texteuro{} (Table~\ref{Tab_rho_V}).} Similarly to times $n=N-1$ and $n=N-12$, the value function for the empty battery dominates. This is due to the quadratic discomfort cost for unmet demand. {\revi The right panel of Figure~\ref{fig_VFr-0} shows that the value function is almost flat for negative residual demand and increases linearly when the residual demand is positive.}

\subsection{Optimal Paths of the State Process}
This subsection illustrates the impact of the optimal decision rule on battery state of charge (SoC) and fuel level over a period of seven days. We consider four different scenarios with uncertainty on the residual demand. In addition, we assume in all scenarios that at the initial time (shortly after midnight) the fuel tank is full, the residual demand is at the maximum, and the state of charge is at $80\%$. We recall that the background colors represent the optimal decision rule, where dark blue indicates over-spilling, light blue the battery charging, green for waiting or \emph{doing nothing}, yellow for discharging the battery in limited mode, orange for discharging the  battery in full mode, and light red and dark red indicate using fuel in limited and full modes, respectively. In all scenarios, the magenta dashed line, the red solid line, and the blue solid line represent the fuel tank level, residual demand, and battery state of charge, respectively. The black dotted line at $R_{Q0}$ represents the threshold above which limited modes can be activated, the black dashed line at 0 corresponds to the zero level of residual demand. {\revi A favorable (adverse) day is modeled by shifting the level to which the deseasonalized residual demand reverts by $-1$\,kW ($+1$\,kW).}

By taking a closer look at the paths of the state process, illustrated through four distinct cases, we highlight, on the one hand, the correlation between residual demand and the battery state of charge for charge and discharge control actions and, on the other hand, the correlation between residual demand and fuel consumption in full or limited mode. For positive residual demand below the threshold ($0\leq r < R_{Q0}$), {\revi it is optimal to discharge the battery in full mode if its state of charge is high enough and to wait otherwise,} which corresponds to the decision rule stated above in Subsection~\ref{sec: Decision}. Figures~\ref{path1} to \ref{path4} show that the generator and the battery do not operate simultaneously. In fact, when the residual demand is negative, the generator is off; therefore, excess production is stored in the battery, and over-spilling is applied when the battery is full. When the residual demand exceeds the threshold ($r\geq R_{Q0}$) without being too high ($r<2.5$), battery discharge or fuel use in limited modes is preferred over full mode to extend battery or fuel lifespan, especially when we are far from terminal time. {\revi For high residual demand ($r\gtrsim2.1$) and a state of charge above about $0.4$, full mode is optimal, leading to sharp drops in the battery level.} In general, we observe that the generator is only activated when the battery is almost empty and the residual demand is above the threshold $R_{Q0}$.

\paragraph{Path of the state  process for favorable weather conditions during the first 3 days and adverse weather conditions during the remaining days}
\begin{figure}[htbp!]
    \centering
    \includegraphics[width=\linewidth]{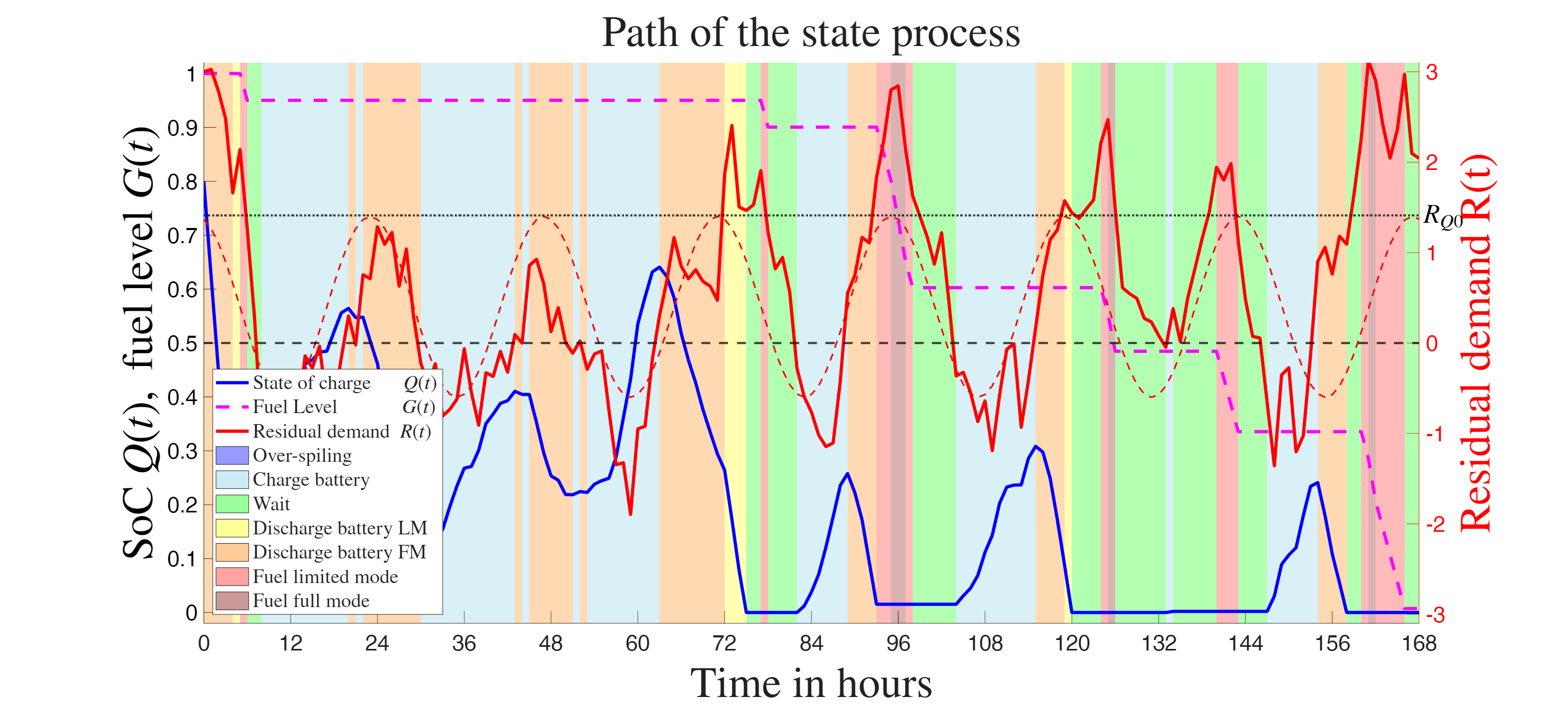}
    \caption{Path of the state process in response to optimal decision rule for favorable weather during the first 3 days and unfavorable weather during the remaining days.}\label{path1}
\end{figure}

{\revi Figure~\ref{path1} shows that, starting with a battery level of $80\%$, the battery is discharged during the first hours and then recharged by the excess production of the favorable days, with over-spilling when it is full; no fuel is used during the first four days. During the subsequent four days, the overproduction is not sufficient to recharge the battery: it is empty after about $96$ hours, and the generator is used in limited mode whenever the residual demand exceeds $R_{Q0}$. At the terminal time the battery is empty and about one fifth of the fuel remains in the tank.}

\paragraph{Path of the state for adverse weather except one day in the middle of the week and the last day}
\begin{figure}[htbp!]
    \centering
    \includegraphics[width=\linewidth]{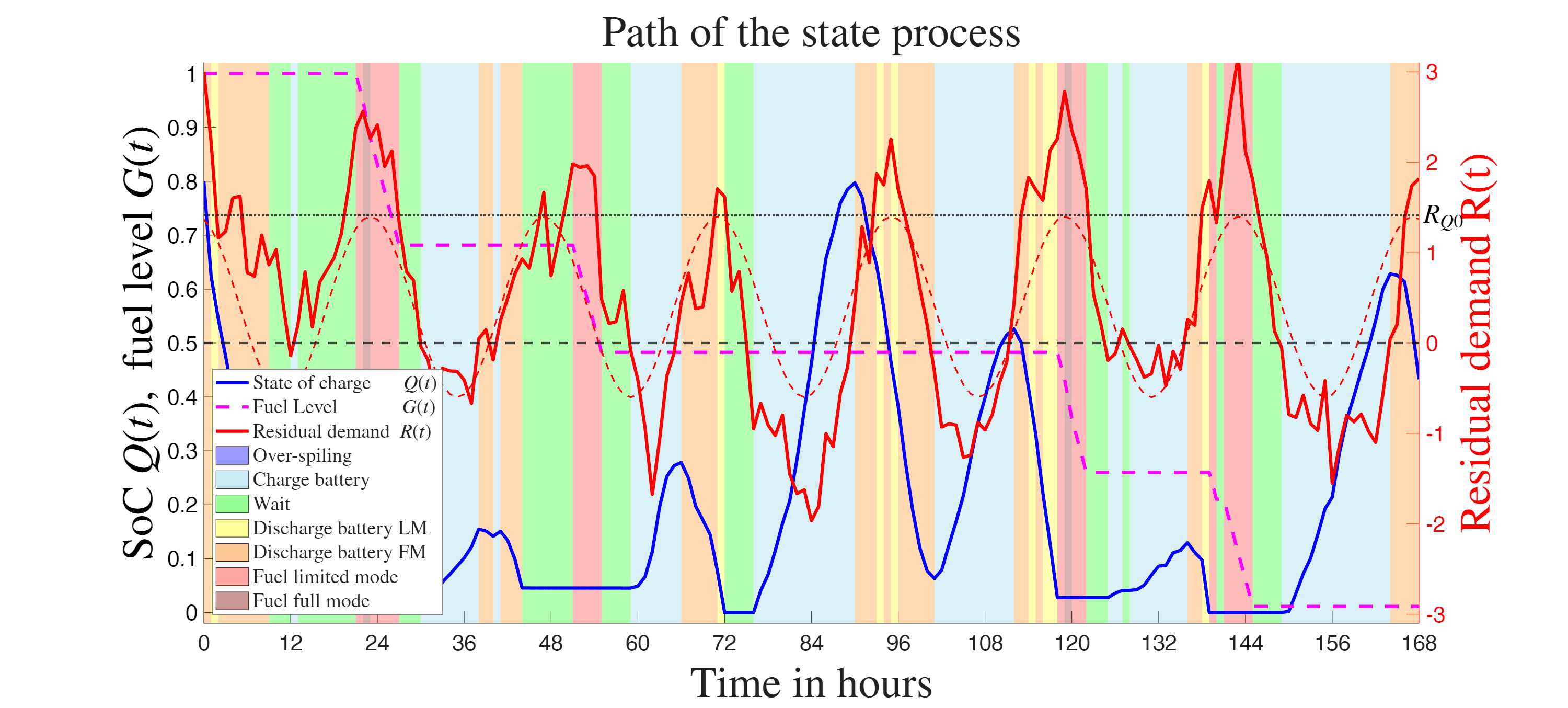}
    \caption{Path of the state process in response to optimal decision rule for adverse weather conditions, with the exception of two days (one midweek day and the last day).}\label{path2}
\end{figure}

{\revi Figure~\ref{path2} shows the state process over a week with adverse weather except on the fourth day and on the last day. During the first three days the battery is almost empty (below $5\%$) from about hour $7$, and the generator is used in limited or full mode whenever the residual demand exceeds $R_{Q0}$ (around hours $20$--$32$, $45$--$50$, and $69$--$72$), which empties the fuel tank after about $72$ hours. On the favorable fourth day, the overproduction recharges the battery to about $75\%$, and this energy is used during the following day; the battery is empty again from about hour $123$. On the last day, the overproduction recharges the battery to about $88\%$, above $q_{ref}$, so that no terminal penalty is paid.}\\

\paragraph{Path of the state process over 7 days for favorable
weather conditions on the last 3 days and adverse weather conditions on the remaining days}
\begin{figure}[htbp!]
    \centering
    \includegraphics[width=\linewidth]{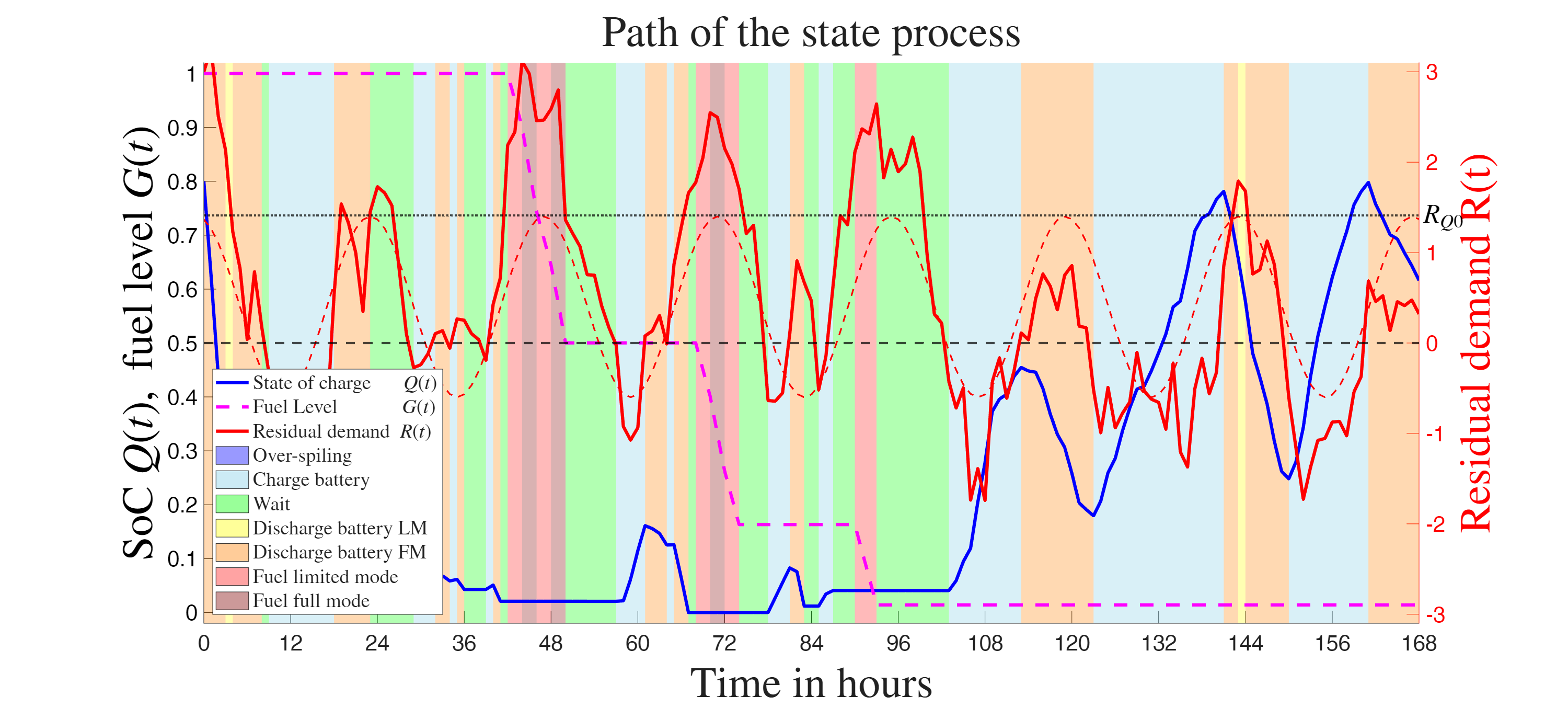}
    \caption{Path of the state process over 7 days in response to the optimal decision rule for favorable weather conditions on the last 3 days and adverse weather conditions on the remaining days.}\label{path3}
\end{figure}

{\revi Figure~\ref{path3} shows the state process over 7 days in response to the optimal decision rule for adverse weather conditions in the first four days and favorable weather conditions in the last 3 days. On the first 4 days, the excess production does not allow the battery level to exceed $12\%$; the battery is empty from about hour $7$ and the generator is used in limited mode whenever the residual demand exceeds $R_{Q0}$, which empties the fuel tank after about $92$ hours. On the last three days, the overproduction recharges the battery up to its full capacity, with over-spilling. The battery level at the terminal time is above $q_{ref}$, so that no terminal penalty is paid, compared to Fig.~\ref{path1}, where the battery is empty at the end of the simulation period.}

\paragraph{Path of the state process over 7 days for adverse
weather conditions throughout the week, with the exception of the last day}
\begin{figure}[htbp!]
    \centering
    \includegraphics[width=\linewidth]{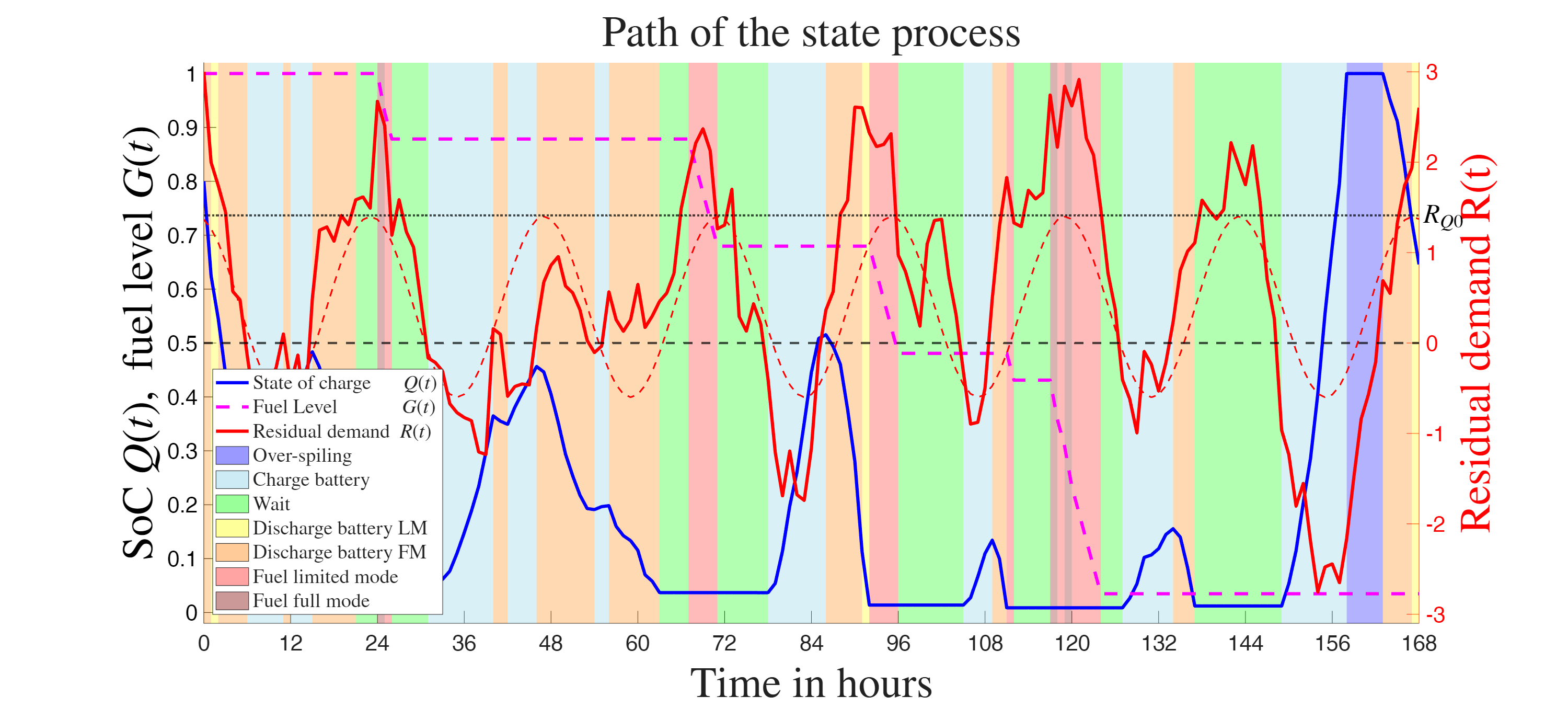}
    \caption{Path of the state process over 7 days in response to the optimal decision rule for adverse weather conditions throughout the week, with the exception of the last day.}\label{path4}
\end{figure}

{\revi Figure~\ref{path4} shows that, due to adverse weather conditions throughout the week, the battery level does not exceed $7\%$ after the first day; the battery is empty from about hour $7$ and the fuel tank is empty after about $50$ hours. Afterwards, it is optimal to wait and to accept unsatisfied demand. On the last day, the excess production recharges the battery to about $65\%$, so that it ends the week at $58\%$ and the terminal penalty is small.}

{\revi Table~\ref{Tab_ENS} reports the expected unsatisfied energy $E_u$ in~\eqref{EENS} and its percentage~\eqref{EENS_perc}, together with the expected cost, for the four scenarios of Figures~\ref{path1}--\ref{path4} ($300$ paths per scenario). The unsatisfied energy increases with the number of adverse days. The costs are discounted to the initial time with the rate $\rho$; the cost is lowest when the favorable days come first, because the adverse days are then discounted more.

\begin{table}[!ht]
	\centering\small
	\caption{Expected unsatisfied energy $E_u$ and expected cost under the optimal policy for the four weather scenarios of Figures~\ref{path1}--\ref{path4} ($300$ paths per scenario).}
	\label{Tab_ENS}
	\resizebox{\linewidth}{!}{\begin{tabular}{lccc}
		\toprule
		Scenario & $E_u$ [kWh] & $E_u\%$ & Cost [\texteuro] \\
		\midrule
		Favorable (days 1--3) / adverse (4--7) & $67.0$ & $52.1\%$ & $3.3$ \\
		Adverse except day 4 and day 7 & $98.4$ & $60.4\%$ & $21.8$ \\
		Adverse (days 1--4) / favorable (5--7) & $81.8$ & $59.3\%$ & $22.8$ \\
		Adverse throughout (except last day) & $141.6$ & $72.4\%$ & $24.7$ \\
		\bottomrule
	\end{tabular}}
\end{table}
}

{\revi
\subsection{Comparison with Benchmark Strategies and Discretization Error}
\label{Subsec_bench_conv}
\paragraph{Comparison with rule-based strategies}
We compare the optimal decision rule with four rule-based strategies that respect the feasible sets of Figure~\ref{Fig_actionspace}. \emph{Battery-first} charges the battery for $r<0$; for $r\ge 0$ it discharges the battery in full mode until it is empty, then fires the generator in full mode until the tank is empty, and waits otherwise. \emph{Generator-first} uses the generator before the battery. The third strategy is battery-first with the limited modes for $r\ge R_{Q0}$ (resp.\ $r\ge R_{G0}$), and the fourth is battery-first with a fuel reserve of $90\%$ ($50\%$): the generator is used only while the tank level exceeds the reserve. All strategies are simulated with $1000$ paths over seven days from $Z(0)=0$, $Q(0)=0.8$, $G(0)=1$, and evaluated with the expected cost of the performance criterion~\eqref{performance_D}, discounted to the initial time with the rate $\rho$, and with the expected unsatisfied energy $E_u$, Equations~\eqref{EENS}--\eqref{EENS_perc}. Table~\ref{Tab_bench} reports, for each strategy, the decomposition of the expected cost into its components and the metrics derived from the simulated paths.
\begin{table}[!ht]
	\centering\footnotesize\setlength{\tabcolsep}{4pt}
	\caption{Expected cost of the optimal decision rule and of rule-based strategies [\texteuro], and the resulting metrics (seven days, $1000$ paths). The total cost is the sum of the fuel cost, the battery degradation cost, the discomfort cost, the terminal penalty, and the liquidation value of the remaining fuel (negative); the standard error of the total cost is between $0.08$ and $0.11$\,\texteuro{}. The metrics are the expected unsatisfied energy $E_u$ in kWh and as a percentage of the demand, and the mean final state of charge $Q(T)$ and fuel level $G(T)$.}
	\label{Tab_bench}
	\resizebox{\linewidth}{!}{\begin{tabular}{lrrrrrrrrrr}
		\toprule
		 & \multicolumn{6}{c}{Cost [\texteuro]} & \multicolumn{4}{c}{Metrics}\\
		\cmidrule(lr){2-7}\cmidrule(lr){8-11}
		Strategy & Fuel & Degrad. & Discomf. & Penalty & Liquid. & Total & $E_u$ [kWh] & $E_u\%$ & $Q(T)$ & $G(T)$ \\
		\midrule
		Optimal (MDP) & $0.55$ & $1.06$ & $2.18$ & $0.04$ & $-0.14$ & $3.70$ & $35.2$ & $40.6$ & $0.42$ & $0.84$ \\
		Battery-first & $4.05$ & $1.10$ & $0.44$ & $0.06$ & $-0.02$ & $5.62$ & $14.5$ & $16.8$ & $0.25$ & $0.15$ \\
		Generator-first & $16.70$ & $0.43$ & $0.94$ & $0.06$ & $-0.00$ & $18.13$ & $24.9$ & $28.9$ & $0.25$ & $0.03$ \\
		Battery-first, economical modes & $3.47$ & $1.07$ & $0.90$ & $0.05$ & $-0.03$ & $5.47$ & $18.8$ & $22.0$ & $0.26$ & $0.20$ \\
		Battery-first, fuel reserve $90\%$ & $0.60$ & $1.08$ & $2.92$ & $0.06$ & $-0.15$ & $4.52$ & $34.3$ & $39.7$ & $0.24$ & $0.93$ \\
		Battery-first, fuel reserve $50\%$ & $2.96$ & $1.09$ & $1.24$ & $0.06$ & $-0.09$ & $5.26$ & $25.2$ & $29.2$ & $0.25$ & $0.55$ \\
		\bottomrule
	\end{tabular}}
\end{table}
The optimal decision rule has the smallest expected cost: it is $32$--$34\%$ below the battery-first strategies, $80\%$ below generator-first, and $18\%$ below the fuel-reserve strategy with $90\%$, which ends the week with the same fuel level. Since the costs are discounted with $\rho=0.03$\,h$^{-1}$, the terminal penalty and the liquidation value of the fuel, discounted by $\mathrm{e}^{-\rho (N-1)\Delta_N}\approx0.007$, are small, and the cost is determined mainly by the first days. The battery-first strategies burn $3.5$--$4.1$\,\texteuro{} of fuel against $0.55$\,\texteuro{}, and the fuel-reserve strategy pays a higher discomfort cost ($2.9$ against $2.2$\,\texteuro{}). The saving comes with a larger unsatisfied energy ($40.6\%$ against $16.8\%$ for battery-first): for small residual demand the discomfort cost is lower than the cost of replacing the discharged energy or of burning fuel, so that accepting a partial loss of service is cost-optimal (Figure~\ref{fig_Val_0}, $0<r<R_{Q0}$). The comparison is restricted to the strategies above and to the OU demand model.

\paragraph{Sensitivity to the discretization}
Table~\ref{Tab_conv} reports the effect of changing one discretization parameter at a time with respect to the baseline grid ($N_R=17$, $N_Q=15$, $N_G=10$, $\Delta_N=1$\,h). The reference is the fine grid ($N_R=33$, $N_Q=30$, $N_G=20$, $\Delta_N=1$\,h), and $R_{Q0}=R_{G0}$ is kept fixed, so that all grids solve the same model. $V(0,x_0)$ is the value at $r=3$, $q=0.8$, $g=1$ (discounted to the initial time); the relative error is the largest difference to the reference at time $0$, divided by the largest absolute value of the reference, on the points $r\in\{-3,-2,-1,1,2,3\}$ and $q,g\in\{0,0.2,\dots,1\}$, and the agreement is the fraction of these points at which the optimal actions coincide.
\begin{table}[!ht]
	\centering
	\caption{Sensitivity of the value function and of the decision rule to the discretization; $V(0,x_0)$ is the value at $r=3$, $q=0.8$, $g=1$.}
	\label{Tab_conv}
	\resizebox{\linewidth}{!}{\begin{tabular}{ccccrrrrr}
		\toprule
		$N_R$ & $N_Q$ & $N_G$ & $\Delta_N$ [h] & $V(0,x_0)$ & $\max V(0)$ & $\max V(N-12)$ & rel.\ error & agreement \\
		\midrule
		$9$  & $15$ & $10$ & $1$   & $6.1$  & $21.9$ & $26.6$ & $5.3\,\%$  & $93\,\%$ \\
		$17$ & $15$ & $10$ & $1$   & $5.8$  & $21.3$ & $26.3$ & $7.5\,\%$ & $93\,\%$ \\
		$33$ & $15$ & $10$ & $1$   & $5.7$  & $21.2$ & $26.3$ & $8.1\,\%$ & $95\,\%$ \\
		$17$ & $5$  & $10$ & $1$   & $2.4$ & $24.0$ & $26.4$ & $39\,\%$   & $86\,\%$ \\
		$17$ & $30$ & $10$ & $1$   & $6.1$  & $21.3$ & $26.3$ & $3.8\,\%$  & $99.5\,\%$ \\
		$17$ & $15$ & $5$  & $1$   & $5.8$  & $21.3$ & $26.3$ & $7.5\,\%$ & $93\,\%$ \\
		$17$ & $15$ & $20$ & $1$   & $5.8$  & $21.3$ & $26.3$ & $7.5\,\%$  & $93\,\%$ \\
		$17$ & $15$ & $10$ & $2$   & $6.2$  & $22.2$ & $25.2$ & $5.0\,\%$  & $84\,\%$ \\
		$17$ & $15$ & $10$ & $0.5$ & $7.1$  & $22.5$ & $26.9$ & $12\,\%$   & $83\,\%$ \\
		\midrule
		$33$ & $30$ & $20$ & $1$   & $6.0$  & $21.1$ & $26.3$ & --         & --       \\
		\bottomrule
	\end{tabular}}
\end{table}
The extreme values $\max V(0)$ and $\max V(N-12)$ differ from those of the reference by less than $4\%$ for $N_Q\ge 15$ and $\Delta_N=1$\,h ($4.9\%$ and $4.0\%$ for $\Delta_N=2$\,h), so that the value functions in Figures~\ref{fig_Val_n-1}--\ref{fig_VFq-N-12} do not depend on the grid at the scale of the plots. The state of charge is the most influential variable: $N_Q=5$ gives an error of $39\%$, whereas the changes are moderate for $N_Q\ge 15$. The error of the baseline grid is $7.5\%$, and at $x_0$ the baseline value is $0.24$\,\texteuro{} ($1.1\%$ of the largest value) below the reference. The optimal actions agree with the reference at $83$--$99.5\%$ of the points ($93\%$ for the baseline grid); the differences are located next to the switching curves, so that the structure of the decision regions in Figures~\ref{fig_Val_n-1}--\ref{fig_Val_0} is preserved. We do not observe convergence in the time step at fixed state grid (error $12\%$ and an increase of $V(0,x_0)$ by $18\%$ for $\Delta_N=0.5$\,h); the time step has to be refined together with $N_Q$ and $N_G$, which is left for future work.

\paragraph{Sensitivity to the discount rate}
Table~\ref{Tab_rho_V} shows the value function for $\rho\in\{0,\,\allowbreak 0.003,\,\allowbreak 0.01,\,\allowbreak 0.03,\,\allowbreak 0.1\}$\,h$^{-1}$ ($\rho=0.03$ is the value of the paper and $\rho=0$ the undiscounted problem, in which $\zeta_1=\Delta_N$); all other data and the grid are unchanged. One hour before the terminal time the value function hardly depends on $\rho$ ($\max V$ decreases from $17.0$ to $15.6$), because only one step is discounted; the decrease grows with the remaining horizon, to $32.9\to17.0$ twelve hours before the terminal time and $55.7\to21.4\to13.8$ at the initial time for $\rho=0$, $0.03$, $0.1$ (for $\rho=0.03$ the mean discount factor over the week is $0.20$). The value at a given state is not monotone in $\rho$: the running costs are positive, whereas the liquidation value of the remaining fuel is negative and is discounted more strongly. For $\rho=0.03$ and $\rho=0.1$ the simulated cost of the decision rule agrees with the value of the recursion ($3.70\pm0.08$ against $3.61$, and $0.63\pm0.01$ against $0.59$), which confirms that the Bellman equation~\eqref{dyn} and the criterion~\eqref{performance_D} use the same discounting; for $\rho\le0.01$ it is larger ($11.0$ against $7.1$ for $\rho=0$), because the discretization error accumulates when the costs are not discounted. A larger $\rho$ makes the manager more short-sighted: the unsatisfied energy decreases from $47.4\%$ to $37.1\%$, and at the initial time $86$--$93\%$ of the optimal actions coincide with those for $\rho=0.03$.
\begin{table}[!ht]
	\centering\footnotesize
	\caption{Sensitivity of the value function to the discount rate $\rho$ [h$^{-1}$] (baseline grid, $T=7$ days; $\rho=0.03$ is the value of the paper): maximum and minimum of $V(n,\cdot)$ over the grid and the value at $r=3$, $q=0.8$, $g=1$, one hour before the terminal time ($n=N-1$), twelve hours before ($n=N-12$), and at the initial time ($n=0$).}
	\label{Tab_rho_V}
	\resizebox{\linewidth}{!}{\begin{tabular}{lrrrrrrrrr}
		\toprule
		 & \multicolumn{3}{c}{$n=N-1$} & \multicolumn{3}{c}{$n=N-12$} & \multicolumn{3}{c}{$n=0$}\\
		\cmidrule(lr){2-4}\cmidrule(lr){5-7}\cmidrule(lr){8-10}
		$\rho$ & $\max V$ & $\min V$ & $V(3,0.8,1)$ & $\max V$ & $\min V$ & $V(3,0.8,1)$ & $\max V$ & $\min V$ & $V(3,0.8,1)$ \\
		\midrule
$0$ & $17.03$ & $-25.00$ & $-22.60$ & $32.87$ & $-24.31$ & $-11.77$ & $55.71$ & $3.82$ & $11.16$ \\
$0.003$ & $16.99$ & $-24.93$ & $-22.53$ & $32.12$ & $-23.45$ & $-11.25$ & $46.47$ & $5.77$ & $12.53$ \\
$0.01$ & $16.89$ & $-24.75$ & $-22.37$ & $30.47$ & $-21.54$ & $-10.12$ & $33.66$ & $5.79$ & $11.46$ \\
$0.03$ (paper) & $16.60$ & $-24.26$ & $-21.91$ & $26.34$ & $-16.92$ & $-7.41$ & $21.35$ & $2.10$ & $5.78$ \\
$0.1$ & $15.64$ & $-22.62$ & $-20.38$ & $16.96$ & $-7.27$ & $-2.52$ & $13.82$ & $0.11$ & $1.42$ \\
		\bottomrule
	\end{tabular}}
\end{table}
}

\section{Conclusion }
\label{conclus}
This paper investigated stochastic optimal control for the cost-optimal management of standalone microgrids, particularly relevant for applications in remote areas. Addressing the inherent uncertainties associated with renewable energy production and fluctuating demand, the microgrid operation was formulated as a Markov decision process. 
We solved this problem using the discrete-time dynamic programming approach and derived the associated Bellman equation. Further, a continuous-state Markov decision process is approximated by a Markov decision process for a finite-state Markov chain for which we characterized the transition probabilities.  In this way, the curse of dimensionality is overcome and the optimal control problem is efficiently solved by numerical computation via backward recursion. 
The first contribution of this work is a systematic approach to managing operational uncertainties and resource limitations within standalone microgrids. The numerical results provide clear guidance for making decisions about the operation of the batteries and generators within the system. {\revi The battery is charged from every surplus, it is discharged in full mode only above a state-of-charge threshold, the limited modes are used for residual demand above $R_{Q0}$, and the generator is used only when the battery is nearly empty. With the discount rate $\rho=0.03$\,h$^{-1}$, the optimal decision rule has an expected cost of $3.7$\,\texteuro{} over a week, $18\%$ to $80\%$ below the rule-based strategies, at the price of a larger unsatisfied energy ($40.6\%$ against $16.8\%$ for battery-first). The discount rate and the discretization change the value function but not the structure of the decision rule.} Finally, these findings offer practical insights for the design and operation of microgrid systems, particularly in remote areas where access to the grid remains a challenge.
The successful implementation of this framework is crucial to the advancement of rural electrification, especially in developing countries.

This work was carried out under the assumption that the economical mode can only be activated when the residual demand is above a certain threshold. However, this assumption can be relaxed, allowing the economical mode to be activated at any time. In this case, to derive the MDP, one has to find a good approximation of the minimum between the residual demand whose dynamics is governed by an SDE and a certain target. 
The residual demand can be divided into two parts: supply and demand.  Furthermore, the energy supply generated by the solar panel can be modeled as a stochastic process whose drift is governed by a finite-state Markov chain. The states of this Markov chain correspond to the states of cloudiness. This will increase the dimension of the state space, which leads to the curse of dimensionality, and the numerical solution via Markov chain becomes intractable. Therefore, alternative optimization techniques for the numerical solution are necessary, including approximate dynamic programming, optimal quantization, model predictive control, and reinforcement learning. Finally,  we can also consider a more realistic battery model, taking into account its degradation, voltage, and the evolution of its capacity over time. This will result in a more complex model whose solution required sophisticated techniques. {\revi In particular, battery aging could be represented by the storage capacity $C_Q$ as an additional, slowly evolving state variable. Extending to non-Gaussian demand, for which the closed-form transition operator is lost, is left for future work.}

{\revi A further extension is to relax Assumption~\ref{ass}(3) and allow concurrent operation of the battery and the generator, such as cycle-charging, which would require a trivariate Gaussian transition (Remark~\ref{Rem-Tkernel}) and an enlarged control set.}  
 
\begin{appendix}

\section{Nomenclature}
\begin{longtable}{p{0.3\textwidth}p{0.68\textwidth}l}
$T$ & Finite time horizon&\\
$R,Z$ & Residual demand, deseasonalized residual demand &\\
\(\beta_R\)& Mean reversion speed for residual demand\\
\(\sigma_R\)& Volatility for the residual demand&\\
$\mu_R $ & Seasonality function &\\
$\mu_0^R$ & Long-term mean of seasonality function&\\
$\kappa_1^R, \kappa_2^R$ & Amplitude constants for yearly, daily seasonality, &\\
$t_1^R, t_2^R$ & Reference time component for yearly,  daily seasonality&\\
$\delta_1, \delta_2$ & Length of yearly,  daily seasonal period  &\\
$Q,G$ & State of charge of the battery, fuel tank level &\\
\(C_Q, C_G\)& Battery, fuel tank capacity&\\
\(\eta_E,~\eta_E^C,~\eta_E^D\)& {\revi Charging/discharging efficiencies} of the battery &\\
$F_0$ & Fixed fuel price &\\
 $c_0,c_1$ & Idle and load dependent consumption rate &\\
$X=(Z,Q,G)$ & Continuous state variable &\\
$(\Omega,\mathcal{F},\mathbb{F},\mathbb{P})$ & Filtered probability space for stochastic processes &\\
$W$ & Wiener process&\\
$\Omega$ & Sample space&\\
$\mathbb{F}=(\mathcal{F})_{t\in[0,T]}$ & Filtration generated by the Wiener process $(W(t))_{t\in[0,T]}$ &\\
$\mathbb{P} $ & Probability measure on a measurable space on \((\Omega,\mathcal{F})\)&\\
\(\stZ\)& Continuous state space of $Z$&\\
\(\mathcal{Q},\mathcal{G}\)& Continuous state space of state of charge and fuel tank&\\
$\mathcal{X}=\stZ\times \stQ \times \stG$ & Continuous state space&\\
$u^W,u^O$ & Wait/do nothing, overspilling &\\
$u^C$ & Charge the battery  &\\
$u^D,u^{DL}$ & Discharge the battery in full, limited mode&\\
$u^F,u^{FL}$ & Generator in full, limited mode&\\
$u=(u(t))_{t\in[0,T]}$ & Continuous-time control process&\\
  \( \controlset \)& Set of feasible control &\\

\(\mathcal{K}(X)\)& State constraint set&\\
  \(J,J^D\) & Continuous and discrete time performance criterion &\\
\(\Psi,\phi,\Psi^D,\phi^D\)& Continuous and discrete running time, terminal cost&\\
 $\gamma_{degr}$ &  Degradation cost per unit of residual demand &\\
$\gamma_{pen}^Q$ & Penalty cost for battery at terminal time &\\
{\revi $q_{ref}$} & {\revi Reference state of charge at the terminal time}&\\
{\revi $\Delta E_n, E_u$} & {\revi Energy not served in step $n$, expected unsatisfied energy}&\\
$\gamma_{liq}^Q, \gamma_{liq}^G$ & Liquidation price for battery, fuel tank at terminal time&\\
  \(t_n, N, \Delta_N \)&  Discrete time points with $N$ number of time steps and $\Delta_N $ step size&\\
\(\desresdemand_n = Z(t_n) \in \stZ\)& {\revi Deseasonalized} residual demand sampled at discrete-time \(n\)&\\
\(\SoC_n = \SoC(t_n) \in \stQ\) & State of charge  sampled at discrete-time \(n\)&\\
\(\Fuellevel_n = \Fuellevel(t_n) \in \stG\) & Fuel tank level sampled  at discrete-time \(n\)&\\
\(\State_n = (\desresdemand_n, \SoC_n, \Fuellevel_n)\)&  State process sampled  at discrete-time \(n\)&\\
\(\mathcal{N}_\dagger\)& Set of indices for  \(\dagger=Z,Q,G\)&\\
\(\overline{\stZ}=[\underline{z},\overline{z}]\)& Continuous state space for $Z$ truncated at $[\underline{z},\overline{z}]$ &\\
\(\discstZ\)& Discrete state space of the deseasonalized residual demand&\\
\(\discstQ,\discstG\)& Discrete state space of the state of charge and fuel level&\\
$\widetilde{\mathcal{X}}=\discstZ \times \discstQ \times \discstG$ & Discrete state space of the state variable&\\
\(u(t_n)=:\alpha_n\) & Constant control between two consecutive time points &\\
\( \alpha = (\alpha_1,\cdots \alpha_{N-1})\) & Discrete time control process &\\
\(m_R(n), \meanZ\)& Conditional mean of $R,~Z$ at time $n$&\\
\(\meanQ,\meanG\)&Conditional mean of the state of charge and the fuel level&\\
\(\Sigma_\dagger^2\)& Conditional variance of \(\dagger=Z,Q,G\)&\\
\(\Sigma_\dagger\) & Standard deviation of  \(\dagger=Z,Q,G\)&\\
\(\CovZQ,\CorrZQ\) & Conditional covariance and correlation coefficient of \(Z_{n+1}\) and \(Q_{n+1}\)&\\
\(\CovZG,\CorrZG\) & Conditional covariance and correlation coefficient of \(Z_{n+1}\) and \(G_{n+1}\)&\\
  \(\mathcal{T}_n=(\mathcal{T^Z}_n,\mathcal{T^Q}_n,\mathcal{T^G}_n)\)& Transition operator  at time \(n\)&\\
    \( \mathcal{T}^\dagger_n \)& Transition operator of \(\dagger\) at time \(n\), with \(\dagger=Z,Q,G\)&\\
\(\mathcal{E} = (\mathcal{E}_n)_{n=1, \dots, N}\)& Sequence of standard normal random vectors in \(\mathbb{R}^3\)&\\
\( \mathcal{G}_n = \sigma(\mathcal{E}_1, \dots, \mathcal{E}_n) \)&
Sigma-algebra generated by the first \(\mathcal{E}_1, \mathcal{E}_2, \dots, \mathcal{E}_n\)&\\
\(\mathbb{G} = (\mathcal{G}_n)_{n=0, 1, \dots, N}\)& Discrete-time filtration with \(\mathcal{G}_0 = \{\emptyset, \Omega\}\) a trivial sigma-algebra&\\
\(\mathcal{U}_\dagger(n,x)\)& Discrete-time set of feasible actions related to \(\dagger=Q,G\)&\\
\(\mathcal{U}(n,x)=\consQ(n,x){\revi \cap}\consG(n,x)\)& Discrete-time state-dependent set of feasible controls &\\
 \(\admiss\)& Set of admissible controls &\\
\(V\)& Discrete time value function&\\
\(\mathcal{P}^a_{x_{m_1},x_{m_2}}\)& Transition probability from state \(x_{m_1}\) to \(x_{m_2}\) under action $a$\\[2ex]
 \textbf{Abbreviations}&&\\[1ex]
 ODE & Ordinary differential equation&\\
 SDE & Stochastic differential equation\\
 MDP & Markov decision process&\\
 SoC& State of charge of the battery&\\
 {\revi OU} & {\revi Ornstein--Uhlenbeck}&\\
 LM & Limited mode&\\
 FM & Full mode&\\
 BLM & Battery in limited mode&\\
 BFM &Battery in full mode&
\end{longtable}

 \section{Calibration of Parameters}
\subsection{Self-discharging Rate of the Battery }
\label{Calib_self}
In idle mode, that is, when we neither charge nor discharge, the dynamics of the battery's SoC is mainly influenced by the self-discharging parameter. To calibrate this parameter, we assume that a fully charged battery loses its capacity after a long period of inactivity.  We assume that at time $t_0$ the SoC is \(Q(t_0) = 1\), and after a period of inactivity $t_1$, the battery's SoC is now \(Q(t_1)=q^*\). In idle mode, the change in the SoC is described by the following ODE:
\[dQ(s) = -\eta_0Q(s)\diff s.\]
  Solving this under conditions 
\(Q(t_0)=1\) and \(Q(t_1)=q^*\) leads to 
\(\eta_0 =- \frac{\ln{(q^*)}}{t_1-t_0}.\)\\
For numerical results, we set the initial time (\(t_0 = 0\)) and assume that after 4-days (\(t_1 = 4 \times 24 \)) hours, the state of charge has dropped by 2\% (\(q^* = 0.98\)) due to self{\revi -}discharge. Then, a straightforward calculation gives \(\eta_0 =  0.0002104 ~[\text{ h}^{-1}]\).
\subsection{Battery's Capacity }
\label{Calib_batC}
A battery is a key component of our energy system. Its capacity \(C_Q\) must be calibrated so that it can meet the energy demand overnight, that is, from time \(t^D_1\) to time $t^D_2$. In addition, we ensure that the battery capacity is large enough to store overproduction throughout the day, that is, from time \(t^C_1\) to time $t^C_2$.  Therefore, the battery capacity must exceed the total cumulative overproduction during the day, given by \(I_C = \displaystyle\int_{t^C_1}^{t^C_2}\eta_E^C(\mu_R(s)+Z(s))ds\), and the total cumulative unmet demand during the night, given by \(I_D = \displaystyle \int_{t^D_1}^{t^D_{2}}\frac{1}{\eta_E^D}(\mu_R(s)+Z(s))\diff s\), with high probability. Here, \(\eta_E^C\) and \(\eta_E^D\) are the charging and discharge efficiencies to account for energy losses while charging and discharging, respectively.
Now, let \(p\) be a predefined confidence level.
Then, the capacity $C_Q$ is chosen so that \(\mathbb{P}(I_C\le C_Q) = p\) and \(\mathbb{P}(I_D\le C_Q) = p\). 
Since the deseasonalized residual demand \(Z(s)\)  is an OU process, then, the integral \(I_\dagger,~\dagger=C,D\) is a Gaussian random variable, that is, \(I_\dagger\sim \mathcal{N}(\mu^{}_{I_\dagger},\Sigma_{I_\dagger}^2)\), where
\begin{align*}
\mu^{}_{I_\dagger} &= \mu_0^R (t^\dagger_2 - t^\dagger_1) + \frac{\kappa_1^{R}\delta_1}{2\pi} \left[ \sin\left(\frac{2\pi(t^\dagger_{2}-t_1^R)}{\delta_1}\right) - \sin\left(\frac{2\pi(t^\dagger_1-t_1^R)}{\delta_1}\right) \right] \\
&\quad + \frac{\kappa_2^{R}\delta_2}{2\pi} \left[ \sin\left(\frac{2\pi(t^\dagger_{2}-t_2^R)}{\delta_2}\right) - \sin\left(\frac{2\pi(t^\dagger_1-t_2^R)}{\delta_2}\right) \right]  + \frac{z_1}{\beta_R} \left(1 - \mathrm{e}^{-\beta_R(t^\dagger_{2}-t^\dagger_1)}\right)
\\
\Sigma_{I_\dagger}^2 &= \frac{\sigma_R^2}{2\beta_R^3} \left(2\beta_R (t^\dagger_{2} - t^\dagger_1) - 3 + 4\mathrm{e}^{-\beta_R (t^\dagger_{2}-t^\dagger_1)} - \mathrm{e}^{-2\beta_R (t^\dagger_{2}-t^\dagger_1)}\right).
\end{align*}

Given  that \(I_\dagger\) is normally distributed, the condition \(\mathbb{P}(I_\dagger\le C_Q) = p\) leads to the following: 
\[ C_{Q,\dagger} = \eta^\dagger (\mu^{}_{I_\dagger} + z_p \Sigma^{}_{I_\dagger}) \quad \text{ with  } \quad \eta^\dagger=\begin{cases}
   \eta_E^C ,\qquad \dagger=C,\\
    \frac{1}{\eta_E^D}, \qquad \dagger=D,
\end{cases}\]
 where \(z_p =\Phi^{-1}(p) \) with \(\Phi\) the cumulative distribution function of the normal random variable.\\
Finally, a battery capacity must not be too large or too small, but must be such that it can store all overproduction throughout the day and can survive all night. Therefore, the minimal  battery capacity required is then given by
\[\capacitybat = \max(C_{Q,C},C_{Q,D}).\]

\subsection{Calibration of Fuel Tank Parameters}
\label{calib_fuelC}
Recall that the fuel consumption of a generator in idle mode or on load varies significantly based on the generator's specific model, engine type, engine size, fuel type, and efficiency. For example, a 3000-watt generator is a generator ( small power station) that produces a continuous output of 3000 watts (3 kW) of power when operating in full mode. It can run devices and appliances requiring a maximum of 3 kW of electrical power. Note that such a  generator uses approximately 0.3 to 0.5 liters  of diesel fuel per kilowatt-hour (kWh). In idle mode, a medium-sized diesel generator may consume approximately 0.5 to 1 liter of fuel per hour. For example, a 3 kW diesel generator's fuel consumption at idle is typically around 0.75 to 1.0 liters per hour. Therefore, $c_0 \in [0.5,1]$ and $c_1 \in [0.3,0.5]$

\subsection{Calibration of the Degradation Cost} 
\label{Calib_deg}
Let $T_0$ be the battery's life span, that is, the time it will need to be replaced in our energy system. We assume that the battery stores the surplus energy produced during the day and helps to meet demand at night. Then, the degradation cost $\gamma_{\text{degr}}$ should be calibrated such that the discounted cumulative cost during the operational period $[0,T_0]$ is equal to the discounted future price of the new battery, indicated by $P_b$. We consider the worst-case scenario, where we satisfy the maximum demand or store the maximum overproduction throughout the operational period. Therefore, $\gamma_{\text{degr}}$ satisfies:
\[\int_0^{T_0}\mathrm{e}^{-\rho t}\gamma_{\text{degr}}\underset{t\in [0,T_0]}{\text{max}}|R(t)|dt=\mathrm{e}^{-\rho T_0}P_b.\]
Thus,
\[\gamma_{\text{degr}}= \frac{\rho P_b\mathrm{e}^{-\rho T_0}}{(1-\mathrm{e}^{-\rho T_0})\underset{t\in [0,T_0]}{\text{max}}|R(t)|}.\]
\section{Time-Discretization Details}
\subsection{Proof of Lemma \ref{lem-Z}} 
\begin{proof}\label{proof-lemZ}

	The dynamics of the deseasonalized residual demand \(Z\) over the interval \(s \in [t_n, t_{n+1})\) is given by 
\begin{align}
   \diff Z(s) = -\beta_R Z(s) ds + \sigma_R dW(s), \quad Z(t_n)=Z_n. 
\end{align} 
 Multiplying the SDE by an integrating factor \(\mathrm{e}^{\beta_Rs}\) yields: $\mathrm{e}^{\beta_Rs}dZ(s) + \beta_R \mathrm{e}^{\beta_Rs} Z(s) ds = \sigma_R \mathrm{e}^{\beta_Rs} dW(s)$. Integrating both sides of the equation from \(t_n\) to \(t_{n+1}\) yields
\begin{align}
    Z(t_{n+1}) =Z_n \mathrm{e}^{-\beta_R(t_{n+1}-t_n)}+ \int_{t_n}^{t_{n+1}}\sigma_{R}\mathrm{e}^{-\beta_R(t_{n+1}-s)}dW_R(s).
    \label{Z_solution}
\end{align}
\end{proof}

\subsection{Proof of Lemma \ref{lem-Q}}
\begin{proof}\label{proof-lemQ}
 The continuous-time dynamics of the state of charge at time $s \in [t_n, t_{n+1})$ is given by Equation \eqref{dyn-Q}, which can be written as:
\begin{equation*}
dQ(s) = \mathcal{H}(Z(s),Q(s),u(s))ds,
\end{equation*}
where
\begin{equation}\label{eq_H_app}
	\mathcal{H}(t,z,q,\nu) = \begin{cases}
		-\frac{1}{C_Q}(\mu_{R}(t)+z ) \efficiency(t,z,q) - \eta_0q ,&\nu \in \{\charging,\discharging\},\\
		-\frac{1}{C_Q}R_{Q0}  \efficiency(t,z,q) - \eta_0q ,&\nu= \ecodischarge,\\
		-\eta_0(q)q, & \text{otherwise}.
	\end{cases}
\end{equation}
Using Assumptions~\ref{Ass-co} and \ref{Ass-par} on the control and parameters, that is, $\eta_E(s) \approx \eta_E^n = \eta_E(t_n)$, $\mu_R(s) \approx \muRn = \mu_R(t_n)$, and $a=\alpha_n=u(t_n)$ for $s\in [t_n, t_{n+1})$, for $Q(t_n)=q$, we obtain an approximate dynamics of the state of charge considering each control mode $a$:

\begin{enumerate}
 \item[(a)]For $a \in \{u^C,u^D\}$ (charging/ or discharging the battery in full mode), the dynamics of the state of charge is given by
\begin{align}
dQ(s) = -\frac{\eta_E(s)}{C_Q}(\mu_{R}(s)+Z(s))\diff s - \eta_0Q(s) \diff s.
\label{chargin_a}
\end{align}
Integrating the ordinary differential equation (ODE) yields
\begin{align}
Q(t_{n+1}) &= q\mathrm{e}^{-\eta_0\Delta_N} -\frac{\eta_E^n}{C_Q} I_1(t_n,t_{n+1}),~~\text{with} ~~ I_1(t_n,t_{n+1})=\int_{t_n}^{t_{n+1}}\mathrm{e}^{-\eta_0(t_{n+1}-s)}(\mu_R(s) +Z(s))\diff s.
\label{Int_I1}
\end{align}
Substituting $Z(s)$ given by Equation \eqref{Z_solution} in the above integral and integrating, taking into account Assumption \ref{Ass-par} on the parameter yields
\begin{align*}
 I_1(t_n,t_{n+1})&=  \int_{t_n}^{t_{n+1}}\mathrm{e}^{-\eta_0(t_{n+1}-s)}\bigg( \mu_R(s)+ z \mathrm{e}^{-\beta_R(s-t_n)}+ \int_{t_n}^{s}\sigma_{R}\mathrm{e}^{-\beta_R(s-u)}dW_R(u)\bigg)ds\\
&=\frac{\muRn}{\eta_0}(1 - \mathrm{e}^{-\eta_0\Delta_N})+\frac{z}{\eta_0-\beta_R}\left(\mathrm{e}^{-\beta_R\Delta_N}-\mathrm{e}^{-\eta_0\Delta_N}\right) +\Upsilon_Q, 
\end{align*}
where
\begin{align}
\Upsilon_Q = \sigma_R \int_{t_n}^{t_{n+1}} \mathrm{e}^{-\eta_0(t_{n+1} - s)} \left( \int_{t_n}^s \mathrm{e}^{-\beta_R(s - u)} dW(u) \right) ds.
\label{A-B-C}
\end{align} 
\item[(b)]For  $a= u^{DL}$ (discharging the battery in limited mode), the dynamics of the state of charge is given by
\[
\diff Q(s) = -\bigg(\frac{\eta_E(s)}{C_Q}R_{{\revi Q0}} +\eta_0Q(s)\bigg)\diff s, ~ \text{ which implies } ~ Q_{n+1}= q\mathrm{e}^{-\eta_0\Delta_N}-\frac{\eta_E^n R_{Q0}}{ \eta_0 C_Q} \left(1-\mathrm{e}^{-\eta_0\Delta_N}\right).
\]
\item[(c)] For $a \in \{u^W,u^O,u^F,u^{FL}\}$ (other control modes), the dynamics reduces to
\[
\diff Q(s) = - \eta_0Q(t)\diff s, \quad \text{ which implies } \quad Q_{n+1} =q\mathrm{e}^{-\eta_0\Delta_N}.
\]
\end{enumerate}
\end{proof}

\subsection{Proof of Lemma \ref{lem-G}}
\begin{proof}\label{proof-lemg}
We recall that the continuous-time dynamics of the fuel level is given for  $ s\in [t_n,t_{n+1}).$ by
\begin{equation}
\diff G(s) =\mathcal{J}(s,\desresdemand(s),\Concontrol(s)), \qquad G(0)=G_0, 
\end{equation}
where $\mathcal{J}$ is defined as
\begin{equation}
	\mathcal{J}(t,z,\nu) = \begin{cases}
		c_0 + c_1(\mu_{R}(t)+z) & \nu=\fueluse,\\
		c_0 + c_1R_{G0} & \nu=\ecofueluse,\\
		0 &\text{otherwise}.
	\end{cases}
\end{equation}
The approximate solution of the dynamics of \(G\) with initial condition $G(t_n) = g$ at time \(t_{n+1}\) is given in terms of control as follows:
\begin{enumerate}
    \item[(a)] For $a = u^{F}$ (generator in full mode), the dynamics of the fuel tank level is given by
\begin{align}
    \diff G(s) =\bigg(-\frac{c_0}{C_G} - \frac{c_1}{C_G}(\mu_{R}(s)+Z(s))\bigg) \diff s.
    \label{G_full}
\end{align}
    Substituting $Z$ given by \eqref{Z_solution} and  integrating from $t_n$ to $t_{n+1}$, we obtain
\begin{align}
    G(t_{n+1})  &= G(t_n)-\frac{c_0}{C_G}\Delta_N - \tilde{c_1}\int_{t_n}^{t_{n+1}} \bigg( \mu_R(s)+ Z_n \mathrm{e}^{-\beta_R(s-t_n)}+ \int_{t_n}^{s}\sigma_{R}\mathrm{e}^{-\beta_R(s-u)}dW_R(u)\bigg)ds.\\
    &=g-\frac{c_0}{C_G}\Delta_N  - \frac{c_1}{C_G} (\muRn \Delta_N+\frac{z}{\beta_R}(1-\mathrm{e}^{-\beta_R\Delta_N})+\Upsilon_G),
\end{align}  
    where
   \begin{align}       
  \Upsilon_G&=\sigma_R\int_{t_n}^{t_{n+1}}\int_{t_n}^{s}\mathrm{e}^{-\beta_R(s-u)}dW(u)ds.
   \label{Up-1}
   \end{align}
\item[(b)] For $a = u^{FL}$ (generator in limited mode), the  dynamics is given by
    \begin{align}
    dG(s) =\bigg(-\frac{c_0}{C_G} - \frac{c_1}{C_G}R_{G0}\bigg)\diff s,\label{g} \quad \text{ which implies } \quad G(n+1)=g-\frac{1}{C_G}(c_0+c_1 R_{G0})\Delta_N. 
  \end{align}
    \item[(c)] For $a \in \{u^W, u^O, u^D, u^{DL}\}$), 
 $\mathcal{J}(s, z, a) = 0$, and the result follows.
\end{enumerate}
\end{proof}

\section{Details on Conditional Distributions}
\subsection{Distribution of Residual Demand}
 \begin{proof}\label{proof-propZ}
From Equation \eqref{Rec-R}, the value of $Z(t_{n+1})$ is given by
\begin{equation}\label{a}
Z_{n+1} = Z_n \mathrm{e}^{-\beta_R\Delta_N} + \sigma_{R}\int_{t_n}^{t_{n+1}} \mathrm{e}^{-\beta_R(t_{n+1}-s)} \, dW(s).
\end{equation}
Taking the conditional expectation in \eqref{a} and using the fact that $\displaystyle\int_{t_n}^{t_{n+1}} \mathrm{e}^{-\beta_R(t_{n+1}-s)} \, dW(s)$ is a martingale, we have
\begin{align*}
\meanZ
&= \mathbb{E}\left[ Z_n \mathrm{e}^{-\beta_R\Delta_N} + \sigma_{R}\int_{t_n}^{t_{n+1}} \mathrm{e}^{-\beta_R(t_{n+1}-s)} \, dW(s) \, \bigg| \, Z_n=z \right] 
= z \mathrm{e}^{-\beta_R\Delta_N}.
\end{align*}
To compute the conditional variance, we first determine the conditional covariance of $Z(t)$ and $Z(s)$ given $Z_n=z$. Under Assumption~\ref{Ass-par}, applying the It\^o isometry property, the closed-form expression \eqref{a} yields
\begin{align}
\text{Cov}(Z(t), Z(s)) &= \mathbb{E}\left[\left(Z(t)-\mathbb{E}[Z(t)|Z_n=z]\right)\left(Z(s)-\mathbb{E}[Z(s)|Z_n=z]\right) \, \bigg| \, Z_n=z\right] \\
&= \mathbb{E}\left[\left(\int_{t_n}^{t}\sigma_R \mathrm{e}^{-\beta_R(t-u)} \, dW(u)\right)\left(\int_{t_n}^{s}\sigma_R \mathrm{e}^{-\beta_R(s-v)} \, dW(v)\right) \, \bigg| \, Z_n=z\right] \\
&=\sigma_{R}^2 \mathrm{e}^{-\beta_R(t+s)} \mathbb{E}\left[\left(\int_{t_n}^{\min(t,s)}\mathrm{e}^{\beta_Ru} \, dW(u)\right)^2 \, \bigg| \, Z_n=z\right] \\
&= \sigma_{R}^2 \mathrm{e}^{-\beta_R(t+s)}\mathbb{E}\left[\int_{t_n}^{\min(t,s)}\mathrm{e}^{2\beta_Ru} \, du\right]
=\frac{\sigma_{R}^2}{2\beta_R} \mathrm{e}^{-\beta_R(t+s)} \left(\mathrm{e}^{2\beta_R\min(t,s)}-\mathrm{e}^{2\beta_Rt_n}\right).\\
\label{Cov_z}
\end{align}
For $s = t = t_{n+1}$, the result follows.
\end{proof}

\subsection{Distribution of Battery State of Charge}
\begin{proof}
\label{proof-propQ}
 We aim to compute the distribution of \(Q_{n+1} = Q(t_{n+1})\)  given that the state is $(Q_n, Z_n=(q,z)$. We proceed by looking at different control actions \(a\):

 \begin{enumerate}
    \item For $a \in \{u^C,u^D\}$,
    the solution of the ODE describing the stat of charge is given for $Q_n=q$ and $Z_n=z$ by
    \begin{align*} 
Q(t_{n+1}) &=q\mathrm{e}^{-\eta_0\Delta_N} -\frac{\eta_E^n}{C_Q}\frac{\muRn}{\eta_0}(1 - \mathrm{e}^{-\eta_0\Delta_N})+\frac{z}{\eta_0-\beta_R}\left(\mathrm{e}^{-\beta_R\Delta_N}-\mathrm{e}^{-\eta_0\Delta_N}\right) +\Upsilon_q,  
\end{align*}
where $\Upsilon_q$ is  given above by \eqref{A-B-C}.\\
 Since  $\Upsilon_q$ is a stochastic integral with respect to a Wiener process, $\mathbb{E}[\Upsilon_q] = 0$, then,
    \begin{align*}
    \meanQ&= q\mathrm{e}^{-\eta_0\Delta_N} -\frac{\eta_E^n}{C_Q}\frac{\muRn}{\eta_0}(1 - \mathrm{e}^{-\eta_0\Delta_N})+\frac{z}{\eta_0-\beta_R}\left(\mathrm{e}^{-\beta_R\Delta_N}-\mathrm{e}^{-\eta_0\Delta_N}\right).
    \end{align*}
  Using the covariance of $Z(t)$ given by \eqref{Cov_z}, the conditional covariance $\text{Cov}(Q(t),Q(s))$ is given by
    \begin{align}
    \text{Cov}(Q(t),Q(s)& = \frac{(\eta_E^n)^2}{C_Q^2}\int_{t_n}^{t} \int_{t_n}^{s}\mathrm{e}^{-\eta_0 (t-u)}\mathrm{e}^{-\eta_0 (s-v)}\text{Cov} (Z(u), Z(v))\diff v\diff u \\
    &= \frac{(\eta_E^n)^2\sigma_{R}^2 }{2\beta_RC_Q^2} \left(\mathrm{e}^{-\eta_0(s+t)}J_1(s,t)-\mathrm{e}^{-\eta_0(s+t)+2\beta_Rt_n}J_2(s,t) \right),
    \label{Cov-QQ}
    \end{align} 
    where 
    $J_1(t,s) =\displaystyle\int_{t_n}^{t}\int_{t_n}^{s} \mathrm{e}^{\eta_0(u+v)-\beta_R|u-v|}dvdu. $ and $J_2(t,s) =\displaystyle\int_{t_n}^{t}\int_{t_n}^{s} \mathrm{e}^{(\eta_0 - \beta_R)(u+v)}\diff v\diff u$.\\
   Now, we investigate the double integrals $J_1$ and $J_2$. We start by $J_2$, we have
    \begin{align*}
     J_2(s,t) &= \int_{t_n}^{t}\int_{t_n}^{s} \mathrm{e}^{(\eta_0 - \beta_R)(u+v)}\diff v\diff u \\  &= \frac{1}{(\eta_0 - \beta_R)^2} \left( \mathrm{e}^{(\eta_0 - \beta_R)(t+s)} - \mathrm{e}^{(\eta_0 - \beta_R)(t+t_n)} - \mathrm{e}^{(\eta_0 - \beta_R)(s+t_n)} + \mathrm{e}^{2(\eta_0 - \beta_R)t_n} \right).
    \end{align*}
    Note that the integrand in $J_1$ contains an absolute value , which must be investigated. Without loss of generality, we assume that $s\leq t$. 
	Then, the domain of integration can be subdivided into 3 subsets as follows: $\Delta=\{(u,v) \in [t_n,s]\times [t_n,t], t_n\leq s\leq t \leq t_{n+1}\}=\Delta_1 \cup \Delta_2 \cup \Delta_3$, where $\Delta_1=\{(u,v) \in \Delta, t_n\leq u\leq s, t_n\leq v\leq u\}$, $\Delta_2=\{(u,v) \in \Delta, t_n\leq u\leq s, u\leq v\leq s\}$, and $\Delta_3=\{(u,v) \in \Delta, t_n\leq u\leq s, s\leq v\leq t\}$.
    Then,  the double integral $J_1$ can be written as follows:
  \begin{align*}
			J_1(s,t) &=\int_{t_n}^{t}\int_{t_n}^{s} \mathrm{e}^{\eta_0(u+v)-\beta_R|u-v|}dvdu\\
			&= \int_{t_n}^{s}\int_{t_n}^{u} \mathrm{e}^{\eta_0(u+v)-\beta_R|u-v|}dvdu + \int_{t_n}^{s}\int_{u}^{s} \mathrm{e}^{\eta_0(u+v)-\beta_R|u-v|}dvdu + \int_{s}^{t}\int_{t_n}^{s} \mathrm{e}^{\eta_0(u+v)-\beta_R|u-v|}dvdu\\
            &= \frac{1}{\eta_0^2-\beta_R^2} \left[ \frac{-\beta_R}{\eta_0} \mathrm{e}^{2s\eta_0} + \frac{\eta_0+\beta_R}{\eta_0} \mathrm{e}^{2t_n\eta_0}  - \mathrm{e}^{t_n(\eta_0+\beta_R)+s(\eta_0-\beta_R)} \right.+ \mathrm{e}^{t(\eta_0-\beta_R)+s(\eta_0+\beta_R)}\\ & ~~~~~~~~~~~~~~~~~~~~~~~~~~~~~~~~~~~~~~~~~~~~~~~~~~~~~~~~~~~~~~~~~~~~~~~~~~~~~~~~~~~~~~~~~~~~ \quad \left. - \mathrm{e}^{t(\eta_0-\beta_R)+t_n(\eta_0+\beta_R)} \right].
		\end{align*}

      Substituting $J_1$ and $J_2$ in the above conditional covariance and setting  \(s=t=t_{n+1}\), we obtain the conditional variance given by Equation \eqref{cond_covQ} of Proposition \ref{prop-Q}.
         \item For $a =u^{DL}$, the ODE for the state of charge becomes deterministic. Therefore, and the conditional  variance \(\ConVarQ = 0\) and the mean is exactly the solution of the ODE given by
      \begin{align*}
      	\meanQ= q\mathrm{e}^{-\eta_0\Delta_N}-\frac{\eta_E^n R_{Q0}}{ \eta_0 C_Q} \left(1-\mathrm{e}^{-\eta_0\Delta_N}\right),
      \end{align*}
       \item For $a \in \{u^W,u^O,u^F,u^{FL}\}$, the dynamics is given by $\diff Q= \eta_0 Q \diff s$. Therefore, a straightforward calculation gives the conditional mean and variance as \( \meanQ = q \mathrm{e}^{\eta_0\Delta_N}\)  and \(  \ConVarQ = 0\), respectively.
  \end{enumerate}
\end{proof}

\subsection{Distribution of Fuel Tank Level}
\begin{proof}
\label{proof-propG}
 We aim to compute the conditional mean and the conditional variance of $G_{n+1} = G(t_{n+1})$  given the state $(G_n, Z_n)=(g,z)$.
To simplify the notation, we let $\tilde{c}_i = c_i/\capacityfuel,~i=0,1$. As in the case of battery state of charge, we proceed by looking at different control actions \(a\).

\begin{enumerate}
    \item For $ a = u^{F}$, the solution of the fuel tank level \(G(t)\) is given for \(G_n = g\) and \(Z_n = z\)  by
    \[ G(t_{n+1}) = g - (\tilde{c_0} + \tilde{c}_1 \muRn) \Delta_N - \tilde{c}_1 \int_{t_n}^{t_{n+1}} Z(s) \diff s,\]
     where $Z$ is the deseasonalized residual demand given by \eqref{Z_solution}. Knowing that $\Upsilon_G$ is a stochastic integral with respect to the Wiener process, one has $\mathbb{E}[\Upsilon_G]=0$. Then, the conditional mean is given by $$ \meanQ = g-\tilde{c}_0\Delta_N  - \tilde{c}_1(\muRn \Delta_N+\frac{z}{\beta_R}(1-\mathrm{e}^{-\beta_R\Delta_N}).$$  
 The conditional covariance of $G$ is computed as follows:
    \begin{align}
    \Cov(G(t),G(s)) &= \Cov\left(- \tilde{c}_1 \int_{t_n}^{t} Z(u)\diff u, - \tilde{c}_1 \int_{t_n}^{s} Z(v)\diff v\right) 
    = \frac{\tilde{c}_1^2 \sigma_{R}^2}{2\beta_R}\left[J_{1}(t,s)-J_{2}(t,s)\right],\\ \label{Con-GG}
    \end{align}
    where $ ~J_{1}(t,s)=\displaystyle\int_{t_n}^{t}\int_{t_n}^{s}\mathrm{e}^{-\beta_R|u-v|} \, \diff v \, \diff u
     =  \frac{2}{\beta_R^2} (\beta_R(t-t_n) + \mathrm{e}^{-\beta_R(t-t_n)} - 1)$\\
and $J~_{2}(t,s)=\displaystyle\int_{t_n}^{t}\int_{t_n}^{s}\mathrm{e}^{-\beta_R(u+v-2t_n)} \, \diff v\diff u = \frac{1}{\beta_R^2} \left(1 - 2\mathrm{e}^{-\beta_R(t-t_n)} + \mathrm{e}^{-2\beta_R(t-t_n)}\right).$ \\
    Substituting $J_1$ and $J_2$ in Equation \eqref{Con-GG} and setting $t=s=t_{n+1}$, we obtain the conditional variance  given by
    \begin{align}
    \ConVarQ&= \left(\frac{c_1}{C_G}\right)^2 \frac{\sigma_{R}^2}{2\beta_R^3} \left[ 2\beta_R\Delta_N - 3 + 4\mathrm{e}^{-\beta_R\Delta_N} - \mathrm{e}^{-2\beta_R\Delta_N} \right]. \label{eq:var_G_final_a}
    \end{align}

    \item For $a = u^{FL}$, the dynamics of the fuel tank level is deterministic with the solution given by  $G(t_{n+1}) = g - (\tilde{c_0} + \tilde{c_1}R_{G0})\Delta_N$. Therefore, the conditional mean and variance are given by \\
    \(\meanQ = g - (\tilde{c_0} + \tilde{c}_1 R_{G0})\Delta_N\) and
    \(\ConVarQ = 0\), respectively.
        \item[(c)] For $a \notin \{u^F, u^{FL}\}$, the dynamics of the fuel tank level is given by $dG(t) = 0$, which implies $G(t_{n+1}) = g$. Therefore, the conditional mean and variance are given by   $\meanG  = g$ and $\ConVarG = 0$, respectively.
\end{enumerate}

\end{proof}

\subsection{Covariance Between Residual Demand and State of Charge}
 \begin{proof}{\label{proof-CovZQ}}
 Let \(t\in[t_n,t_{n+1}]\).  We aim to compute the conditional covariance between \(Z_{n+1} \text{ and } Q_{n+1}\)  given that $(Q_n, Z_n)=(q,z)$.
 \begin{enumerate}[a)]
    \item For \(a \in \{u^C, u^D\}\), the conditional covariance is computed using the bilinearity of the covariance function as follows:
	\begin{align}
 \CovZQ &= \Cov\left( Z(t_{n+1}), -\frac{\eta_E(t_n)}{C_Q}\int_{t_n}^{t_{n+1}}\mathrm{e}^{-\eta_0(t_{n+1}-u')} Z(u')du' \right) \\
    &= -\frac{\eta_E(t_n)}{C_Q}\int_{t_n}^{t_{n+1}}\mathrm{e}^{-\eta_0(t-u')}\Cov(Z(t_{n+1}),Z(u'))du'\\
    &= -\frac{\eta_E(t_n)}{C_Q}\frac{\sigma_{R}^2}{2\beta_R}
    \int_{t_n}^{t_{n+1}}\mathrm{e}^{-\eta_0(t_{n+1}-u')} \left(\mathrm{e}^{-\beta_R(t_{n+1}-u')}-\mathrm{e}^{-\beta_R(t_{n+1}+u'-2t_n)}\right) du'\\
    &= \frac{1 - \mathrm{e}^{-(\eta_0+\beta_R)\Delta_N}}{\eta_0+\beta_R}  - \frac{\mathrm{e}^{-(\eta_0+\beta_R)t_{n+1} + 2\beta_R t_n}}{\eta_0-\beta_R} (\mathrm{e}^{(\eta_0-\beta_R)t_{n+1}} - \mathrm{e}^{(\eta_0-\beta_R)t_n}).
    \label{G_q}
\end{align}
    \item For \(a \notin \{u^C, u^D\}\), the dynamics of $Q(t)$ is a deterministic ODE. Therefore, \(Z_{n+1}\) and \(Q_{n+1}\) are independent, which implies that
\( \CovZQ = 0\)
  \end{enumerate}
\end{proof}
\subsection{Covariance Between Residual Demand and Fuel Tank Level}
\begin{proof}{\label{proof-CovZG}}

We aim to compute the conditional covariance between \(Z_{n+1} \text{ and } G_{n+1}\)  given the state $(G_n, Z_n)=(g,z)$. As for the state of charge, the covariance depends on controls as follows:
\begin{enumerate}[a)]
    \item For $a = u^{F}$ (generator in full mode),
the conditional covariance is given by
\begin{align*}
    \CovZG &= \Cov\left( Z(t), - \frac{c_1}{C_G}\int_{t_n}^{t_{n+1}} Z(u)du \right) \\&
    = - \frac{c_1}{C_G}\frac{\sigma_{R}^2}{2\beta_R} \int_{t_n}^{t_{n+1}} \left(\mathrm{e}^{-\beta_R(t_{n+1}-u)}-\mathrm{e}^{-\beta_R(t_{n+1}+u-2t_n)}\right) du= - \frac{\tilde{c_1} \sigma_R^2}{2  \beta_R^2} (1-\mathrm{e}^{-\beta_R\Delta_N})^2.
\end{align*}
\item  For $a\neq u^{F}$, the dynamics of $G(t)$ is a deterministic ODE. Therefore, \(Z_{n+1}\) and \(G_{n+1}\) are independent, which implies that
\( \CovZQ = 0\) 
\end{enumerate}
\end{proof}

\section{Details on Performance Criterion}
\label{appen_Proof_condiC}
\begin{proof}
 We recall that the running cost \(\psi\) is defined as:
\begin{equation}
    \psi(t,x,a)=\begin{cases}
    F_0(c_0 +c_1(\mu_R(t)+Z(t)) & a=u^{F},\\
    F_0(c_0 +c_1 R_{G0})+k_0(\mu_R(t)+Z(t)-R_{G0})^2 & a=u^{FL},\\
    \gamma_{deg}(\mu_R(t)+Z(t))& a = u^D,\\   
     \gamma_{deg} R_{Q0}+k_0(\mu_R(t)+Z(t)-R_{Q0})^2& a= u^{DL},\\
    -\gamma_{deg}(\mu_R(t)+Z(t)) & a = u^C,\\    k_0(\mu_R(t)+Z(t))^2&a=u^W,\\
    0&a=u^O.
    \end{cases}
\end{equation}
Let \(\Delta_N = t_{k+1} - t_k\) be the time step, \(Z(t)\) be given in closed form by \eqref{Z_solution} and considering Assumption \ref{Ass-par} on the constant parameters, we obtain:
\begin{enumerate}
    \item \text{For the control action \(a = u^F\),}  the conditional expectation of the discounted running cost is given by
 \begin{align*}
        \Psi_{F}(k,x) &= \mathbb{E}\left[\int_{t_k}^{t_{k+1}} F_0 \mathrm{e}^{-\rho(s-t_k)}\left(c_0 +c_1(\mu_R(s)+Z(s))\right)ds \mid X(t_k) = x\right] \\
        &= F_0\mathbb{E}\left[\int_{t_k}^{t_{k+1}} \left(c_0\mathrm{e}^{-\rho(s-t_k)}+\mathrm{e}^{-\rho(s-t_k)}c_1\mu_R(s)+\mathrm{e}^{-\rho(s-t_k)}c_1Z(s)\right)ds \mid X(t_k) = x\right]\\
      & = F_0\left(\frac{(c_0+c_1\mu_{R,k}}{\rho}\left(1-\mathrm{e}^{-\rho\Delta_N}\right) + \frac{zc_1}{(\rho+\beta_R)}\left(1-\mathrm{e}^{-(\rho+\beta_R)\Delta_N}\right)\right).
    \end{align*}

\item \text{ For the control action \(a= u^{FL}\),}  
the conditional expectation can be  evaluated analogously to the case of control $a=u^F$,  we obtain
\begin{align*}
    &\Psi_{FL}(k,x)
    \approx \int_{t_k}^{t_{k+1}}\mathrm{e}^{-\rho(s-t_k)} \left[ F_0 (c_0 + c_1 R_{G0}) + k_0\mathbb{E}[(R(s)-R_{G0})^2] \right] \diff s \\
  &= \int_{t_k}^{t_{k+1}} \mathrm{e}^{-\rho(s-t_k)} F_0 (c_0 + c_1 R_{G0}) \diff s + k_0 \int_{t_k}^{t_{k+1}} \mathrm{e}^{-\rho(s-t_k)} \mathbb{E}[(R(s)-R_{G0})^2] \diff s = J_1(k)+k_0J_2(k) 
\end{align*}
with 
\begin{align*}
    J_1(k) = \int_{t_k}^{t_{k+1}} \mathrm{e}^{-\rho(s-t_k)} F_0 (c_0 + c_1 R_{G0}) \diff s = \frac{F_0 (c_0 + c_1 R_{G0})}{\rho} \left(1 - \mathrm{e}^{-\rho\Delta_N}\right)
\end{align*}
and 
\begin{align*}
    &J_2(k) = \int_{t_k}^{t_{k+1}} \mathrm{e}^{-\rho(s-t_k)} \mathbb{E}\left[(\mu_R(t_k) + Z(s) - R_{G0})^2\right]\diff s  \\
&= \int_{t_k}^{t_{k+1}} \mathrm{e}^{-\rho(s-t_k)} \left((\mu_R(s) - R_{G0})^2 + 2(\mu_R(s) - R_{G0})z \mathrm{e}^{-\beta_R (s-t_k)} + \frac{\sigma_R^2}{2\beta_R} (1 - \mathrm{e}^{-2\beta_R(s-t_k)})\right.\\&\left.\qquad + z^2 \mathrm{e}^{-2\beta_R (s-t_k)}\right)\diff s\\
&=\zeta_1\left( (\mu_{R,k} - R_{G0})^2 + \frac{\sigma_R^2}{2\beta_R} \right) + 2z(\mu_{R,k} - R_{G0}) \zeta_2 + \left(z^2 - \frac{\sigma_R^2}{2\beta_R}\right) \zeta_3,
\end{align*}
where \begin{equation}\label{zeta_app}
    \zeta_1 = \frac{1 - \mathrm{e}^{-\rho \Delta_N}}{\rho},~~\zeta_2 =\frac{1 - \mathrm{e}^{-(\rho + \beta_R)\Delta_N}}{\rho + \beta_R},  ~~ \text{ and }  ~~\zeta_3  =\frac{1 - \mathrm{e}^{-(\rho + 2\beta_R)\Delta_N}}{\rho + 2\beta_R}. 
\end{equation}
Substituting $J_1$ and $J_2$ in the expression of $\Psi_{FL}$ above yields the result.
Therefore, \begin{align*}
\Psi_{FL}(k,x)
    & \approx \zeta_1 \left(F_0 (c_0 + c_1 R_{G0}) + k_0(\mu_{R,k} - R_{G0})^2 + \frac{\sigma_R^2k_0}{2\beta_R} \right) + 2z k_0(\mu_{R,k} - R_{G0})\zeta_2\\&+ k_0\left(z^2 - \frac{\sigma_R^2}{2\beta_R}\right) \zeta_3.
\end{align*}

    \item \text{ For the control action \(a=u^D\)}, the conditional expectation is given by
    \begin{align}
        \Psi_{D}(k,x)&= \mathbb{E}\left[\int_{t_k}^{t_{k+1}}\gamma_{deg}\mathrm{e}^{-\rho(s-t_k)}(\mu_R(s)+Z(s))\diff s \mid X(t_k)=x\right]\\
        &=\gamma_{deg} \left( \int_{t_k}^{t_{k+1}}\mathrm{e}^{-\rho(s-t_k)}\mu_R(s)ds+\int_{t_k}^{t_{k+1}}\mathrm{e}^{-\rho(s-t_k)}\mathbb{E}\left[Z(s) \mid X(t_k)=x\right]ds\right)\\
        &=\gamma_{deg} \left( \frac{\muRn}{\rho}\left(1-\mathrm{e}^{-\rho\Delta_N}\right) +  \frac{z}{\rho+\beta_R}\left(1-\mathrm{e}^{-(\rho+\beta_R)\Delta_N}\right)\right).
    \end{align}

\item \text{ For the control action \(a=u^{DL}\)}, we have:
\begin{align}
  \Psi_{DL}(k,x) &= \gamma_{deg}\int_{t_{\revi k}}^{t_{{\revi k}+1}} \mathrm{e}^{-\rho(s-t_{\revi k})}R_{Q0}ds+ k_0\int_{t_{\revi k}}^{t_{{\revi k}+1}} \mathrm{e}^{-\rho(s-t_{\revi k})}\mathbb{E}\left[((\mu_R(s)+Z(s)-R_{Q0}))^2\right]ds.
\end{align}
Similarly to the case of control \(a=u^{FL}\), we have
\begin{align*}
     \Psi_{DL}(k,x) &\approx \int_{t_k}^{t_{k+1}} \mathrm{e}^{-\rho(s-t_k)}\left[ \gamma_{\text{deg}}R_{Q0} + k_0\mathbb{E}[(R(s)-{\revi R_{Q0}})^2{\revi ]} \right] ds \\
    &= \zeta_1\left(\gamma_{\text{deg}}R_{Q0}  + k_0(\mu_{R,k} - {\revi R_{Q0}})^2 + \frac{\sigma_R^2k_0}{2\beta_R} \right) + 2z k_0(\mu_{R,k} - {\revi R_{Q0}}) \zeta_2+k_0\left(z^2 - \frac{\sigma_R^2}{2\beta_R}\right) \zeta_3,
\end{align*}
where $\zeta_1, \zeta_2,$ and $\zeta_3$ are given by \eqref{zeta}.

\item \text{For the control action \(a=u^C\)}, we have:
\begin{align}
   \Psi_C(k,x) &= -\gamma_{deg}\int_{t_k}^{t_{k+1}} \mathrm{e}^{-\rho(s-t_k)}(\mu_R(s)+\mathbb{E}(Z(s)~|~ Z_k=z))ds \\
    &= -\gamma_{deg}\left(\frac{\mu_{R,k}}{\rho}\left(1-\mathrm{e}^{-\rho\Delta_N}\right)+ \frac{z}{\rho+\beta_R}\left(1-\mathrm{e}^{-(\rho+\beta)\Delta_N}\right)\right).
\end{align}
\item \text{For the control action \(a = u^W\),} the conditional expectation is computed as follows:
\begin{align*}
    \Psi_W(k,x) &= k_0\int_{t_{\revi k}}^{t_{{\revi k}+1}} \mathrm{e}^{-\rho(s-t_{\revi k})}\mathbb{E}\left[(\mu_R(s)+Z(s))^2\right]ds\\
    &= k_0\bigg[\zeta_1\left(\mu^2_{R,k}+ \frac{\sigma_R^2}{2\beta_R} \right) + 2z \zeta_2\mu_{R,k}+\left(z^2 - \frac{\sigma_R^2}{2\beta_R}\right) \zeta_3\bigg],
\end{align*}
where $\zeta_1, \zeta_2,$ and $\zeta_3$ are given by \eqref{zeta}.
\item \text{For the control action \(\nu = u^O\),}  the running cost is zero and the result follows:
\hfill $\square$
\end{enumerate}
\end{proof}

\end{appendix}
\bigskip
\begin{footnotesize}	 
	\noindent\textbf{Acknowledgments~}
	The authors thank  
 Ralf {\revi Wunderlich} (BTU Cottbus-Senftenberg), Maalvladedon Ganet Some (University of Rwanda),  and Florent Onana Assouga (Hamburg University of Technology), 
	for valuable discussions that improved this paper.	\\
                
	
    \smallskip\noindent    
	\textbf{Funding~} { Nathalie Fruiba  gratefully acknowledges the mobility support by ERASMUS+, award number KA171 (2023).
	   Paul Honore Takam gratefully acknowledge the support by the Deutsche Forschungsgemeinschaft (DFG), award number 541347508.	
    }\\
	\smallskip\noindent
	\textbf{Code availability~} The MATLAB R2025b source code that reproduces the numerical results of this paper (the
	value function and the optimal decision rule, the weather scenarios, the comparison with rule-based strategies, and the
	sensitivity analyses with respect to the discount rate and to the discretization) is provided as supplementary material,
	with a description of its use. It requires base MATLAB only.
\end{footnotesize}
\begin{footnotesize}
\let\oldbibliography\thebibliography
\renewcommand{\thebibliography}[1]{%
	\oldbibliography{#1}%
    \setlength{\itemsep}{-1.0ex plus .05ex}%
}
\bibliography{template-bibliography}
\end{footnotesize}
\end{document}